\documentclass[a4paper,10pt]{article}
\usepackage[utf8]{inputenc}

\usepackage[T1]{fontenc}
\usepackage{amsfonts}
\usepackage{amsmath}
\usepackage{amssymb}
\usepackage{amsthm}
\usepackage{hyperref}
\usepackage{color}
\renewcommand{\H}{{\textbf{[H]}}}
\newcommand{\HT}{{\textbf{[HT]}}}

\renewcommand{\P}{{\mathbb P}}

\newcommand{\C}{{\mathbb C}}
\newcommand{\R}{{\mathbb R}}

\newcommand{\M}{{\mathbb M}}

\newcommand{\N}{{\mathbb N}}

\newcommand{\E}{{\mathbb E}}
\newcommand{\proba}{{\mathbb P}}

\newcommand{\T}{{\mathbb T}}

\newcommand{\ind}{{\textbf{1}}}
\newcommand{\I}{{\textbf{1}}}

\newcommand{\leftB}{{{[\![}}}
\newcommand{\rightB}{{{]\!]}}}

\theoremstyle{plain}
\newtheorem{theorem}{Theorem}[section]
\newtheorem{lemma}[theorem]{Lemma}
\newtheorem{proposition}[theorem]{Proposition}
\newtheorem{corollary}[theorem]{Corollary}

\theoremstyle{definition}
\newtheorem{remark}{Remark}[section]

\newcommand{\mysection}{\setcounter{equation}{0} \section}

\title{A Parametrix Approach for some Degenerate Stable Driven SDEs}
\author{Lorick Huang \\ Universit\'e Paris 7 - Diderot
\footnote{Laboratoire de Probabilit\'es et Mod\`eles Al\'eatoires,  B\^atiment Sophie Germain, 5 rue Thomas Mann, 75205 Paris CEDEX 13, France.  huang@math.univ-paris-diderot.fr}
        \and St\'ephane Menozzi \\ Universit\'e d'Evry Val-d'Essonne \footnote{Laboratoire de Mod\'elisation Math\'ematique d'Evry, UMR CNRS 8071, 23 Boulevard de France, 91037 Evry, France and Laboratory of Stochastic Analysis, HSE, Moscow. stephane.menozzi@univ-evry.fr} 
        }

\begin{document}

\newpage

\maketitle

\begin{abstract}

We consider a stable driven degenerate stochastic differential equation, whose coefficients satisfy a kind of weak H\"ormander condition. Under mild smoothness assumptions we prove the uniqueness of the martingale problem for the associated generator under some dimension constraints. Also, when the driving noise is scalar and tempered, we establish density bounds reflecting the multi-scale behavior of the process.
\end{abstract}


\mysection{Introduction.}

The aim of this paper is to study degenerate stable driven stochastic differential equations of the following form:
\begin{eqnarray}\label{EDS 1}
dX_t^1 &=& \left(a_t^{1,1}X_t^1 + \cdots +a_t^{1,n}X_t^n\right) dt+ \sigma(t,X_{t^-})dZ_t\\
dX_t^2 &=& \left(a_t^{2,1}X_t^1 + \cdots +a_t^{2,n} X_t^n \right) dt\nonumber \\
dX_t^3 &=& \left(a_t^{3,2}X_t^2 + \cdots +a_t^{3,n} X_t^n \right) dt\nonumber \\
&\vdots&\nonumber\\
dX_t^n &=& \left(a^{n,n-1}_t X_t^{n-1} +a_t^{n,n}X_t^n\right)dt,\ X_0=x\in \R^{nd},\nonumber
\end{eqnarray}
where $Z$ is an $\R^{d}$ valued symmetric $\alpha$ stable process (possibly tempered and with $\alpha \in (0,2)$), $\sigma : \R^+\times\R^{nd} \rightarrow \R^d \otimes \R^d$, $a^{i,j}:\R^+\rightarrow \R^d\otimes \R^d,\ i\in \leftB 1,n\rightB,\ j\in \leftB (i-1)\vee 1,n\rightB$.  Observe that $(X_t)_{t\ge 0}=(X_t^1,\cdots,X_t^n)_{t\ge 0}$ is $\R^{nd}$ valued. 
We will often use the shortened form:
\begin{equation}\label{EDS 1 short}
dX_t = A_t X_tdt + B\sigma(t,X_{t^-})dZ_t,\ X_0=x,
\end{equation}
where $B=\left( I_{d\times d}\ 0_{(n-1)d\times d}\right)^*$ denotes the injection matrix from $\R^d$ into $\R^{nd}$, with $*$ standing for the transposition, and $A_t$ is the matrix :
$$
A_t = 
\begin{pmatrix}
a_t^{1,1} & \dots & \dots & \dots &a_t^{1,n}\\
a^{2,1}_t  & \ddots & & &a_t^{2,n}\\
0 &a_t^{3,2} & \ddots & &a_t^{3,n}\\
\vdots& \ddots&\ddots & \ddots &\vdots \\
0& \hdots & 0& a_t^{n,n-1}
 &a_t^{n,n}
\end{pmatrix}.
$$
The previous system appears in many applicative fields.  It is for instance related for $n=2$ to the pricing of Asian options in jump diffusions models (see e.g. Jeanblanc \textit{et al} \cite{jean:yor:ches:09} or Barucci \textit{et al} \cite{baru:poli:vesp:01} in the Brownian case). The Hamiltonian formulation in mechanics can lead to systems corresponding to the drift part of \eqref{EDS 1} (still with $n=2$). The associated Brownian perturbation has been thoroughly studied, see e.g. Talay \cite{tala:02} or Stuart \textit{et al.} \cite{matt:stua:high:02} for the convergence of approximation schemes to equilibrium, but to the best of our knowledge other perturbations, like the current stable one, have not yet been much considered. For a general $n$, equation \eqref{EDS 1} can be seen as the linear dynamics of $n$ coupled oscillators in dimension $d$ perturbed by a stable anisotropic noise. Observe also that in the diffusive case these oscillator chains naturally appear in statistical mechanics, see e.g Eckman \textit{et al.} \cite{eckm:99}.

Equation \eqref{EDS 1} is degenerate in the sense that the noise only acts on the first component of the system. Additionally to the non-degeneracy of the \textit{volatility}  $\sigma$, we will assume a kind of weak H\"ormander condition on the drift component in order to allow the noise propagation into the system. 

A huge literature exists on degenerate Brownian diffusions under the \textit{strong} H\"ormander condition, i.e. when the underlying space is spanned by the diffusive vector fields and their iterated Lie brackets. The major works in that framework have been obtained in a series of papers by Kusuoka ans Stroock, \cite{kusu:stro:84}, \cite{kusu:stro:85}, \cite{kusu:stro:87}, using a Malliavin calculus approach.

For the weak H\"ormander case, many questions are still open even in the Brownian setting. Let us mention in this framework the papers \cite{D&M}, \cite{menozzi2011parametrix} and \cite{kona:meno:molc:10} dealing respectively with density estimates, martingale problems and random walk approximations for systems of type \eqref{EDS 1} or that can be linearized around such systems. In those works a global multi-scale Gaussian regime holds. For highly non-linear first order vector fields, Franchi \cite{fran:12} and Cinti \textit{et al.} \cite{cint:meno:poli:12}  address issues for which there is not a single regime anymore. A specificity of the weak H\"ormander condition is the unbounded first order term which does not lead to a time-space separation
in the off-diagonal bounds for the density estimates as in the sub-Riemannian setting, see e.g. \cite{kusu:stro:87}, Ben Arous and L\'eandre \cite{bena:lean:91} and references therein. The energy of the associated deterministic control problem has to be considered instead, see e.g. \cite{D&M}. We have a similar feature in our current stable setting.  

In this work we are first interested in proving the uniqueness of the martingale problem associated with the generator $(L_t)_{t\ge } $ of \eqref{EDS 1}, i.e. for all $\varphi\in C_0^2(\R^{nd},\R)$ (twice continuously differentiable functions with compact support) 
\begin{equation}
\label{GEN}
\begin{split}
\forall x\in \R^{nd},\ L_t \varphi(x)&= \langle A_t x,\nabla\varphi(x)\rangle\\
&+\int_{\R^d} \biggl(\varphi(x+B\sigma(t,x)z)-\varphi(x)-\frac{\langle\nabla \varphi(x), B\sigma(t,x) z\rangle}{1+|z|^2} \biggr)g(|z|) \nu(dz),
\end{split}
\end{equation}
under some mild assumptions on the volatility $\sigma $,  the L\'evy measure $\nu$ of a symmetric $\alpha $ stable process and the tempering function $g $ (which is set to $1$ in the stable case).
To this end, the key tool consists in exploiting some properties of the joint densities of (possibly tempered) stable processes and their iterated integrals, corresponding to the proxy model in a \textit{parametrix} continuity technique (see e.g. Friedman \cite{frie:64} or McKean and Singer \cite{McKeanSinger}). Following the strategies developed in \cite{bass_perkins_martingale}, \cite{menozzi2011parametrix} we then derive uniqueness exploiting the smoothing properties of the parametrix kernel. For this approach to work, we anyhow consider some restrictions on the dimensions $n,d$.  
Let us indeed emphasize that the density of a $d$-dimensional $\alpha $-stable process and its $n-1$ iterated integrals behaves as the density of an $\alpha $ stable process in dimension $nd$ with a modified L\'evy measure and different time-scales. The first point can be checked through Fourier arguments (see Proposition \ref{LAB_EDFR} and Remark \ref{REM_MS}).
Also, the typical time scale of the initial stable process is $t^{1/\alpha} $ and $t^{(i-1)+1/\alpha} $ for the associated $(i-1)^{{\rm th}} $ integral.  
One of the difficulties is now that the associated L\'evy measures (on $\R^{nd} $) have spectral parts that are either not equivalent or singular with respect to the Lebesgue measure of $S^{nd-1} $. The link between the behavior of the stable density and the corresponding spectral measure is discussed intensively in Watanabe \cite{wata:07} and can lead to rather subtle phenomena.  Roughly speaking, the lower is the dimension of the support of the spectral measure, the heavier the tail. This is what leads us to consider some restrictions on the dimensions.
Also, using the resolvent\footnote{We carefully mention that we use the term \textit{resolvent} in the sense of ordinary differential equations.} associated with the ordinary differential equation obtained from \eqref{EDS 1} setting $\sigma=0$, i.e.  $\frac{d}{dt}R_t=A_t R_t,\ R_0=I_{nd\times nd}$, the mean of the process is $R_tx $ at time $t$ (transport of the initial condition by the resolvent). The process will deviate from its mean accordingly to the associated component wise time scales.

When turning to density estimates, one of the dramatic differences with the Gaussian case is the lack of integrability of the driving process. For non-degenerate stable driven SDEs, this difficulty can be bypassed to derive two-sided pointwise bounds for the SDE that are homogeneous to the density of the driving process. Kolokoltsov \cite{kolo:00} establishes in the stable case the analogue of the \textit{Aronson} bounds for diffusions, see e.g. Sheu \cite{sheu:91} or \cite{aron:67}. For approximation schemes of non-degenerate stable-driven SDEs we also mention \cite{K&M}.
In our current degenerate framework, working under somehow minimal assumptions to derive pointwise density bounds, that is H\"older continuity of the coefficients, we did not succeed to get rid of those integrability problems. We are also faced with a new difficulty due to the degeneracy and the non-local character of $L$. Namely, we have a disturbing \textit{rediagonalization phenomenon}:
when the 
density $\tilde p(t,x,\cdot) $ of $\tilde X_t=x+\int_0^t A_s \tilde X_sds +BZ_t$ is in a \textit{large deviation regime}, estimating the non local part of $L_t\tilde p $ (which is the crucial quantity to control in a parametrix approach), the very large jumps can lead to integrate $\tilde p $ on a set where it is in its typical regime. This phenomenon already appears in the non-degenerate case, but the difficulty here is that there is a dimension mismatch between the tail behavior of $\tilde p(t,x,\cdot )  $, density of $\tilde X_t $, multi-scale stable process of dimension $nd$, and the one of the jump, stable process of dimension $d$.

This quite tricky phenomenon leads us to temper the driving noise in order to obtain density estimates through a parametrix 
continuity technique.
For technical reasons that will appear later on, we establish when $d=1,n=2 $ (scalar non-degenerate diffusion and associated non-degenerate integral) the expected upper-bound up to an additional logarithmic contribution, when the coefficient $\sigma(t,x)=\sigma(t,x^2) $ depends on the \textit{fast} variable. This constraint appears in order to compensate an additional time singularity deriving from the rediagonalization. Roughly speaking, the dependence on the fast variable only gives a better smoothing effect for the parametrix kernel.
Eventually, we derive the expected diagonal lower bound, see Theorem \ref{MTHM}. To this end we use a parametrix approach similar to the one of Mc Kean and Singer \cite{McKeanSinger}. Working with smoother coefficients would have allowed to consider Malliavin calculus type techniques. In the jump case, this approach has been investigated to establish existence/smoothness of the density for SDEs by Bichteler \textit{et al.} in the non-degenerate case \cite{BGJ}, and L\'eandre in the degenerate one, see \cite{lean:85},\cite{leandre1988regularite}. Also, we mention the recent work of Zhang \cite{zhan:14} who obtained existence and smoothness results for the density of equations of type \eqref{EDS 1} in arbitrary dimension for  smooth coefficients, and a possibly non linear drift, still satisfying a weak H\"ormander condition. His approach relies on the  \textit{subordinated} Malliavin calculus, which consists in applying the \textit{usual} Malliavin calculus techniques on a Brownian motion observed along the path of an $\alpha $-stable subordinator.

Let us eventually mention some related works. Priola and Zabczyk  establish in \cite{priola2009densities} existence of the density for processes of type \eqref{EDS 1}, under the same kind of weak H\"ormander assumption and when  $ \sigma$ is constant,  for a general driving L\'evy process $Z$ provided its L\'evy measure is infinite and has itself a density on compact sets. Also, Picard, \cite{picard1996existence} investigates similar problems for singular L\'evy measures. Other results concerning the smoothness of the density of L\'evy driven SDEs have been obtained by Ishikawa and Kunita \cite{ishi:kuni:06} in the non-degenerate case but with mild conditions on the L\'evy measure and by Cass \cite{cass:09} who gets smoothness in the weak H\"ormander framework under technical restrictions. Also, we refer to the work of Watanabe \cite{wata:07} for two-sided heat-kernel estimates for stable processes with very general spectral measures. Those estimates have been extended to the tempered stable case by Sztonyk \cite{szto:10}.\\

The article is organized as follows. 
We state our main results in Section \ref{Assumptions, and Main Result}. In Section \ref{Continuity techniques}, we explain the procedure to derive those results and also state the density estimates on the process in \eqref{EDS 1} when $\sigma(t,x)=\sigma(t) $ (\textit{frozen process}). We then prove the uniqueness of the martingale problem in Section \ref{uniqueness for MP}. A non linear extension is discussed in Appendix \ref{NON_LIN_MART}. Sections \ref{TEC_SEC1} and \ref{SECTION_TECH} are the technical core of the paper. In particular, we prove there the existence of the density and the associated estimates for the frozen process and establish the smoothing properties of the parametrix kernel. Appendices \ref{DIAG_LB} and \ref{controles_phi} are dedicated to the derivation of stable density bounds and kernels combining the approaches of \cite{kolo:00} and \cite{wata:07}, \cite{szto:10} in our current degenerate setting. 
We emphasize that the tempering procedure allows to get rid of the integrability problems but does not prevent
from the rediagonalization phenomenon. This difficulty would occur even in the truncated case, thoroughly studied in the non-degenerate case by Chen \textit{et al.} \cite{chen:kim:kuma:08}. The truncation would certainly \textit{relocalize} the operator but the rediagonalization would still perturb the parametrix iteration in the stable regime.

\mysection{Assumptions, and Main Result.}
\label{Assumptions, and Main Result}


We will make the following assumptions:
\begin{trivlist}
\item[\textbf{About the Coefficients:}]
The coefficients are assumed to be bounded and measurable in time and also to satisfy the conditions below. 
\item[\textbf{[H-1]}:]  (\textbf{H\"older regularity in space}) $\exists H>0,\  \eta \in (0,1],\ \forall x,y \in \R^{nd}$ and $\forall t\ge 0$,
$$
||\sigma(t,x) - \sigma(t,y)|| \leq H |x-y|^{\eta}.
$$

\item[\textbf{[H-2]:}] (\textbf{Ellipticity}) $\exists$ $\kappa\ge 1$, $\forall \xi \in \R^{d}$, $\forall z \in \R^{nd}$ and $\forall t \ge 0$,
\begin{equation}\label{ellipticity}
\kappa^{-1} |\xi|^2 \leq \langle \xi,\sigma\sigma^*(t,z) \xi \rangle \leq \kappa |\xi|^2.
\end{equation}

\item[\textbf{[H-3]:}](\textbf{H\"ormander-like condition for $(A_t)_{t\ge 0}$})  $\exists \overline{\alpha},\ \underline{\alpha}\in \R^{+*}$, $\forall \xi\in \R^{nd} $ and $\forall t\ge 0$, $\underline{\alpha}|\xi|^2 \leq \langle a_t^{i,i-1} \xi, \xi \rangle \leq \overline{\alpha} |\xi|^2, \ \forall i\in \leftB 2, n-1\rightB $. Also, for all $ (i,j)\in \leftB 1,n\rightB^2, \ \|a_t^{i,j}\|\le \overline{\alpha}$.\\

\item[\textbf{About the Driving Noise.}]
\item[\textit{Stable Case:}]
Let us first consider $(Z_t)_{t\geq0}$ to be an $\alpha$ stable symmetric process, defined on some filtered probability space  $(\Omega,\mathcal{F},  (\mathcal{F}_t)_{t\geq 0},\proba)$, that is a L\'evy process with Fourier exponent:
$$
\E{e^{i \langle p, Z_t \rangle}} = \exp \left(-t\int_{S^{d-1}} |\langle p,  \varsigma \rangle|^\alpha \mu(d\varsigma) \right),\ \forall p\in \R^d.
$$
In the above expression, we denote by $S^{d-1}$ the unit sphere in $\R^d$, and by $\mu$ the spectral measure of $Z$.
This measure is related to the L\'evy measure of $Z$ as follows.
If $\nu$ is the L\'evy measure of $Z$, its decomposition in polar coordinates writes:
\begin{equation}
\label{POLA_MES}
\nu(dz) = \frac{d\rho}{\rho^{1+\alpha}} \tilde{\mu}(d\varsigma),\ z= \rho \varsigma,\ (\rho,\varsigma)\in \R^+\times S^{d-1}. 
\end{equation}
Then, $\mu = C_{\alpha,d} \tilde{\mu}$ (
see Sato \cite{sato} for the exact value of $C_{\alpha,d}$). In that case we suppose
\item[\textbf{[H-4]}:]  (\textbf{Non degeneracy of the spectral measure}) We assume that $\mu $ is  absolutely continuous w.r.t. 
to the Lebesgue measure of $S^{d-1} $ 
with  
Lipschitz density $h$ and that there exists $ \lambda\ge 1$ ,
s.t. for all $ u \in \R^d$, 

\begin{equation}
\label{ND}
\lambda^{-1} |u|^\alpha \leq \int_{S^{d-1}} | \langle u ,\varsigma \rangle|^\alpha \mu(d\varsigma) \leq \lambda |u|^\alpha.
\end{equation}
\item[\textit{Tempered Case:}]
In the tempered case we simply assume that $(Z_t)_{t\ge 0}$ has generator:
\begin{equation}
\label{GEN_TEMP}
L_Z \phi(x)=\int_{\R^d}\bigl\{\phi(x+z)-\phi(x)-\frac{\langle \nabla \phi(x),z\rangle}{1+|z|^2} \bigr\}  g(|z|)\nu(dz),\ \phi\in C_0^2(\R^d,\R),
\end{equation}
where the measure $\nu$
is as in the stable case and the tempering function $g:\R^{+*}\rightarrow \R^{+*}$ satisfies
\begin{trivlist}
\item[\textbf{[T]}:] \textbf{(Smoothness, Doubling property and Decay associated with the tempering function $g$)}
We first assume that $g\in C^1(\R^+,\R^{+*}) $ and that there exists $a>0$ s.t.  $g\in C^2([0,a],\R^{+*})$ if $\alpha \in [1,2) $.
We also suppose that there exists $c>0$ s.t. for all $r>0, g(r)+r\sup_{ u \in [\kappa^{-1},\kappa]} g'(u r)\le c \theta(r)$ for $\kappa $ as in \textbf{[H-2]} and where $\theta:\R^{+*}\rightarrow \R^{+*}$ is a bounded non-increasing function satisfying: 
$$ \exists D\ge 1,\ \forall r>0,\ \theta(r)\le D\theta(2r),\ (1+r)\theta(r):=\Theta(r)\underset{r\rightarrow+\infty}{\rightarrow} 0.$$
\end{trivlist} 
Typical examples of tempering functions satisfying \textbf{[T]} are for instance $r\rightarrow g(r)=\exp(-cr),c>0, g(r)=(1+r)^{-m},\ m\ge 2$.
\end{trivlist}

We say that \textbf{[HS]} (resp. \textbf{[HT]}) holds if conditions \textbf{[H-1]} to \textbf{[H-4]} are fulfilled and the driving noise $Z$ is a symmetric stable process (resp. a tempered stable process satisfying \textbf{[T]}).
We say that \textbf{[H]} is satisfied if \textbf{[HS]} or \textbf{[HT]} holds, i.e. the results under \textbf{[H]} hold for both the stable and the tempered stable driving process.

Our main results are the following. 

\begin{theorem}[\textbf{Weak Uniqueness}]
\label{MART_PB_THM}
Under  \textbf{[H]}, i.e. in both the stable and the tempered stable case, the martingale problem associated with the generator $(L_t)_{t\ge 0}$, defined in \eqref{GEN}, of the degenerate equation \eqref{EDS 1}:
$$
dX_t = A_t X_tdt + B\sigma(t,X_{t-})dZ_t,
$$
admits a unique solution provided $d(1-n)+1+\alpha >0$. 
That is, for every $x\in\R^{nd}$,
there exists a unique probability measure $\proba$ on $\Omega = \mathcal{D}(\R^+, \R^{nd})$ the space of c\`adl\`ag functions, such that for all $f \in 
 C^{1,2}_0(\R^+\times \R^{nd},\R)$, denoting by $(X_t)_{t \ge 0}$ the canonical process, we have:
$$
\proba(X_0=x ) =1\ \mbox{ and } \ f(t,X_t) - \int_0^t(\partial_u+L_u) f(u,X_u)du \   \mbox{  is a $\proba$- martingale.}
$$
Hence, weak uniqueness holds for \eqref{EDS 1}.

Now, if $d=1,n=2 $ and $\alpha>1 $, the well-posedness of the martingale problem extends to the case of a non-linear Lipschitz drift satisfying a H\"ormander-like non-degeneracy condition. Namely, weak uniqueness holds for:
\begin{eqnarray}
\label{NON_LIN_SDE}
X_t^1&=& x^1+\int_0^t F_1(s,X_s)ds+\int_0^t\sigma(s,X_{s^-})dZ_s,\nonumber\\
X_t^2&=& x^2+\int_0^t F_2(s,X_s)ds,
\end{eqnarray}
provided $F=(F_1,F_2)^*:\R^+\times \R^2\rightarrow \R^2$ is measurable and bounded in time, uniformly Lipschitz in space and such that $\partial_{x^1}F_2\in [c_0,c_0^{-1}], c_0\in (0,1] $ and $\partial_{x^1} F_2$ is $\eta $-H\"older continuous, $\eta \in (0,1] $.
\end{theorem}
\begin{remark}
The dimension constraint comes from the \textit{worst} asymptotic behavior of the stable densities in our current case.
Viewing the density of the stable process $Z$ and its iterated integrals as the density of an $nd$-dimensional multi-scale stable process yields to consider a L\'evy measure on $\R^{nd} $ for which the support of the spectral measure has dimension $(d-1)+1=d $. Thus, from Theorem 1.1 in Watanabe \cite{wata:07} we have that, at time $1$ (to get rid of the multi-scale feature), the tails will behave at least as $|x|^{-(d+1+\alpha)}$ for large values of $|x|, x\in \R^{nd}$. The condition in the previous Theorem is imposed in order to have the integrability of the worst bound in $\R^{nd}$. We refer to Section \ref{Estimates on the frozen density} for details.  In practice the condition is fulfilled for:
\begin{trivlist}
\item[-] $d=1,n=2$ for $\alpha \in (0,2)$.
\item[-] $d=1,n=3$ for $\alpha \in (1,2) $.
\item[-] $d=2,n=2$ for $\alpha \in (1,2) $.
\end{trivlist}
\end{remark}
\begin{remark}
We point out that the tempering function $g$ does not play any role in the proof of the previous theorem. Furthermore, it cannot be used to weaken the previous dimension constraints. Indeed, it can be seen from the estimates in Proposition \ref{EST_DENS_GEL_T} that the additional multiplicative term in $\theta $ makes the worst bound integrable but also yields an explosive contribution in small time.
\end{remark}


Also, when $d=1$ and $n=2$ in \eqref{EDS 1} we are able to prove the following density estimates in the tempered case.

\begin{theorem}[\textbf{Density Estimates}]
\label{MTHM}
Assume that $d=1,\ n=2$. Under \textbf{[HT]} and for $\sigma(t,x):=\sigma(t,x^2) $, i.e. the diffusion coefficient depends on the fast component,  provided $1\ge \eta>\frac{1}{(1\wedge \alpha)(1+\alpha)} $, the unique weak solution of \eqref{EDS 1} has for every $s>0$ a density with respect to the Lebesgue measure. Precisely, for all $0\le t<s $ and $ x\in \R^{2}$,  
\begin{equation}\label{VraiDens} \proba(X_s \in dy | X_t =x)=p(t,s,x,y) dy.\end{equation}
Also, for a deterministic time horizon $T>0$, and a fixed threshold $K>0$,
there exists
$ C_{\ref{MTHM}}:=C_{\ref{MTHM}}(\HT,T,K)\ge 1,\ {s.t.} \ \forall 0\le t<s\le T,\ \forall (x,y)\in (\R^{2})^2,$ 
\begin{equation}\label{borne theoreme}
p(t,s,x,y) \leq C_{\ref{MTHM}} 
\bar {p}_{\alpha,\Theta}(t,s,x,y)\left(1+\log( K \vee |(\T_{s-t}^{\alpha})^{-1}(y-R_{s,t}x)| \right),
\end{equation}
where for all $u\in \R^+,\ \T_u^\alpha:={\rm Diag}\bigl((u^{1/\alpha},u^{1+1/\alpha})  \bigr),\ \M_u:={\rm Diag}\bigl(1,u)$ and 
\begin{equation*}
\bar p_{\alpha,\Theta}(t,s,x,y)=\bar C_{\alpha,\Theta}\frac{{\rm det}(\T_{s-t}^\alpha)^{-1}}{\{K \vee |(\T_{s-t}^{\alpha})^{-1}(y-R_{s,t}x)|\}^{2+\alpha}}\Theta(|\M_{s-t}^{-1}(y-R_{s,t}x) |).
\end{equation*}
Here, $R_{s,t} $ stands for the resolvent associated with the deterministic part of \eqref{EDS 1}, i.e. $\frac{d}{ds}R_{s,t}=A_s R_{s,t},\ R_{t,t}=I_{2\times 2} $,
and $\bar C_{\alpha,\Theta}$ is s.t. $\int_{\R^{2}}\bar p_{\alpha,\Theta}(t,s,x,y)dy=1 $. 

Eventually for $0<T\le T_0 :=T_0(\HT,K)$ small enough, the following diagonal lower bound holds:
\begin{equation}
\forall 0\le t<s\le T, \ \forall (x,y)\in (\R^2)^2\ {s.t.}\ |(\T_{s-t}^{\alpha})^{-1}(y-R_{s,t}x)|\le K,\ p(t,s,x,y)\ge C_{\ref{MTHM}}^{-1}{\rm det}(\T_{s-t}^{\alpha})^{-1}.
\end{equation}
\end{theorem}

Under the current assumptions, Theorem \ref{MART_PB_THM} is proved following the lines of \cite{bass_perkins_martingale} and \cite{menozzi2011parametrix}. In the Gaussian framework, those assumptions are sufficient to derive homogeneous two-sided multi-scale Gaussian bounds, see \cite{D&M}. However, in the current context, we only managed to obtain the expected upper bound up to a logarithmic factor and a diagonal lower bound for $d=1$ and $n=2$ for a \textit{tempered} driving noise and $\sigma(t,x)=\sigma(t,x^2) $. This is mainly due to a lack of integrability of the stable process and the rediagonalization phenomenon which becomes really delicate to handle in the degenerate case. 
Precisely, the parametrix technique consists in applying the difference of two non-local generators of the form \eqref{GEN} to the density of some process
which is meant to locally behave as \eqref{EDS 1} and for which estimates are available. Such a process is known as the \textit{parametrix} or \textit{proxy}. The density of the stable $nd$-dimensional process we will use in the degenerate setting as \textit{parametrix} will have decays of order $d+1+\alpha$ in the \textit{large deviation} regime. It is indeed delicate to use other bounds than the worst one in a global approach like the parametrix. Let us mention that this is not the decay of a rotationally invariant stable process in dimension $nd $ (which would be $nd+\alpha$) except if $n=2,d=1$. 
Observe now from \eqref{GEN}, \eqref{POLA_MES} that we have a dimension mismatch between the decays of the densities of the parametrix and those of the jump measure $\nu$, which are in $d+\alpha $. 
We recall that the large jumps can lead to integrate the density on a set on which it is in its diagonal regime, when applying the non-local generator to the density. This is what we actually call \textit{rediagonalization} and  leads in our degenerate framework to additional time-singularities in the parametrix kernel. We manage to handle those singularities when $\sigma $ depends on the fast component, yielding a better smoothing property in time, see Section \ref{SECTION_TECH}. 

Observe that, in the non-degenerate context,  the decays of the rotationally invariant stable densities and the jump measure in \eqref{POLA_MES} correspond. This allows 
Kolokoltsov \cite{kolo:00} to successfully give two sided bounds for the density of the SDE which are homogeneous to those of the rotationally stable case provided the density of the spectral measure is positive. The technical reasons leading to the restriction of Theorem \ref{MTHM} will be discussed thoroughly in the dedicated sections (see Sections \ref{CTR_ITER_KER} and \ref{SECTION_TECH}). Let us mention that 
 the above results could be extended to 
the case of a $d$-dimensional non-degenerate SDE driven by a tempered stable process and the integral of \textbf{one} of its components. We emphasize as well, 
that our estimates still hold if we had a non-linear bounded drift in the dynamics of $X^1 $ if $\alpha>1 $ (see Remark \ref{drift borne}).
We conclude this paragraph saying
that the uniqueness of the martingale problem and the estimates  of Section \ref{SECTION_TECH} allow to extend in the non-degenerate case, the stable two-sided \textit{Aronson like} estimates of \cite{kolo:00} for H\"older coefficients. \\

\textbf{Constants and usual notations:}
\begin{itemize}
\item The capital letter $C$ will denote a constant whose value may change from line to line, and can depend on the hypotheses \textbf{[H]}. Other dependencies (in particular in time), will be specified, using explicit under scripts.
\item We will often use the notation $\asymp$ to express equivalence between functions. If $f$ and $g$ are two real valued nonnegative functions, we denote $f(x) \asymp g(x),\ x\in I\subset\R^p,\ p\in \N $, when there exists a constant $C\ge 1$, possibly depending on \textbf{[H]}, $I$ s.t. $C^{-1} f(x) \leq g(x) \leq C f(x),\ \forall x \in I$.
\item For $x=(x_1,\cdots,x_{nd})\in \R^{nd}$ and for all $k\in \leftB 1,n \rightB$, we define $ x^k:=(x_{(k-1)d+1},\cdots,x_{kd})\in \R^d$. Accordingly, $x=(x^1,\cdots, x^n) $. Also $x^{2:n}:=(x^2,\cdots,x^n) $. 
 \end{itemize}

From now on, we assume \textbf{[H]} to be in force, specifying when needed, which results are valid under {\HT} only.

\mysection{Continuity techniques : the Frozen equation and the parametrix series.}
\label{Continuity techniques}

\indent 

For density estimates, a continuity technique consists in considering a simpler equation as \textit{proxy} model for the initial equation.
The \textit{proxy} will be significant if it achieves two properties:
\begin{trivlist}
\item[-] It admits an explicit density or a density that is well estimated.
\item[-] The difference between the density of the initial SDE and the one of the proxy can be well controlled. 
\end{trivlist}

For the last point a usual strategy consists in expressing the difference of the densities through the difference of the generators of the two SDEs, using Kolmogorov's equations. This approach is known as the \textit{parametrix} method. In the current work, we will use the procedure developed by Mc Kean and Singer \cite{McKeanSinger}, which turns out to be well-suited to handle coefficients with mild smoothness properties.

We first introduce the \textit{proxy} model in Section \ref{SEC_FROZ}, and give some associated density bounds. We then analyze in Sections \ref{SEC_PARAM}, \ref{CTR_ITER_KER} how this choice can formally lead through a parametrix expansion to a density estimate, exploiting some suitable regularization properties in time. These arguments can be made rigorous provided that the initial SDE admits a Feller transition function.
The uniqueness of the martingale problem will actually give this property.  

\subsection{The Frozen Process.}
\label{SEC_FROZ}

\indent

In this section, we give results that hold in any dimension $d$, and for any fixed number of oscillators $n$.
Let $T>0$ (arbitrary deterministic time) and  $y\in \R^{nd}$ (final freezing point) be given.
Heuristically, $y$ is the point where we want to estimate the density of \eqref{EDS 1} at time $T$ provided it exists.
We introduce the \textit{frozen} process as follows:
\begin{equation}\label{FrozenSDE}
d\tilde{X}_s^{T,y} = A_s\tilde{X}_s^{T,y} ds+ B\sigma(s,R_{s,T}y)dZ_s.
\end{equation}
In this equation, $R_{s,T}y$ is the resolvent of the  associated deterministic equation, i.e. it satisfies $\frac{d}{ds}R_{s,T} = A_s R_{s,T}$, with $R_{T,T} = I_{nd\times nd}$ in $\R^{nd}\otimes \R^{nd}$. 
Let us emphasize that the previous choice can seem awkward at first sight. Indeed, a very natural approach
for a proxy model would consist in freezing the diffusion coefficient at the terminal point, see e.g. Kolokoltsov \cite{kolo:00}. In our current weak H\"ormander setting we need to take into account the backward transport of the final point by the deterministic differential
system. This particular choice is actually imposed by the natural metric appearing in the density of the frozen process, see Proposition \ref{EST_DENS_GEL}. This allows the comparison of the singular parts of the generators of \eqref{EDS 1} and \eqref{FrozenSDE} applied to the  frozen density, see Proposition \ref{serie parametrix} and Lemma \ref{LEMME_IT_KER}.


\begin{proposition}
\label{PROP_ProcGel}
Fix $(t,x) \in [0,T]\times \R^{nd}$. The unique solution of \eqref{FrozenSDE} starting from $x$ at time $t$ writes:
\begin{eqnarray}\label{ProcGel}
\tilde{X}_s^{t,x,T,y} &=& R_{s,t}x+ \int_t^s R_{s,u} B\sigma (u,R_{u,T}y) dZ_u.
\end{eqnarray}
\end{proposition}

\begin{proof}
Equation \eqref{FrozenSDE} is a linear SDE, with deterministic diffusion coefficient. As such, it admits a unique strong solution. The representation \eqref{ProcGel} follows from It\^o's formula.
%
%
\end{proof}

Introduce for all $u\in \R^+ $, the diagonal time scale matrix:
\begin{equation}\label{scale_matrix}
\T_{u}^\alpha=
\begin{pmatrix}
u^{\frac{1}{\alpha}} I_{d\times d}& &  &0 \\
0 & u^{1+\frac{1}{\alpha}} I_{d\times d} &  &0 \\
 & & \ddots&\\
0&& & u^{n-1+\frac{1}{\alpha}} I_{d\times d} 
\end{pmatrix},\ \M_{u}=u^{-\frac{1}\alpha}\T_u^\alpha
\begin{pmatrix}
 I_{d\times d}&        &  & 0 \\
0     & u I_{d\times d}& &  0\\
 & & \ddots&\\
0&  & & u^{n-1} I_{d\times d} 
\end{pmatrix}.
\end{equation}
 This extends the definitions of Theorem \ref{MTHM} for $n=2$. 
 The entries of the matrix $\T_u^\alpha $ correspond to the \textit{intrinsic} time scales of the iterated integrals of a stable process with index $\alpha $ observed at time $u$. They reflect the multi-scale behavior of our system. The matrix $\M_u$ appears in the tempered case.
 We first give an expression of the density of $\tilde X_s^{t,x,T,y}$ in terms of its inverse Fourier transform. We refer to Section \ref{Estimates on the frozen density} for the proof of this result.

\begin{proposition} \label{FROZEN_DENSITY}
The frozen process $(\tilde{X}_s^{t,x,T,y})_{s\ge t}$ has for all $s>t $ a density w.r.t. the Lebesgue measure, that is:
$$
\proba( \tilde{X}_s^{T,y}  \in dz | \tilde{X}_t^{T,y} = x) =\tilde{p}^{T,y}_\alpha(t,s,x,z) dz.
$$
For $0< T-t\le T_0:=T_0(\H)\le 1$ we have: 
\begin{eqnarray}\label{density}
 \tilde{p}^{T,y}_\alpha(t,s,x,z)
 =  \frac{\det {  (\M_{s-t})  }^{-1}  }{(2 \pi) ^{nd}}\nonumber\\
 \int_{\R^{nd}}  e^{-i\langle q, { ( \M^{\alpha}_{s-t})  }^{-1}(  z-R_{s,t}x) \rangle}
  \exp \left( - (s-t)\int_{\R^{nd}} \{1-\cos(\langle q,\xi\rangle) \} \nu_S(d\xi) \right) dq,
\end{eqnarray}
where $\nu_S:=\nu_S(t,T,s,y) $ is a symmetric measure on $\R^{nd} $ s.t. uniformly in $s\in (t,t+T_0]$ for all $A\subset \R^{nd} $:
\begin{equation*}
\nu_S(A)\le \int_{\R^+}\frac{d\rho}{\rho^{1+\alpha}}\int_{S^{nd-1}}\I_{A}(s\xi) g(c\rho)\bar \mu(d\xi),
\end{equation*}
with $\bar \mu $ satisfying \textbf{[H-4]}
 and ${\rm dim}({\rm supp}(\bar \mu))=d $. In the stable case, i.e. $g=1 $, we have the equality in the above equation, so that $\nu_S $ indeed corresponds to a stable L\'evy measure.
\end{proposition}

\begin{remark}
The above proposition is important in that it shows in the stable case \textbf{[HS]} why the density of a $d $-dimensional stable process with index $\alpha\in(0,2) $ and its $n-1$ iterated integrals actually behaves as the density of an $nd$-dimensional \textit{multi-scale} stable process, where the various scales are read through the matrix $\T^\alpha $. Also, the fact that the associated spectral measure is either non equivalent or singular w.r.t. the Lebesgue measure of $S^{nd-1}$ leads to consider delicate asymptotics for the tails of the density which yields the dimension constraints in Theorem \ref{MART_PB_THM} and the restrictions of Theorem \ref{MTHM}.
\end{remark}

From the previous remark and the dimension of the support of $\bar \mu $ in Proposition \ref{FROZEN_DENSITY} we derive from points \textit{i)} and \textit{iii)} in Theorem 1.1 in Watanabe \cite{wata:07} the following estimate in the stable case. 

\begin{proposition}[Density Estimates for the frozen process under \textbf{[HS]}]\label{EST_DENS_GEL}
Fix $T>0$, a threshold $K>0$ and $y\in \R^{nd} $. For all $(t,x)\in [0,T)\times \R^{nd} $, 
the density  $\tilde{p}^{T,y}_\alpha(t,s,x,z)$ of the frozen process  $ (\tilde X_s^{t,x,T,y})_{s\in (t,T]}$ 
in \eqref{ProcGel}
satisfies the following estimates.
There exists $C_{\ref{EST_DENS_GEL}}:=C_{\ref{EST_DENS_GEL}}$ (\textbf{H-2,H-3,H-4},K)$\ge 1$, s.t. for all $0\le t<s\le T,\  (x,z)\in (\R^{nd})^2$: 
\begin{equation}
\label{BORNE_COMPACTE}
C_{\ref{EST_DENS_GEL}}^{-1}\underline{p}_\alpha(t,s,x,z)\le \tilde p_\alpha^{T,y}(t,s,x,z)\le C_{\ref{EST_DENS_GEL}}\bar p_\alpha(t,s,x,z),
\end{equation}
where 
we write:
\begin{equation}
\label{BAR_P}
\bar p_\alpha(t,s,x,z)= C_{\alpha}\frac{{\rm det}(\T_{s-t}^\alpha)^{-1}}{\{K \vee |(\T_{s-t}^{\alpha})^{-1}(z-R_{s,t}x)|\}^{d+1+\alpha}},
\end{equation}
and also
\begin{equation*}
\underline p_\alpha(t,s,x,z)= C_{\alpha}^{-1}\frac{{\rm det}(\T_{s-t}^\alpha)^{-1}}{\{K \vee |(\T_{s-t}^{\alpha})^{-1}(z-R_{s,t}x)|\}^{n d(1+\alpha)}},\ C_\alpha:=C_\alpha(\H)\ge 1.
\end{equation*}
\end{proposition}
We refer to Section \ref{Estimates on the frozen density} for the proof of this result. Observe that $\bar p_\alpha(t,s,x,.)$ can be identified with a probability density only under the condition $d(1-n)+1+\alpha >0 $ appearing in Theorem \ref{MART_PB_THM}. Roughly speaking the upper-bound in \eqref{BORNE_COMPACTE} is the worst possible considering the underlying dimension of the support of the spectral measure in $S^{nd-1} $, which is here $d$. On the other hand, the lower bound corresponds to the highest possible concentration of a spectral measure on $S^{nd-1}$ satisfying \textbf{[H-4]}, see again Section \ref{Estimates on the frozen density}. This control would for instance correspond to the product at a given point of the one-dimensional stable asymptotics in each direction. 

From Proposition \ref{FROZEN_DENSITY} and Theorem 1 in Sztonyk \cite{szto:10} we also derive the following result in the tempered case.
\begin{proposition}[Density Estimates for the frozen process under \textbf{[HT]}]\label{EST_DENS_GEL_T}
Fix $T>0$, a threshold $K>0$ and $y\in \R^{nd} $. For all $(t,x)\in [0,T)\times \R^{nd} $, 
the density  $\tilde{p}^{T,y}_\alpha(t,s,x,z)$ of the frozen process  $ (\tilde X_s^{t,x,T,y})_{s\in (t,T]}$ 
in \eqref{ProcGel}
satisfies the following estimates.
There exists $C_{\ref{EST_DENS_GEL_T}}:=C_{\ref{EST_DENS_GEL_T}}$ (\textbf{H-2,H-3,H-4},K)$\ge 1$, s.t. for all $0\le t<s\le T,\  (x,z)\in (\R^{nd})^2$: 
\begin{equation}
\label{BORNE_COMPACTE_T}
 \tilde p_\alpha^{T,y}(t,s,x,z)\le C_{\ref{EST_DENS_GEL_T}}\bar p_\alpha(t,s,x,z),
\end{equation}
where:
\begin{equation}
\label{BAR_PT}
\bar p_\alpha(t,s,x,y)= C_{\alpha}\frac{{\rm det}(\T_{s-t}^\alpha)^{-1}}{\{K \vee |(\T_{s-t}^{\alpha})^{-1}(y-R_{s,t}x)|\}^{d+1+\alpha}} \theta(|(\M_{s-t})^{-1}(y-R_{s,t}x)|
).
\end{equation}
\end{proposition}

As a corollary, we have the following important property.
\begin{corollary}[\textbf{``Semigroup" property}]\label{SG_PROP} Under \H,
when $d=1,n=2$, which is the only case for which $nd+\alpha=d+1+\alpha $ so that $\bar p_\alpha $ can be identified with the density of a, possibly tempered, multi-scale stable process in dimension $nd$ whose spectral measure is absolutely continuous with respect to the Lebesgue measure of $S^{nd-1} $, there exists $C_{\ref{SG_PROP}}:=C_{\ref{SG_PROP}}$ (\textbf{H-2,H-3,H-4},K) $\ge 1$ s.t. for all $0\le t<\tau<s,\ (x,y)\in (\R^{nd})^{2} $:
$$\int_{\R^{nd}}\bar p_\alpha(t,\tau,x,z)\bar p_\alpha(\tau,s,z,y) dz\le C_{\ref{SG_PROP}}\bar p_{\alpha}(t,s,x,y).$$
\end{corollary}
The above control is important since it allows to give estimates on the convolution of the frozen densities with possible different freezing points.
%
Namely, for all $T_1,T_2 > 0$, $y_1,y_2 \in \R^{nd}$, for all $t<\tau< s$ and $x,y \in \R^{nd}$:
\begin{equation}\label{SG_DIFF_SPEC}
\int_{\R^{nd}} \tilde{p}^{T_1,y_1}_\alpha(t,\tau,x,z)\tilde{p}^{T_2,y_2}_\alpha(\tau,s,z,y) dz\le C_{\ref{SG_PROP}} \bar{p}_{\alpha}(t,s,x,y).
\end{equation}


\subsection{The Parametrix Series.}
\label{SEC_PARAM}
We assume here that the generator $(L_t)_{t\ge 0} $ of \eqref{EDS 1} generates a  two-parameter Feller semigroup $(P_{t,s})_{0\le t\le  s}$.
Using the Chapman-Kolmogorov equations satisfied by the semigroup and the pointwise Kolmogorov equations for the \textit{proxy} model, we derive a formal representation of the semigroup in terms of a series, involving the difference of the generators of the initial and frozen processes.  
Let $L_t$ (already defined in \eqref{GEN}) and $\tilde{L}_t^{T,y}$ denote the generators of $X^{t,x}$ and $\tilde{X}^{t,x,T,y}$ at time $t$ respectively.  For $\varphi \in {C}^2_0(\R^{nd},\R)$, from \eqref{GEN_TEMP} (setting $g=1$ in the stable case), 
we have for all $ x\in \R^{nd}$: 
\begin{equation*} 
L_t\varphi(x) =\langle \nabla \varphi(x), A_t x \rangle +
\int_{\R^{d}} \left( \varphi(x + B\sigma(t,x)z) - \varphi(x) - \dfrac{\langle \nabla \varphi(x), B\sigma(t,x)z \rangle }{1+ |z|^2 }  \right)g(|z|)\nu(dz),
\end{equation*}
\begin{equation}\label{frozen generator}
\tilde{L}^{T,y}_t\varphi(x) =\langle \nabla \varphi(x) ,A_t x \rangle + 
\int_{\R^{d}} \left(\varphi(x+ B\sigma(t,R_{t,T}y)z) -\varphi(x) - \frac{\langle \nabla \varphi(x), B\sigma(t,R_{t,T}y)z \rangle}{1+|z|^2} \right)g(|z|)\nu(dz).
\end{equation}

Observe that for $\tilde{X}^{t,x,T,y}_s$ defined in \eqref{ProcGel}, its density $\tilde{p}^{T,y}_\alpha(t,s,x, \cdot)$ exists and is smooth under  \textbf{[H]} for $s>t$ (see Proposition \ref{FROZEN_DENSITY} above). 

\begin{proposition}\label{serie parametrix}
Suppose that there exists a unique weak solution $(X_s^{t,x})_{0\le t\le s}$ to \eqref{EDS 1} which  has a two-parameter Feller semigroup $(P_{t,s})_{0\le t\le s}$. We have the following \textit{formal} representation. For all $0\le t<T,\ (x,y)\in (\R^{nd})^2 $ and any bounded measurable $f:\R^{nd}\rightarrow \R$:
\begin{equation}
\label{SG_PARAM}
P_{t,T}f(x)=\E[f(X_T)|X_t=x]=\int_{\R^{nd}} \left(\sum_{r=0}^{+\infty} (\tilde{p}_{\alpha} \otimes H^{(r)}) (t,T,x,y) \right)f(y)dy,
\end{equation}
where $H$ is the parametrix kernel:
\begin{equation}\label{noyau}
\forall 0\le t<T,\ (x,y)\in (\R^{nd})^2,\ H(t,T,x,y) := (L_t-\tilde{L}_t^{T,y})\tilde{p}_\alpha
(t,T,x,y).
\end{equation}
In equations \eqref{SG_PARAM}, \eqref{noyau}, we denote for all $0\le t<u\le T,\ (x,z)\in (\R^{nd})^2,\ \tilde p_\alpha(t,u,x,z):=\tilde p_\alpha^{u,z}(t,u,x,z)$, i.e. we omit the superscript when the freezing terminal time and point are  those where the density is considered. 
Also, the notation $\otimes  $ stands for the time space convolution: 
$$
f \otimes h(t,T,x,y) = \int_t^T du\int_{\R^{nd}} dz f(t,u,x,z)h(u,T,z,y).
$$
Besides, $H^{(0)} = I$ and $\forall r \in \N, \ H^{(r)} (t,T,x,y) =H^{(r-1)} \otimes H(t,T,x,y)$.

Furthermore, when the above representation can be justified, it yields the existence as well as a representation for the density of the initial process. Namely $\P[X_T\in dy|X_t=x]=p(t,T,x,y)dy $ where :
\begin{equation} \label{parametrix}
\forall 0\le t<T,\ (x,y)\in (\R^{nd})^2,\ p(t,T,x,y) = \sum_{r=0}^{+\infty} (\tilde{p}_{\alpha} \otimes H^{(r)}) (t,T,x,y).
\end{equation}
\end{proposition}

\begin{proof}
Let us first emphasize that the density $\tilde{p}_\alpha^{T,y}(t,s,x,z)$ of $\tilde X_{s}^{t,x,T,y} $ at point $z$ solves the Kolmogorov \textit{backward} equation:
\begin{equation}\label{KFE}
\frac{\partial \tilde{p}_\alpha^{T,y}}{\partial t}(t,s,x,z) = -\tilde{L}^{T,y}_t\tilde{p}_\alpha^{T,y}(t,s,x,z), \mbox{ for all $t<s$, $(x,z) \in \R^{nd}\times \R^{nd}$, $\lim_{t\uparrow s}\tilde{p}_\alpha^{T,y}(t,s,\cdot,z) = \delta_{z}(\cdot)$}.
\end{equation}
Here, $\tilde{L}^{T,y}_t $ acts on the variable $x$. Let us now introduce the family of operators $(\tilde P_{t,s})_{0\le t\le s} $. For $0\le t\le s $ and any bounded measurable function $f:\R^{nd}\rightarrow \R$:
\begin{eqnarray}
\label{SG_TILDE}
\tilde P_{t,s}f(x):=\int_{\R^{nd}}\tilde p_\alpha(t,T,x,y)f(y)dy:=\int_{\R^{nd}}\tilde p_\alpha^{T,y}(t,T,x,y)f(y)dy.
\end{eqnarray}
Observe that the family $(\tilde P_{t,s})_{0\le t\le s}  $ is not a two-parameter semigroup. Anyhow, we can still establish, see Lemma \ref{convergence_dirac}, that for a continuous $f$:
\begin{equation}
\label{LIM_SG_TILDE}
\lim_{s\rightarrow t}\tilde P_{s,t}f(x)=f(x).
\end{equation}
This convergence is not a direct consequence of the bounded convergence theorem since the freezing parameter is also the integration variable.

The boundary condition \eqref{LIM_SG_TILDE} and the Feller property yield:
$$
(P_{t,T}-\tilde P_{t,T})f(x) = \int_t^T du \frac{\partial}{\partial u}\biggl\{P_{t,u}(\tilde P_{u,T}f(x))\biggr\}.
$$
Computing the derivative under the integral leads to:
$$
(P_{t,T}-\tilde P_{t,T})f(x) = \int_t^T du \biggl\{\partial_u P_{t,u}(\tilde P_{u,T}f(x))+P_{t,u}(\partial_u(\tilde P_{u,T}f(x)))\biggr\}.
$$
Using the Kolmogorov equation \eqref{KFE} and the Chapman-Kolmogorov relation $\partial_{u}P_{t,u}\varphi(x) =P_{t,u}(L_u \varphi(x)),\ \forall \varphi\in C_b^{2}(\R^{nd},\R) $  we get:
\begin{eqnarray*}
(P_{t,T}-\tilde P_{t,T})f(x) = \int_t^T du  P_{t,u}\Big(L_u \tilde{P}_{u,T}f\Big)(x) - P_{t,u}\left(\int_{\R^{nd}}f(y) \tilde L_u^{T,y}\tilde p_\alpha(u,T,\cdot,y) dy \right)(x).
\end{eqnarray*}
Define now the operator:
\begin{equation}
\label{H_CUR}
\mathcal{H}_{u,T}\varphi(z) := \int_{\R^{nd}}\varphi(y)(L_u - \tilde L_u^{T,y})\tilde p_\alpha(u,T,z,y)dy=\int_{\R^{nd}}\varphi(y) H(u,T,z,y)dy.
\end{equation}
We can thus rewrite:
$$
P_{t,T}f(x) = \tilde P_{t,T}f(x) + \int_t^T P_{t,u}\big(\mathcal{H}_{u,T}(f) \big)(x)du.
$$
The idea is now to reproduce this procedure for $P_{t,u} $ applied to $\mathcal{H}_{u,T}(f) $.
This recursively yields the formal representation:  
$$
P_{t,T}f(x) = \tilde P_{t,T}f(x)+ \sum_{r \ge 1} \int_t^T du_1 \int_t^{u_1} du_2 \dots \int_t^{u_{r-1}} du_r \tilde{P}_{t,u_r}
\big(\mathcal{H}_{u_r,u_{r-1}}\circ \dots \circ \mathcal{H}_{u_1,T}\big)(f)(x).
$$
Equation \eqref{SG_PARAM} then formally follows from the following identification. For all $r\in \N^*$:
 $$
 \int_t^T du_1 \int_t^{u_1} du_2 \dots \int_t^{u_{r-1}} du_r \tilde{P}_{t,u_r}
\big(\mathcal{H}_{u_r,u_{r-1}}\circ \dots \circ \mathcal{H}_{u_1,T}\big)(f)(x)du = \int_{\R^{nd}}f(y) \tilde p_\alpha \otimes H^{(r)}(t,T,x,y)dy.
 $$
 We can proceed by immediate induction:
 \begin{eqnarray}
 &&\int_t^T du_1 \int_t^{u_1} du_2 \dots \int_t^{u_{r-1}} du_r \tilde{P}_{t,u_r}
\big(\mathcal{H}_{u_r,u_{r-1}}\circ \dots \circ \mathcal{H}_{u_1,T}\big)(f)(x)du \nonumber\\
&=& 
\int_t^T du_1 \int_t^{u_1} du_2 \dots \int_t^{u_{r-1}} du_r \int_{\R^{nd}}dz\mathcal{H}_{u_r,u_{r-1}}\circ \dots \circ \mathcal{H}_{u_1,T}(f)(z) \tilde p_\alpha(t,u_{r},x,z) \label{POUR_BOUCLER}\\
&\overset{\eqref{H_CUR}}{=}&\int_t^T du_1 \int_t^{u_1} du_2 \dots \int_t^{u_{r-1}} du_r \int_{\R^{nd}}dz\int_{\R^{nd}}dy \mathcal{H}_{u_{r-1},u_{r-2}}\circ \dots \circ \mathcal{H}_{u_1,T}(f)(y) H(u_r,u_{r-1},z,y)\tilde p_\alpha(t,u_{r},x,z) \nonumber \\
&=&\int_t^T du_1 \int_t^{u_1} du_2 \dots \int_t^{u_{r-2}} du_{r-1}\int_{\R^{nd}}dy \mathcal{H}_{u_{r-1},u_{r-2}}\circ \dots \circ \mathcal{H}_{u_1,T}(f)(y) \tilde p_\alpha \otimes H (t,u_{r-1},x,y) .\nonumber
 \end{eqnarray}
 
Thus, we can iterate the procedure from \eqref{POUR_BOUCLER} with $\tilde p_\alpha \otimes H$ instead of $\tilde p_\alpha$.

\end{proof}
Observe that in order to make the identification above, we have exchanged various integrals.
Hence, so far the representation \eqref{parametrix} is  \textit{formal}.
It will become rigorous provided that we manage to show the convergence of the series and get integrable bounds on its sum.
To achieve these points, one needs to give precise bounds on the iterated time-space convolutions appearing in the series. Such controls are stated in Section \ref{CTR_ITER_KER} and proved in Section \ref{SECTION_TECH} below.

\subsection{Controls on the iterated kernels.}
\label{CTR_ITER_KER}

\indent

From now on, we assume w.l.o.g. that $0<T\le T_0:=T_0(\H)\le 1$. The choice of $T_0$ depends on the constants appearing in $\H$ and will be clear from the proof of Lemma \ref{Form of the resolvent}.
Theorems \ref{MART_PB_THM} and \ref{MTHM} can anyhow be obtained for an arbitrary fixed finite $T>0$, from the results for $T$ sufficiently small. Indeed, the uniqueness of the martingale problem simply follows from the Markov property whereas the upper density estimate stems from the \textit{semigroup} property  of $\bar p_\alpha $ (see Corollary \ref{SG_PROP} and Lemma \ref{LEMME_CONV_LOG} for the convolutions involving the logarithmic correction). From now on, we consider that the threshold $K>0$ appearing in Proposition \ref{EST_DENS_GEL} is fixed.

We first give pointwise results on the convolution kernel, that hold in any dimension $d$, and for any number of oscillators $n$.
\begin{lemma}[\textbf{Control of the kernel}]
\label{CTR_KER_PTW}
Fix $K, \delta>0 $.  There exists $C_{\ref{CTR_KER_PTW}}:=C_{\ref{CTR_KER_PTW}}(\H,K,\delta) >0
$  s.t. for all $T\in (0,T_0]$ and $(t,x,y)\in [0,T)\times (\R^{nd})^2 $: 
\begin{equation}\label{majoration H}
|H(t,T,x,y)| \leq C_{\ref{CTR_KER_PTW}}\frac{\delta \wedge |x-R_{t,T}y|^{\eta(\alpha\wedge 1)}}{T-t} \left\{\bar p_\alpha(t,T,x,y)+\breve{p}_\alpha(t,T,x,y) \right\},
\end{equation}
where $\bar p_\alpha $ is as in \eqref{BAR_P} in the stable case \textbf{[HS]} and as in \eqref{BAR_PT} in the tempered case \textbf{[HT]}. Also,
\begin{eqnarray*}
\breve{p}_\alpha(t,T,x,y)=\frac{\I_{|(x-R_{t,T}y)^1|/(T-t)^{1/\alpha}| \asymp |(\T_{T-t}^{\alpha})^{-1}(x-R_{t,T}y)|\ge K}}{(T-t)^{d/\alpha}(1+\frac{|(x-R_{t,T}y)^1|}{(T-t)^{1/\alpha}})^{d+\alpha}}\\
\times \frac{1}{(T-t)^{\frac{(n-1)d}\alpha+\frac{n(n-1)d}2}(1+|((\T_{T-t}^{\alpha})^{-1}(x-R_{t,T}y)^{2:n}|)^{1+\alpha}} \theta( |\M_{T-t}^{-1}(x-R_{t,T}y)|
),
\end{eqnarray*}
recalling that under \textbf{[HS]}, $g(r)=1, r>0$.
\end{lemma}
The contribution in $\breve{p}_\alpha  $ comes from the \textit{rediagonalization} phenomenon which is specific to the degenerate, non-local case and only appears when the rescaled first (slow) component is equivalent to the energy $|(\T_{T-t}^{\alpha})^{-1}(x-R_{t,T}y)| $.
Observe that if $|(\T_{T-t}^{\alpha})^{-1}(x-R_{t,T}y)|\le K $, diagonal regime, both contributions $\bar p $ and $\breve{p} $ can be upper-bounded by $(T-t)^{-nd/\alpha+n(n-1)d/2} $.  In the off-diagonal case, we also have that if there exists $i\in \leftB 2,n\rightB$ s.t.
 $|((\T_{T-t}^{\alpha})^{-1}(x-R_{t,T}y))^{i}|\asymp |((\T_{T-t}^\alpha)^{-1}(x-R_{t,T}y))^{1}| $ then $\breve{p}_\alpha(t,T,x,y)\le \bar p_\alpha(t,T,x,y) $.

Once integrated in space, under the dimension constraints of Theorem \ref{MART_PB_THM}, this pointwise estimate yields the following \textit{smoothing} property in time.
\begin{lemma}\label{lemme}
Assume that $d(1-n)+1+\alpha >0 $. Then, there exists $C_{\ref{lemme}}:=C_{\ref{lemme}}(\H,K)$ and $\omega:=\omega(d,n,\alpha)>0 $ s.t. for all $T\in (0,T_0], (x,y)\in (\R^{nd})^2$, $\tau \in [t,T)$, we have the estimate
\begin{eqnarray}
\int_{\R^{nd}} \delta \wedge |z-R_{\tau,T}y|^{\eta(\alpha \wedge 1)}(\bar{p}_\alpha+\breve{p}_\alpha)(\tau,T,z,y) dz &\le & C_{\ref{lemme}} (T-\tau)^{\omega},\\ \label{lemme backward}
\int_{\R^{nd}} \delta \wedge |z-R_{\tau,t}x|^{\eta(\alpha \wedge 1)}
\bar{p}_\alpha
(t,\tau,x,z) dz &\le & C_{\ref{lemme}} (\tau-t)^{\omega}. \label{lemme forward}
\end{eqnarray}
Also, when $d=1,n=2$ one has the following better smoothing property for the \textit{fast} variable:
\begin{eqnarray}
\int_{\R^{nd}} \delta \wedge |(z-R_{\tau,T}y)^2|^{\eta(\alpha \wedge 1)}(\bar{p}_\alpha+\breve{p}_\alpha)(\tau,T,z,y) dz &\le & C_{\ref{lemme}} (T-\tau)^{\tilde \omega},\\ \label{lemme backward spec}
\int_{\R^{nd}} \delta \wedge \{(\tau-t)|(z-R_{\tau,t}x)^1|+ |(z-R_{\tau,t}x)^2|\}^{\eta(\alpha \wedge 1)}
\bar{p}_\alpha
(t,\tau,x,z) dz &\le & C_{\ref{lemme}} (\tau-t)^{\tilde \omega}, \label{lemme forward spec}
\end{eqnarray}
with $\tilde \omega=(1+1/\alpha)\eta(\alpha \wedge 1) $.
\end{lemma}
The proof of these results will be given in Section \ref{Estimates on the kernel} and Appendix \ref{controles_phi}.
\begin{remark}
We can now justify from this Lemma our previous choice for the \textit{proxy} model. Indeed, the contributions $|z-R_{\tau,T}y|^{\eta(\alpha \wedge 1)} $, $|z-R_{\tau,t}x|^{\eta(\alpha \wedge 1)} $ come from the difference of the generators and turn out to be compatible, up to using the Lipschitz property of the flow, with the bounds appearing in Proposition \ref{EST_DENS_GEL} for the frozen density. This is what gives this smoothing property and thus allows to get rid of  the diagonal singularities coming from the bound \eqref{majoration H}.    
\end{remark}
\begin{remark}
The l.h.s. of equations \eqref{lemme backward spec}, \eqref{lemme forward spec} naturally appear in the case $\sigma(t,x)=\sigma(t,x^2) $ which is the one considered for the density estimates in Theorem \ref{MTHM}. Intuitively, the higher smoothing effect in this case permits to compensate the difficulties arising from the rediagonalization in the degenerate case.
\end{remark}

When dealing with convolutions of the kernel and the frozen density we restrict to the case $d=1, n=2$ for which we have the \textit{semigroup} property, which is important to handle the off-diagonal regimes. 
In this framework, the technical computations in Section \ref{SECTION_TECH}, based on the previous controls on the kernel $H$, yield the following bound for the first  step of the parametrix procedure.
\begin{lemma}\label{CTRL_PRELIM_NON_FONCTIONNEL} 
There exist $C_{\ref{CTRL_PRELIM_NON_FONCTIONNEL}}:=C_{\ref{CTRL_PRELIM_NON_FONCTIONNEL}}(\H,K),\ \omega:=\omega(\H)\in (0,1]$ s.t.
for all $T\in (0,T_0]$ and $(t,x,y)\in [0,T)\times (\R^{nd})^2 $:
\begin{eqnarray*}
|\tilde p_\alpha\otimes H(t,T,x,y)| \le C_{\ref{CTRL_PRELIM_NON_FONCTIONNEL}}\Big (\bar p_{\alpha 
} (t,T,x,y)
\big ((T-t)^\omega+\\ 
 \delta \wedge |x-R_{t,T}y|^{\eta(\alpha \wedge 1)}
\{1+\log(K \vee |(\T_{T-t}^\alpha)^{-1}(y-R_{T,t}x)|)\} \big)+
\delta \wedge |x-R_{t,T}y|^{\eta(\alpha \wedge 1)}\check{p}_\alpha(t,T,x,y)\}\Big)
 ,
\end{eqnarray*}
where 
\begin{eqnarray}
\check p(t,T,x,y):=
\inf_{\tau\in [t,T]}\frac{\I_{\frac{(R_{\tau,t}x-R_{\tau,T}y)^1}{(T-t)^{1/\alpha}}\asymp |(\T_{T-t}^\alpha)^{-1}(R_{\tau,t}x-R_{\tau,T}y)|}}{(T-t)^{1/\alpha}(1+|(\T_{T-t}^\alpha)^{-1}(R_{\tau,t}x-R_{\tau,T}y)|)^{1+\alpha}}\theta(|\M_{T-t}^{-1}(R_{\tau,t}x-R_{\tau,T}y)|)\nonumber\\
\times
\frac{1}{(T-t)^{1+1/\alpha}(1+\frac{|(R_{\tau,t}x-R_{\tau,T}y)^2|}{(T-t)^{1+\frac 1\alpha}})^{1+\alpha}}.
\label{CHECK_P}
\end{eqnarray}
\end{lemma}

The contribution in $\check p $, comes from the bad rediagonalization which is intrinsic to the degenerate case. It first generates a loss of concentration in the estimate, which leads us to temper the driving noise. It also turns out to be very difficult to handle in the iterated convolutions of the kernel.
Up to the end of section we thus restrict under \textbf{[HT]} to the case $d=1,n=2$ and $\sigma(t,x):=\sigma(t,x^2) $, for which we have been able to refine the above results and to derive the convergence of \eqref{parametrix}.
This restriction will be discussed thoroughly in Section \ref{SECTION_TECH}. 

\begin{lemma}[\textbf{Control of the iterated kernels}]\label{LEMME_IT_KER} 
Assume under \textbf{[HT]} that $d=1, n=2, \sigma(t,x)=\sigma(t,x^2)$ and $1\ge \eta>((\alpha \wedge 1) (1+\alpha))^{-1} $.
Then there exist $C_{\ref{LEMME_IT_KER}}:=C_{\ref{LEMME_IT_KER}}(\HT,K)$, $\omega:=\omega(\HT)\in (0,1]$ s.t.
for all $T
\le T_0$ and $(t,x,y)\in [0,T)\times (\R^{2})^2 $: 
\begin{eqnarray*}
|\tilde p_\alpha\otimes H(t,T,x,y)|&\le &C_{\ref{LEMME_IT_KER}}\Big ((T-t)^\omega \bar p_{\alpha,\Theta} (t,T,x,y)+\bar q_{\alpha,\Theta}(t,T,x,y) \Big),\\
|\bar{q}_\alpha \otimes H(t,T,x,y)|&\le &C_{\ref{LEMME_IT_KER}}  (T-t)^\omega \Big( \bar p_{\alpha,\Theta}(t,T,x,y)+\bar q_{\alpha,\Theta}(t,T,x,y) \Big),
\end{eqnarray*}
where we denoted  
\begin{eqnarray*}
\bar q_{\alpha,\Theta}(t,T,x,y) &=& \delta \wedge \{(T-t)|(x-R_{t,T}y)^1|+|(x-R_{t,T}y)^2|\}^{\eta(\alpha \wedge 1)}\} \\
&& \times \Big\{ \bar{p}_{\alpha,\Theta}(t,T,x,y) \Big( 1+ \log(K \vee |(\T_{T-t}^\alpha)^{-1}(y-R_{T,t}x)|) \ \Big) 
\Big\}  .
\end{eqnarray*}
Now for all $k\ge 1$,
\begin{eqnarray*}
|\tilde p_\alpha\otimes H^{(2k)}(t,T,x,y)|\le (4C_{\ref{LEMME_IT_KER}})^{2k} (T-t)^{k\omega} \Big( (T-t)^{k\omega} \bar {p}_{\alpha,\Theta}(t,T,x,y) + (\bar{p}_{\alpha,\Theta} + \bar{q}_{\alpha,\Theta})(t,T,x,y) \Big),  \\
|\tilde p_\alpha\otimes H^{(2k+1)}(t,T,x,y)|\le (4C_{\ref{LEMME_IT_KER}})^{2k+1}(T-t)^{k\omega} \Big( (T-t)^{(k+1)\omega} \bar{p}_{\alpha,\Theta}+ (T-t)^\omega(\bar{p}_{\alpha,\Theta} + \bar{q}_{\alpha,\Theta})+ \bar{q}_{\alpha,\Theta} \Big)(t,T,x,y).
\end{eqnarray*}
\end{lemma}


The above controls allow to derive under the sole assumption $\HT$ an upper bound for the sum of the parametrix series \eqref{parametrix} in small time. 
\begin{proposition}[\textbf{Sum of the parametrix series}]\label{CTR_SUM}
Under the assumptions of Lemma \ref{LEMME_IT_KER}, for $T_0$ small enough, there exists 
 $C_{\ref{CTR_SUM}}:=C_{\ref{CTR_SUM}}(\HT,K, T_0) $ s.t. for all $T\in (0, T_0]$  and $(t,x,y)\in [0,T)\times (\R^{2})^2 $:
\begin{eqnarray*}
\sum_{r\ge 0}|\tilde p_\alpha \otimes H^{(r)}(t,T,x,y)|&\le& C_{\ref{CTR_SUM}}\Big(\bar p_{\alpha,\Theta}(t,T,x,y) + \bar q_{\alpha,\Theta}(t,T,x,y) \Big),\\
C_{\ref{CTR_SUM}}{\rm det}(\T_{T-t}^\alpha)^{-1}&\le &\sum_{r\ge 0}\tilde p_\alpha \otimes H^{(r)}(t,T,x,y),\ {\rm for}\ |(\T_{T-t}^{\alpha})^{-1}(R_{T,t}x-y)|\le K.
\end{eqnarray*}
\end{proposition}
The proofs of Lemmas \ref{CTR_KER_PTW} and \ref{LEMME_IT_KER} are postponed to Section \ref{KER_ITER_ET_CONV_SERIE}. 
Using those controls on the iterated convolutions, we can prove Proposition \ref{CTR_SUM}.
\begin{proof}
The upper-bound can be readily derived from Lemma \ref{LEMME_IT_KER} for $T_0$ small enough (sum of a geometric series).
To get the diagonal lower bound, we first write:
$$
\sum_{k \geq 0} \tilde{p}_\alpha \otimes H^{(k)}(t,T,x,y) = \tilde{p}_\alpha(t,T,x,y)+\left(\sum_{k \geq 0} \tilde{p}_\alpha \otimes H^{(k)} \right) \otimes H(t,T,x,y).
$$
Now, since 
$$
\sum_{k \geq 0}| \tilde{p}_\alpha \otimes H^{(k)}(t,T,x,y)| \leq C (\bar{p}_{\alpha,\Theta} + \bar{q}_{\alpha,\Theta})(t,T,x,y),
$$
we derive:
$$
\left|\left(\sum_{k \geq 0} \tilde{p}_\alpha \otimes H^{(k)} \right) \otimes H(t,T,x,y)\right| \leq C|(\bar{p}_{\alpha,\Theta} + \bar{q}_{\alpha,\Theta}) \otimes H(t,T,x,y)|.
$$
Using once again the first part of Lemma \ref{LEMME_IT_KER}, we thus get that
\begin{eqnarray*}
\left|\left(\sum_{k \geq 0} \bar {p}_\alpha \otimes H^{(k)} \right) \otimes H(t,T,x,y)\right| 
\leq C\big\{(T-t)^\omega \bar{p}_{\alpha,\Theta}(t,T,x,y) + \bar{q}_{\alpha,\Theta}(t,T,x,y) \\
+ (T-t)^\omega (\bar{p}_{\alpha,\Theta} + \bar{q}_{\alpha,\Theta})(t,T,x,y)\big\}.
\end{eqnarray*}
Now, if the global regime is diagonal, i.e. $|(\T_{T-t}^\alpha)^{-1}(y-R_{T,t}x)| \le K$, the logarithm contribution vanishes in $\bar q_{\alpha,\Theta} $.
Observe also that 
\begin{eqnarray*}
\delta \wedge |x-R_{t,T}y|^{\eta(\alpha \wedge 1)} &\le& C^{\eta(\alpha \wedge 1)}|R_{T,t}x-y|^{\eta(\alpha \wedge 1)}\le C^{\eta(\alpha \wedge 1)}(T-t)^{\eta(1/\alpha\wedge 1)}|(\T_{T-t}^{\alpha})^{-1}(R_{T,t}x-y)|^{\eta(\alpha \wedge 1)}\\
&\le& (CK)^{\eta(\alpha \wedge 1)}(T-t)^{\eta(1/\alpha\wedge 1)}. 
\end{eqnarray*}
Hence
 $\left|\left(\sum_{k \geq 0} \tilde{p}_\alpha \otimes H^{(k)} \right) \otimes H(t,T,x,y)\right|  \le C (T-t)^\omega \det(\T_{T-t}^\alpha)^{-1}$. Taking $T-t$ small enough yields the announced bound.
\end{proof}

We conclude anyhow the section stating a  Lemma that allows to extend the upper bound in Theorem \ref{MTHM} to an arbitrary given fixed time. The arguments for its proof would be similar to those of Lemma \ref{DeuxiemeCoup}.
\begin{lemma}[\textbf{Semigroup property for $\bar q_{\alpha,\Theta} $}]\label{LEMME_CONV_LOG}
With the notations of Proposition \ref{CTR_SUM}, for any $T\in [0,\bar T_0)$, we have that there exists $C_{\ref{LEMME_CONV_LOG}}:=C_{\ref{LEMME_CONV_LOG}}([\mathbf{HT}],\bar T_0)\ge 1$ s.t.:
\begin{equation*}
\forall (x,y)\in \R^{nd},\ \forall n\in \N,\ \int_{\R^{nd}} \bar q_{\alpha,\Theta} (0,nT,x,z)\bar q_{\alpha,\Theta}(nT,(n+1)T,z,y)dz\le C_{\ref{LEMME_CONV_LOG}}^{n+2}\bar q_{\alpha,\Theta}(0,(n+1)T,x,y).
\end{equation*}
\end{lemma}

Observe now that Theorem \ref{MART_PB_THM} yields that $(X_t)_{t\ge 0}$, the canonical process of $\P$, admits a Feller transition function.
On the other hand, when $d=1,n=2 $ we have from Proposition \ref{CTR_SUM} that the series appearing in equation \eqref{SG_PARAM} of Proposition \ref{serie parametrix} is absolutely convergent. This allows to derive that the Feller transition is absolutely continuous, which in particular means that  the process $(X_t)_{t\ge 0}$ admits for all $t>0$ a density, satisfying the bounds of Proposition \ref{CTR_SUM}.

\mysection{Proof of the uniqueness of the Martingale Problem associated with \eqref{EDS 1}.}
\label{uniqueness for MP}

In this section, $d$ and  $n$ satisfy the conditions $d(1-n)+1+\alpha>0 $. 
As a corollary to the bounds of Section \ref{CTR_ITER_KER}, specifically Lemmas \ref{CTR_KER_PTW} and \ref{lemme} (controls on the kernel and associated smoothing effect), we prove here Theorem \ref{MART_PB_THM}. The existence of a solution to the martingale problem can be derived by compactness arguments adapting the proof of Theorem 2.2 from \cite{stroock1975diffusion}, even though our coefficients are not bounded.

\begin{proof}[Uniqueness of the Martingale Problem associated with \eqref{GEN}.]
Suppose we are given two solutions $\mathbb \P^1$ and $\mathbb \P^2$ of the martingale problem associated with $(L_s)_{s\in [t,T]} $, 
starting in $x$ at time $t$. We can assume w.l.o.g. that $T\le T_0:=T_0(\H)$.
Define for a bounded Borel function $f:[0,T]\times \R^{nd}\rightarrow \R$, 
$$S^i f = \E^i \left( \int_t^T f(s,X_s) ds\right),\ i\in \{1,2\},$$
where $(X_s)_{s\in [t,T]}$ stands for the canonical process associated with $(\P^i)_{i\in \{ 1,2\}} $. Let us specify that $S^if$ is \textit{a priori} only a linear functional and not a function since $\P^i$ does not need to come from a Markov process.
We denote: 
$$S^\Delta f = S^1f-S^2f.$$ 
If $f \in 
C^{1,2}_0([0,T)\times \R^{nd},\R)$, since $(\P^i)_{i\in \{1,2\}}$ both solve the martingale problem, we have:
\begin{equation}
\label{REL_MART_P}
f(t,x) + \E^i\left(\int_t^T (\partial_s+L_s) f(s,X_s) ds \right) =0,\ i\in\{ 1,2\}.
\end{equation}
For a fixed point $y\in \R^{nd}$ and a given $\varepsilon\ge 0$, introduce now for all $ f\in C^{1,2}_0([0,T)\times \R^{nd},\R)$ the Green function:
$$
\forall (t,x) \in [0,T)\times\R^{nd},  G^{\varepsilon,y}f(t,x) = \int_t^T ds \int_{\R^{nd}} dz \tilde{p}_\alpha^{s+\varepsilon, y}(t,s, x,z) f(s,z).
$$
We recall here that $\tilde p_\alpha^{s+\varepsilon,y}(t,s,x,z) $ stands for the density at time $s$ and point $z$ of the process $\tilde X^{s+\varepsilon,y} $ defined in \eqref{ProcGel} starting from $x$ at time $t$. In particular, $\varepsilon $ can be equal to zero in the previous definition.
One now easily checks that:
\begin{equation}\label{relation differentielle}
\forall (t,x,z) \in [0,s) \times (\R^{nd})^2,  \Big(\partial_t + \tilde{L}^{s+\varepsilon, y}_t \Big) \tilde{p}_{\alpha}^{s+\varepsilon,y} (t,s,x,z) = 0, 
\lim_{s \downarrow t} \tilde{p}_\alpha^{s+\varepsilon,y}(t,s,x, \cdot) = \delta_x(.).
\end{equation}
Introducing for all $f\in C^{1,2}_0([0,T)\times \R^{nd},\R) $ the quantity:
\begin{equation}
M^{\varepsilon,y}_{t,x} f(t,x) =\int_t^Tds \int_{\R^{nd}} dz \tilde{L}_t^{s+\varepsilon,y} \tilde{p}_\alpha^{s+\varepsilon,y}(t,s,x,z)f(s,z),
\end{equation}
we derive from \eqref{relation differentielle} and the definition of $G^{\varepsilon,y} $ that the following equality holds:
\begin{equation}\label{relation differentielle 2}
\partial_t G^{\varepsilon,y} f(t,x)+ M^{\varepsilon,y}_{t,x}f(t,x) = -f(t,x),\ \forall (t,x)\in [0,T)\times \R^{nd}.
\end{equation}

Now, let  $h \in C^{1,2}_0([0,T)\times \R^{nd},\R)$ be an arbitrary function and define for all $(t,x)\in [0,T)\times \R^{nd} $:
\begin{eqnarray*}
\phi^{\varepsilon,y} (t,x) := \tilde{p}_\alpha^{t+\varepsilon,y}(t,t+\varepsilon,x,y)h(s,y), 
\Psi_\varepsilon(t,x):= \int_{\R^{nd}} dy G^{\varepsilon,y}(\phi^{\varepsilon,y})(t,x).
\end{eqnarray*}
Then, by semigroup property, we have:
\begin{eqnarray*}
\Psi_\varepsilon(t,x) &=& \int_{\R^{nd}} dy \int_t^T ds \int_{\R^{nd}}dz \tilde{p}_\alpha^{s+\varepsilon,y}(t,s,x,z) \tilde{p}_\alpha^{s+\varepsilon,y}(s,s+\varepsilon,z,y)h(s,y) \\
&=& \int_{\R^{nd}}dy \int_t^T ds \tilde{p}_\alpha^{s+\varepsilon,y}(t,s+\varepsilon,x,y) h(s,y).
\end{eqnarray*}
Hence, 
\begin{eqnarray*}
(\partial_t+L_t)\Psi_\varepsilon(t,x) &=& \int_{\R^{nd}}dy (\partial_t+L_t)(G^{\varepsilon,y} \phi^{\varepsilon,y})(t,x)\\
&=&  \int_{\R^{nd}}dy \{\partial_t G^{\varepsilon,y} \phi^{\varepsilon,y}(t,x) + M^{\varepsilon,y}_{t,x} \phi^{\varepsilon,y}(t,x)\}\\
&&+ \int_{\R^{nd}} dy \{L_tG^{\varepsilon,y} \phi^{\varepsilon,y}(t,x) -  M^{\varepsilon,y}_{t,x} \phi^{\varepsilon,y}(t,x)\}\\
&\overset{\eqref{relation differentielle 2}}{=}& -\int_{\R^{nd}} dy\phi^{\varepsilon,y}(t,x) +  \int_{\R^{nd}} dy \{L_tG^{\varepsilon,y} \phi^{\varepsilon,y}(t,x) -  M^{\varepsilon,y}_{t,x} \phi^{\varepsilon,y}(t,x)\}\\
&=& I^\varepsilon_1 + I^\varepsilon_2.
\end{eqnarray*}
We now need the following lemma whose proof is postponed to
the end of Section \ref{Estimates on the frozen density}.
\begin{lemma}\label{convergence_dirac}
For all bounded continuous function $f:\R^{nd}\rightarrow \R, x\in \R^{nd}$:
\begin{equation}
\left| \int_{\R^{nd}} f(y) \tilde{p}_{\alpha}^{T,y}(t,T,x,y)dy -f(x) \right| \underset{T\downarrow t} {\longrightarrow} 0.
\end{equation}
\end{lemma}
We emphasize that the above lemma is not a direct consequence of the convergence of the law of the frozen process towards the Dirac mass when $T\downarrow t$. Indeed, the integration parameter is also the freezing parameter which makes things more subtle.
Lemma \ref{convergence_dirac} yields $ I^\varepsilon_1\underset{\varepsilon \rightarrow 0}{\longrightarrow} -h(t,x) $.
On the other hand, we have the following identity:
\begin{eqnarray*}
I^\varepsilon_2 &=& \int_t^Tds \int_{\R^{nd}} dy ( L_t - \tilde{L}^{s+\varepsilon}_t )\tilde{p}_\alpha^{s+\varepsilon,y}(t,s+\varepsilon,x,y)h(s,y)\\
&=&\int_t^T ds \int_{\R^{nd}} dy H(t,s+\varepsilon,x,y)h(s,y).
\end{eqnarray*}
The bound of Lemmas \ref{CTR_KER_PTW} and \ref{lemme} now yield:
\begin{eqnarray*}
|I^\varepsilon_2 |&\leq&C\int_t^T ds\int_{\R^{nd}} dy \frac{\delta \wedge |x-R_{t,s+\varepsilon}y|^{\eta(\alpha \wedge 1)}}{s+\varepsilon - t}(\bar{p}_\alpha+\breve{p}_\alpha )(t,s+\varepsilon,x,y)|h(s,y)|\\
&\leq& C|h|_{\infty} \int_t^T(s+\varepsilon-t)^{\eta(\frac{1}{\alpha}\wedge1)-1} ds\le C|h|_{\infty} [(T-t)\vee \varepsilon]^{\eta(\frac{1}{\alpha}\wedge 1)}.
\end{eqnarray*}
Hence, we may choose $T$ and $\varepsilon$ small enough to obtain
\begin{equation}
\label{boundI2}
|I^\varepsilon_2|\leq 1/2 |h|_{\infty}.
\end{equation}
Observe now that \eqref{REL_MART_P} gives $S^\Delta \Big( (\partial_\cdot+L_\cdot) \Psi_\varepsilon\Big) =0 $ so that $|S^\Delta(I_1^\varepsilon)|=|S^\Delta(I_2^\varepsilon)| $.
From Lemma \ref{convergence_dirac} and \eqref{boundI2} we derive:
$$
|S^\Delta h| = \lim_{\varepsilon \rightarrow 0} |S^\Delta I^\varepsilon_1 | = 
\lim_{\varepsilon \rightarrow 0} |S^\Delta I^\varepsilon_2 | \leq \|S^\Delta\| \limsup_{\varepsilon \rightarrow 0}| I^\varepsilon_2 | \leq  1/2 \|S^\Delta\| |h|_{\infty},\ \|S^\Delta\|:=\sup_{|f|_\infty\le 1}|S^\Delta f|.
$$
By a monotone class argument, the previous inequality still holds for bounded Borel functions 
 $h$ compactly supported in $[0,T)\times \R^{nd} $.
Taking the supremum over $|h|_\infty\le 1$ leads to $\|S^\Delta\| \leq 1/2 \|S^\Delta\|$. Since  $\|S^\Delta\| \leq T-t$, we deduce that $\|S^\Delta\|=0$ which proves the result on $[0,T] $. Regular conditional probabilities allow to extend the result on $\R^+ $, see e.g. Theorem 4, Chapter II, \S 7, in \cite{shir:96}, see also Chapter 6 in \cite{stro:vara:79} and \cite{stroock1975diffusion}. 

\end{proof}

\mysection{Proof of the results involving the Frozen process.}
\label{TEC_SEC1}
Introduce for a given $t>0$ and all $s\ge t$ the process: 
\begin{eqnarray}
\label{DEF_LAMBDA}
\Lambda_s:=\int_t^s R_{s,u}B\sigma_u dZ_u, 
\end{eqnarray}
solving $d\Lambda_s=A_s\Lambda_s ds+B\sigma_s dZ_s,\ Z_t=0 $, i.e. $\Lambda_s$ can be viewed as the process of the iterated integrals of $Z$ weighted by the entries of the resolvent. In \eqref{DEF_LAMBDA}, $(\sigma_u)_{u\ge t} $ is a deterministic $ \R^d\otimes\R^d$-valued function s.t. $(\sigma_u\sigma_u^*)_{u\ge t} $ satisfies [\textbf{H-2}] (uniform ellipticity). It can be seen from Proposition \ref{PROP_ProcGel} that the frozen process will have a density if and only if $\Lambda $ does for $s>t$.
This is what we establish through Fourier inversion.
 The structure of the resolvent is crucial: it gives the multi-scale behaviour of the frozen process and allows to prove in Proposition \ref{LAB_EDFR} that the Fourier transform is integrable. Recalling as well that $B $ stands for the embedding matrix from $\R^d $ into $\R^{nd} $, we observe that only the first  $d$ columns of the resolvent 
 are taken into account in \eqref{DEF_LAMBDA}. Reasoning by blocks we rewrite: $R_{s,t}=\left( \begin{array}{ccc}
 R_{s,t}^{1,1}& \cdots &R_{s,t}^{1,n}\\
 \vdots & \ddots &  \vdots\\
 R_{s,t}^{n,1}& \cdots & R_{s,t}^{n,n}
 \end{array}
 \right)$, where the entries $(R_{s,t}^{i,j})_{(i,j)\in \leftB 1,n\rightB^2} $ belong to $\R^d\otimes \R^d $.

\subsection{Analysis of the Resolvent.}

\begin{lemma}[\textbf{Form of the Resolvent}]\label{Form of the resolvent}
Let $0\le t\le s\le T\le T_0:=T_0(\H)\le 1$. We can write the first column of the resolvent in the following way:
\begin{equation}\label{R1}
R_{s,t}^{\cdot,1} = 
\begin{pmatrix}
\bar R^{1}_{s,t}\\
(s-t)\bar R^2_{s,t}\\
\vdots\\
\frac{(s-t)^{n-1}}{(n-1)!} \bar R_{s,t}^n
\end{pmatrix},
\end{equation}
where the $(\bar R_{s,t}^i)_{i\in \leftB 1,n\rightB}$ are non-degenerate and bounded matrices in $\R^d\otimes  \R^d$, i.e.  $ \exists C:=C(\H,T_0)$ s.t. for all $\xi \in S^{d-1},\ C^{-1}\le |\bar R_{s,t}^i \xi|\le C $.
%
\end{lemma}

\begin{proof}
We are going to prove the result by induction. Let us first consider the case $n=2$.
We have, for $i\in \{ 1,2\}$:
$$
\frac{d}{ds}R_{s,t}^{1,1} = a_s^{1,1}R_{s,t}^{1,1} +a_s^{1,2}R_{s,t}^{2,1},\ \frac{d}{ds}R_{s,t}^{2,1} = a_s^{2,1}R_{s,t}^{1,1} +a_s^{2,2}R_{s,t}^{2,1}.
$$
In order to obtain,  for $i\in \{1,2\}$, a semi-integrated representation of the entry $R_{s,t}^{i,1} $, we use the resolvent $\Gamma_{u,v}^i $ satisfying $\frac{d}{du}\Gamma_{u,v}^i = a_u^{i,i} \Gamma_{u,v}^i,\ \Gamma_{v,v}^i=I_{d\times d}$.
This yields:
\begin{eqnarray*}
R_{s,t}^{1,1}=\Gamma_{s,t}^1  + \int_t^s \Gamma_{s,u}^1 a_u^{1,2}R_{u,t}^{2,1}du,\
R_{s,t}^{2,1}=\int_t^s \Gamma_{s,u}^2 \biggl \{ a_u^{2,1}R_{u,t}^{1,1}\biggr \} du.
\end{eqnarray*}
Hence for all $0\le t\le s \le T $:
\begin{eqnarray*}
R_{s,t}^{1,1}&=&\Gamma_{s,t}^1  + \int_t^s \Gamma_{s,u}^1 a_u^{1,2}\biggl\{\int_t^u \Gamma_{u,v}^2 \biggl \{ a_v^{2,1}R_{v,t}^{1,1}\biggr \} dv\biggr\}
du,\\
|R_{s,t}^{1,1}|&\le&  C_T(1+\int_t^s |R_{v,t}^{1,1}|(s-t)dv) \le C_T,\ |R_{s,t}^{2,1}|\le C_T(s-t),
\end{eqnarray*}
using Gronwall's lemma for the last but one inequality. This in particular yields
\begin{eqnarray*}
R_{s,t}^{2,1}=\int_t^s \Gamma_{s,u}^2 a_u^{2,1}(\Gamma_{u,t}^1+O((u-t)^2))du.
\end{eqnarray*}
From the non-degeneracy of $a^{2,1}$ (H\"ormander like assumption \textbf{[H-3]}) and the resolvents on a compact set we derive that for $T$ small enough $R_{s,t}^{2,1}=(t-s)\bar R_{s,t}^2 $ where $\bar R_{s,t}^2 $ is non-degenerate and bounded. Rewriting 
$R_{s,t}^{1,1}= \Gamma_{s,t}^1+O((s-t)^2)$ we derive similarly that $R_{s,t}^{1,1}=\bar R_{s,t}^1 $, $\bar R_{s,t}^1 $ being non-degenerate and bounded. This proves \eqref{R1} for $n=2$. Let us now assume that \eqref{R1} holds for a given $n\ge 2$ and let us prove it for $n+1$.

We first need to introduce some notations to keep track of the induction hypothesis. To this end, we denote by $A_t^{n+1}:=A_t $ and $R_{s,t}^{n+1}:=R_{s,t} $ the matrices in $\R^{(n+1)d}\otimes \R^{(n+1)d} $ associated with the linear system $\frac{d}{ds} R_{s,t}=A_t R_{s,t},\ R_{t,t}=I_{(n+1)d\times (n+1)d} $. Observe now that:
\begin{eqnarray*}
A_t^{n+1}=\left(
\begin{array}{c|ccc}
a_t^{1,1} &\cdots & \cdots & a_t^{1,n+1}\\
\hline
a_t^{2,1}& & &\\
0 & & &\\
  \vdots & &A_t^n&\\
0 &&&\\
\end{array}
 \right),
\end{eqnarray*}
where $A_t^n$ is an $\R^{nd}\otimes \R^{nd} $ matrix satisfying \textbf{[H-3]}. Hence, denoting by $R_{s,t}^n$ the associated resolvent, i.e. $\frac{d}{ds}R_{s,t}^n=A_s^nR_{s,t}^n,\ R_{t,t}^n=I_{nd\times nd} $, $R_{s,t}^n $ satisfies \eqref{R1} from the induction hypothesis, so that
$$\forall i\in \leftB 1,n\rightB,\ \forall 0\le t\le s\le T,\ (R_{s,t}^n)^{i,1}=\frac{(s-t)^{i-1}}{(i-1)!}\bar R_{s,t}^{i,n},$$
where the $(\bar R_{s,t}^{i,n})_{i\in \leftB 1,n\rightB} $ are non-degenerate and bounded. Let us now observe that the differential dynamics of $(R_{s,t}^{n+1})^{2:n+1,1}:=\bigl( (R_{s,t}^{n+1})^{2,1},\cdots,(R_{s,t}^{n+1})^{n+1,1}\bigl)^* $ writes:
\begin{eqnarray*}
\frac{d}{d s } (R_{s,t}^{n+1})^{2:n+1,1}=A_s^n(R_{s,t}^{n+1})^{2:n+1,1}+G_{s,t}^{n+1},\ G_{s,t}^{n+1}:=\left(  a_s^{2,1}(R_{s,t}^{n+1})^{1,1}\ 
0_{n\times n}\
\cdots
\ 0_{n\times n}\right)^*,
\end{eqnarray*}
where 
\begin{equation}
\label{LIEN_PREM}
(R_{s,t}^{n+1})^{1,1}=\Gamma_{s,t}^{n+1,1}+\int_{t}^s \Gamma_{s,u}^{n+1,1}\left\{ \sum_{j=2}^{n+1} a_u^{1,j} (R_{u,t}^{n+1})^{j,1}\right\} du, 
\end{equation}
$\Gamma^{n+1,1} $ standing for the resolvent associated with $a^{1,1} $. Using now the resolvent $R_{s,t}^n $, the above equation can be integrated. We get:
\begin{eqnarray}
\label{RES_ECHELLE}
\bigl(R_{s,t}^{n+1}\bigr)^{2:n+1,1}=\int_t^s R_{s,u}^n G_{u,t}^{n+1}du.
\end{eqnarray}
From the above representation, using the induction assumption, \eqref{LIEN_PREM} and Gronwall's lemma we derive:
\begin{eqnarray*}
|(R_{s,t}^{n+1})^{n+1,1}|\le C_T\int_t^s \frac{(s-u)^{n-1}}{(n-1)!}\biggl\{1+\int_t^u \sum_{j=2}^{n} |(R_{v,t}^{n+1})^{j,1}|dv\biggr\} du.
\end{eqnarray*}
By induction one also derives for all $i\in \leftB 2,n+1\rightB $:
\begin{eqnarray*}
|(R_{s,t}^{n+1})^{i,1}|\le C_T\int_t^s \frac{(s-u)^{i-2}}{(i-2)!}\biggl\{1+\int_t^u \sum_{j=2}^{i-1} |(R_{v,t}^{n+1})^{j,1}|dv\biggr\} du,
\end{eqnarray*}
up to modifications of $C_T$ at each step. These controls yield that for all $i\in \leftB 2,n\rightB $, $0\le t\le s \le T $:
\begin{equation}
\label{OR_RESTES}
|(R_{s,t}^{n+1})^{i,1}|=O((s-t)^{i-1}).
\end{equation}
Now from \eqref{RES_ECHELLE}, \eqref{LIEN_PREM} and the induction assumption, we obtain, for all $i\in \leftB 2,n\rightB $, $0\le t\le s \le T $:
\begin{eqnarray*}
(R_{s,t}^{n+1})^{i,1}=\int_t^s\frac{(s-u)^{i-2}}{(i-2)!}\bar R_{s,u}^{i-1,n}a_u^{2,1}\{\Gamma_{u,t}^{n+1,1}+\int_t^u \Gamma_{u,v}^{n+1,1}\bigl\{ \sum_{j=2}^{n+1} a_v^{1,j}(R_{v,t}^{n+1})^{j,1} dv\bigr\}\} du.
\end{eqnarray*}
From the non degeneracy of $\bar R^{i-1,n},a^{2,1},\Gamma^{n+1,1} $ and \eqref{OR_RESTES}, we can conclude as for the case $n=2 $. 
\end{proof}
We can also mention some related analysis, emphasizing various specific time-scales, in Chaleyat-Maurel and Elie p. 255-279 in \cite{azen:etal:81}, Kolokoltsov \cite{kolo:mono:00} and \cite{D&M}. These procedures were  performed to derive small time asymptotics of the covariance matrix of, possibly perturbed,  Gaussian hypoelliptic diffusions.  


To conclude our analysis of the resolvent $R_{s,t}$, we give here a technical lemma that will be useful for the controls of Section \ref{SECTION_TECH}. 
\begin{lemma}[\textbf{Scaling Lemma}]\label{ScalingLemma} Under \textbf{[H-3]}, the resolvent $(R_{s,T})_{s\in [t,T]} $, for $0\le t < T$  associated with the linear system $\frac{d}{ds}R_{s,T}=A_sR_{s,T},\ R_{T,T}=I_{nd\times nd}$ can be written as
$$
R_{s,T}= \T_{T-t}^\alpha \hat{R}_{\frac{s-t}{T-t}}^{t,T} (\T_{T-t}^\alpha)^{-1},
$$
where $\hat{R}_{\frac{s-t}{T-t}}^{t,T}$  is non-degenerate and bounded uniformly on $s\in [t,T] $ with constants depending on $T$.
\end{lemma}
\begin{proof}
The proof of the above statement follows from the structure of the matrix $A_t$ (Assumption \textbf{[H-3]}), setting for all $u\in [0,1],\ \hat R_u^{t,T}:=(\T_{T-t}^{\alpha})^{-1}R_{t+u(T-t),T} \T _{T-t}^{\alpha}$ and differentiating:
\begin{eqnarray*}
\partial_u \hat R_u^{t,T}&=&(T-t)(\T_{T-t}^{\alpha})^{-1}A_{t+u(T-t)}R_{t+u(T-t),T} \T _{T-t}^{\alpha}\\
&=&\biggl((T-t)(\T_{T-t}^{\alpha})^{-1}A_{t+u(T-t)}\T_{T-t}^{\alpha}\biggr)\hat R_u^{t,T}:=A_u^{t,T} \hat R_u^{t,T}.
\end{eqnarray*}
\end{proof}

\begin{remark}
Let us observe that the scaling Lemma already gives the right orders for the entries $(R_{t,s}^{i,1})_{i\in \leftB 1,n\rightB} $ of the resolvent. However for the analysis of the Fourier transform of $\Lambda $, we explicitly need that those entries  write in the form of equation \eqref{R1}.
\end{remark}

\subsection{Estimates on the frozen density.}
\label{Estimates on the frozen density}


\subsubsection{Existence and first estimates.}

The main result of this section is the following.
\begin{proposition}[\textbf{Existence of the density}]
\label{LAB_EDFR}
Let $ T_0:=T_0(\H)$ be as in Lemma \ref{Form of the resolvent}. The process $(\Lambda_s)_{s \in[ t,t+T_0]},\ t\ge 0$, defined in \eqref{DEF_LAMBDA} has for all $s\in (t,t+T_0]$ a density $p_{\Lambda_s}$ given for all $z\in \R^{nd} $ by:
\begin{eqnarray*}
p_{\Lambda_s}(z) &=&  \frac{\det ( \M_{s-t}  )^{-1}  }{(2 \pi) ^{nd}}\int_{\R^{nd}}  e^{-i\langle q, ( \M_{s-t}  )^{-1}  z \rangle}
  \exp \left( -(s-t)\int_{\R^{nd}}\{1- \cos(\langle q,\xi \rangle)\} 
  \nu_S(d\xi) \right) dq,
\end{eqnarray*}
where $\nu_S:=\nu_S(t,T,s,\sigma)$ is a symmetric measure on $ S^{nd-1}$ s.t. uniformly in $s\in (t,t+T_0]$ for all $A\subset \R^{nd} $:
\begin{equation}
\label{DOMI_NU}
\nu_S(A)\le \int_{\R^+}\frac{d\rho}{\rho^{1+\alpha}}\int_{S^{nd-1}}\I_{A}(\rho\eta) g(c\rho)\bar \mu(d\eta),
\end{equation}
where $\bar \mu $ satisfies \textbf{[H-4]}
 and ${\rm dim}({\rm supp}(\bar \mu))=d $. As a consequence of this representation,
we get the following global (diagonal) estimate: 
\begin{equation}
\label{first-diag}
\exists  C:=C(\H,T_0),\ \forall s\in (t,t+T_0],\ \forall z\in \R^{nd},\ p_{\Lambda_s}(z)\le C\det (\T^{\alpha}_{s-t} )^{-1} .
\end{equation}
\end{proposition}


\begin{remark}
\label{REM_MS}
The previous result emphasizes that the process $(\Lambda_s)_{s\in[ t,t+T_0]}$ can actually be seen as a possibly tempered $\alpha$-stable symmetric process in dimension $nd$, with non-degenerate spectral measure, (left) multiplied by the intrinsic scale factor $(\M_{s-t})_{s\in[ t,t+T_0]} $ .
\end{remark}
\begin{proof}
The proof is divided into two steps:
\begin{trivlist}
\item[-]The first step is to compute the Fourier transform.
\end{trivlist}
Starting from the 
representation \eqref{DEF_LAMBDA}, we write the integral as a limit of its increments. Let $\tau_n:=\{(t_i)_{i\in \leftB 0,n\rightB}; t=t_0<t_1<\cdots<t_n=s \}$ be a subdivision of $[t,s]$, whose mesh $|\tau_n|:=\max_{i\in \leftB 0,n-1\rightB}|t_{i+1}-t_i| $ tends to zero when $n\rightarrow \infty$. Write now for all $p\in \R^{nd} $:
\begin{eqnarray*}
\langle p , \Lambda_s \rangle &=& \lim_{|\tau_n| \rightarrow 0} \sum_{i=0}^{n-1} \langle p , R_{s,t_i} B \sigma_{t_i}(Z_{t_{i+1}}-Z_{t_i}) \rangle
=
\lim_{|\tau_n| \rightarrow 0} \sum_{i=0}^{n-1} 
\langle  \sigma_{t_i}^*B^*R_{s,t_i}^* p , (Z_{t_{i+1}}-Z_{t_i}) \rangle.
\end{eqnarray*}
Since $Z$ has independent increments, we get from \eqref{GEN_TEMP} and the bounded convergence theorem that:
\begin{eqnarray}
\label{TF}
\forall p\in \R^{nd},\ \varphi_{\Lambda_s}(p):=\E(e^{i \langle p , \Lambda_s \rangle}) =  \exp\left(\int_t^s \int_{\R^d} \{ \cos( \langle p , R_{s,u}B\sigma_u z\rangle)-1\}g(|z|)\nu(dz) du\right).
\end{eqnarray}

\begin{trivlist}
\item[-]The second one is to prove its integrability.
\end{trivlist}
Setting $v = (s-u)/(s-t)$ and denoting $u(v):=s-v(s-t) $, the exponent in \eqref{TF} writes:
$$\int_t^s \int_{\R^d} \{ \cos( \langle p , R_{s,u}B\sigma_u z\rangle)-1\}g(|z|)\nu(dz) du=
(t-s) \int_0^1 \int_{\R^d}    \{ \cos(\langle p,R_{s,u(v)}^{\cdot,1}\sigma_{u(v)}z \rangle)-1\} g(|z|) \nu(dz)dv.
$$
Now, from Lemma \ref{Form of the resolvent}, we have the identity
$$
R_{s,u(v)}^{\cdot,1} = \M_{s-t} \bar R_v,
$$
setting with a slight abuse of notation 
$\bar R_v = \begin{pmatrix}\bar R_v^1\\ v \bar R_v^2 \\ \vdots \\ \frac{v^{n-1}}{(n-1)!}\bar R_v^n\end{pmatrix}$, where the $(\bar R_v^k)_{k\in \leftB 1,n\rightB} \in \R^d\otimes \R^d$ are non-degenerate and bounded.
The exponent in \eqref{TF} thus rewrites:
\begin{eqnarray}
\int_t^s \int_{\R^d} \{ \cos( \langle p , R_{s,u}B\sigma_u z\rangle)-1\}g(|z|)\nu(dz) du
&=& (t-s) \int_0^1 \int_{\R^d}    \{ \cos(\langle \M_{s-t}p,\bar R_{v}\sigma_{u(v)}z \rangle)-1\} g(|z|) \nu(dz)dv\nonumber \\
&=&(t-s) \int_0^1 \int_{\R^d}    \{ \cos(\langle \sigma_{u(v)}^* \bar R_{v}^* \M_{s-t}p,z \rangle)-1\} g(|z|) \nu(dz)dv.\nonumber\\
\label{PREAL_REP_DENS}
\end{eqnarray}
Observe now from \textbf{[H-4]} and \textbf{[T]} (recall that $g$ is $C^1$ for $\alpha\in (0,1) $ or $C^2$ for $\alpha\in [1,2) $, in a neighborhood of $0$) that:
\begin{eqnarray}
\int_t^s \int_{\R^d} \{ \cos( \langle p , R_{s,u}B\sigma_u z\rangle)-1\}g(|z|)\nu(dz) du=\nonumber\\
(t-s) \int_0^1 \int_{\R^d}    \{ \cos(\langle \sigma_{u(v)}^* \bar R_{v}^* \M_{s-t}p,z \rangle)-1\} g(0) \nu(dz)dv+\nonumber\\
(t-s) \int_0^1 \int_{\R^d}    \{ \cos(\langle \sigma_{u(v)}^* \bar R_{v}^* \M_{s-t}p,z \rangle)-1\} (g(|z|)-g(0)) \nu(dz)dv\nonumber\\
\le c(t-s)\{-\int_0^1 |\sigma_{u(v)}^* \bar R_v^* \M_{s-t} p |^\alpha dv+1\}\le c(t-s)\{-\int_0^1 |\bar R_v^* \M_{s-t} p |^\alpha dv+
1\},\ c:=c(\H),
\label{DOMI_STABLE_PROV}
\end{eqnarray}
using the uniform ellipticity of $\sigma$ in assumptions {\H} for the last inequality. In the above computations,  introducing $g(0)$ allows to exploit the explicit expression for the integral of the Fourier exponent of the stable L\'evy measure $\nu $ and to do Taylor expansions in a neighborhood of $0$ for the term $g(|z|)-g(0) $ thanks to the smoothness of $g$.
Now, the lower bound of the following lemma, whose proof is postponed to Subsection \ref{SUBSECPROOF},  gives that $\varphi_{\Lambda_s}\in L^1(\R^{nd})$ and therefore yields the existence of the density. 
\begin{lemma}\label{bound}
There exists a constant $C_{\ref{bound}}:=C_{\ref{bound}}(\H,T_0)>0$,
 such that for all $s\in [t,t+T_0] $:
\begin{equation}\label{lower bound}
 \int_0^1 | \bar R_v^* \M_{s-t} p |^\alpha dv \geq C_{\ref{bound}} |\M_{s-t} p|^\alpha.
\end{equation}
\end{lemma}
Since  $\varphi_{\Lambda_s}$ is integrable, we can write by \eqref{PREAL_REP_DENS} and Fourier inversion that for all $z\in \R^{nd} $:
\begin{eqnarray*}
 p_{\Lambda_s}(z) &=& \frac{1}{(2 \pi) ^{nd}}\int_{\R^{nd}} dp e^{-i\langle p,z \rangle} \exp \left(- (t-s) \int_0^1 \int_{\R^d}  \{1- \cos(\langle \M_{s-t}p,\bar R_{v}\sigma_{u(v)}z \rangle)\} g(|z|)  \nu(dz)dv\right)\nonumber\\
 &\le & \frac{1}{(2 \pi) ^{nd}}\int_{\R^{nd}} dp \exp(-c(t-s)\{C_{\ref{lower bound}}|\M_{s-t}p|^\alpha +1\}),
 \end{eqnarray*}
 using \eqref{DOMI_STABLE_PROV} and \eqref{lower bound} for the last inequality. This readily gives the global (diagonal) upper bound for the density.
Now, let us also write from \eqref{POLA_MES} and \eqref{GEN_TEMP}
 \begin{eqnarray}
 \label{density_lambda}
 p_{\Lambda_s}(z)&=& \frac{1}{(2 \pi) ^{nd}}\int_{\R^{nd}} dp e^{-i\langle p,z \rangle} \exp \left(- (t-s) \int_{\R^+} \frac{d \rho}{\rho^{1+\alpha}} \int_0^1 \int_{S^{d-1}}  \{1- \cos(\langle \M_{s-t}p,\bar R_{v}\sigma_{u(v)}\rho \varsigma \rangle)\} g(\rho)  \tilde \mu(d\varsigma)dv\right)\nonumber\\
&=& \frac{1}{(2 \pi) ^{nd}}\int_{\R^{nd}} dp e^{-i\langle p,z \rangle} \exp \left(- (t-s) \int_{\R^+} \frac{d \tilde \rho}{\tilde \rho^{1+\alpha}} \int_0^1 \int_{S^{d-1}}  \{1- \cos(\langle \M_{s-t}p,\frac{\bar R_{v}\sigma_{u(v)}\varsigma}{|\bar R_{v}\sigma_{u(v)}\varsigma|}\tilde \rho  \rangle)\}\right. \nonumber\\
&&\times \left. g\left(\frac{\tilde \rho}{|\bar R_{v}\sigma_{u(v)}\varsigma|} \right)  |\bar R_{v}\sigma_{u(v)}\varsigma|^\alpha\tilde \mu(d\varsigma)dv\right).
\end{eqnarray}
We now define the function 
$$
\begin{array}{cccc}
f: &[0,1] \times S^{d-1} &\longrightarrow &S^{nd-1}\\
&(v,\varsigma) &\longmapsto &\frac{\bar R_v \sigma_{u(v)} \varsigma}{|\bar R_v \sigma_{u(v)} \varsigma|},
\end{array}
$$
and on $[0,1]\times S^{d-1} $ the measure:
$$
m_{\alpha,\tilde \rho}(dv,d\varsigma) =  g\left(\frac{\tilde \rho}{|\bar R_{v}\sigma_{u(v)}\varsigma|} \right)  |\bar R_v \sigma_{u(v)} \varsigma|^\alpha \tilde \mu(d\varsigma)dv.
$$
The exponent in \eqref{density_lambda} thus rewrites: 
\begin{eqnarray*}
\int_0^1 \int_{\R^d}  \{1- \cos(\langle \M_{s-t}p,\bar R_{v}\sigma_{u(v)}z \rangle)\} g(|z|)  \nu(dz)dv
&=&
 \int_{\R^+} \frac{d \tilde \rho}{\tilde \rho^{1+\alpha}} \int_0^1 \int_{S^{d-1}}  \{1- \cos(\langle \M_{s-t}p,f(v,\varsigma)\tilde \rho\rangle )\} m_{\alpha,\tilde \rho}(dv,d\varsigma)\\
 &=&\int_{\R^+} \frac{d \tilde \rho}{\tilde \rho^{1+\alpha}}\int_{S^{nd-1}} \{1- \cos(\langle \M_{s-t}p,\eta \tilde \rho\rangle )\} \mu_{\tilde \rho}^*(d\eta), 
\end{eqnarray*}
denoting by $\mu_{\tilde \rho}^*$ the image measure of $m_{\alpha,\tilde \rho}$ by $f$ (which is a measure on $S^{nd-1}$).
Symmetrizing $\mu_{\tilde \rho}^*$ introducing
$\mu^*_{S,\tilde \rho}(A) = \frac{\mu_{\tilde \rho}^*(A)+\mu_{\tilde \rho}^*(-A)}{2}$, by parity of the cosine, we can write the exponent as:
$$
\int_0^1 \int_{\R^d}  \{1- \cos(\langle \M_{s-t}p,\bar R_{v}\sigma_{u(v)}z \rangle)\} g(|z|)  \nu(dz)dv
 =\int_{\R^+} \frac{d \tilde \rho}{\tilde \rho^{1+\alpha}}\int_{S^{nd-1}} \{1- \cos(\langle \M_{s-t}p,\eta \tilde \rho\rangle )\} \mu_{S,\tilde \rho}^*(d\eta).
$$
We eventually derive:
\begin{eqnarray}
p_{\Lambda_s}(z) &=& \frac{1}{(2 \pi) ^{nd}}\int_{\R^{nd}} dp e^{-i\langle p,z \rangle} \exp \left(- (t-s)\int_{\R^+} \frac{d \tilde \rho}{\tilde \rho^{1+\alpha}}\int_{S^{nd-1}} \{1- \cos(\langle \M_{s-t}p,\eta \tilde \rho\rangle )\} \mu_{S,\tilde \rho}^*(d\eta)\right)\nonumber\\
&=&\frac{1}{(2 \pi) ^{nd}\det(\M_{s-t})}\int_{\R^{nd}} dp e^{-i\langle p,\M_{s-t}^{-1}z \rangle} \exp \big((t-s)\int_{\R^{nd}} \{\cos(\langle p,\xi \rangle)-1\}\nu_{S}(d\xi)\big),\label{density_Lambda}
\end{eqnarray}
where $\nu_S $ is a symmetric measure on $\R^{nd}$.
Also, from \eqref{density_lambda} and Lemma \ref{bound} we get that there exists a symmetric bounded measure $\bar \mu $ on $S^{nd-1}$ and a constant $c>0$  s.t. for all $A\subset \R^{nd} $:
\begin{equation}
\label{DOMI_ALA_SZTONYK}
\nu_S(A)\le \int_{\R^+}\frac{d\rho}{\rho^{1+\alpha}}\int_{S^{nd-1}}\I_{A}(s\xi) g(c\rho)\bar \mu(d\xi),
\end{equation}
where $\bar \mu $ satisfies \textbf{[H-4]} and ${\rm dim}({\rm supp}(\bar \mu))=d  $, recalling for this last property that $\tilde \mu $ is absolutely continuous w.r.t. the Lebesgue measure of $S^{d-1}$. In the stable case, corresponding to $g=1 $ the equality holds in  \eqref{DOMI_ALA_SZTONYK}, and $\bar \mu $ is the spherical part of $ \nu_S $. In that case $\bar  \mu=\mu_{S,\tilde \rho}^*:=\mu_{S} $ since the measure $\mu_{S,\tilde \rho}^*$ introduced above would not depend on $\rho $. In the general case, the domination in \eqref{DOMI_ALA_SZTONYK} can be simply derived from the fact that in \eqref{density_lambda} one has $g\left(\frac{\tilde \rho}{|\bar R_{v}\sigma_{u(v)}\varsigma|} \right)\ge g\left(c \tilde \rho \right),\ \forall (v,\varsigma)\in [0,1]\times S^{d-1}  $.

\end{proof}

\subsubsection{Final derivation of the density bounds.}

\textit{Diagonal controls.} We first consider the case $|(\T_{s-t}^\alpha)^{-1}(R_{s,t}x-y)|\le K $. The upper-bound in \eqref{BORNE_COMPACTE} has already been proven. To obtain the lower-bound we perform computations rather similar to the ones in \cite{kolo:00} which are recalled in Appendix \ref{DIAG_LB}.\\

\noindent\textit{Off-diagonal controls.} We now consider the case $|(\T_{s-t}^\alpha)^{-1}(R_{s,t}x-y)|>K $.
We begin this paragraph recalling some results of Watanabe \cite{wata:07}. The striking and subtle thing with multi-dimensional stable
 processes is that their large scale asymptotics highly depend on the spectral measure. Namely, for a given symmetric spectral measure $\bar \mu $ on $S^{nd-1}$ satisfying \textbf{[H-4]}, implying that the associated symmetric stable process $(\bar S_t)_{t\ge 0}$ has a density on $\R^{nd} $ for $t>0$, the tail asymptotics of $\bar S_1$ can behave, when $|x|\rightarrow +\infty $,  as $p_{\bar S}(1,x)\asymp |x|^{-b} $ for $b\in [(1+\alpha),nd(1+\alpha)] $.
 Indeed, the behavior in $|x|^{-(1+\alpha)} $ would correspond to the decay of a scalar stable process and can appear if $\bar \mu=\sum_{i=1}^{nd} c_i(\delta_{e_i}+\delta_{-e_i})$, where the $(c_i)_{i\in \leftB 1,nd\rightB} $ are positive and $(e_i)_{i\in \leftB 1,nd\rightB} $ stand for the vectors of the canonical basis of $\R^{nd}$, when considering the asymptotics along \textbf{one} direction. On the other hand, the fastest possible decay of $|x|^{-nd(1+\alpha)} $ is also associated with this kind of spectral measure when investigating the large asymptotics for \textbf{all} the directions. Generally speaking, in the current framework, if $\bar \mu $ has support of dimension $k
 \in \leftB 0,nd-1\rightB$ the asymptotics of $\bar S_1 $ satisfy that there exists $\bar C\ge 1$ s.t.:
\begin{equation}
\label{ASYMP_WATA}
\frac{\bar C^{-1}}{|x|^{nd(1+\alpha)}}  \le p_{\bar S}(1,x)\le \frac{\bar C}{|x|^{k+1+\alpha}}.
 \end{equation}
We refer to Theorem 1.1 points \textit{i)} and \textit{iii)} in \cite{wata:07} for the proof of these results. 
The strategy to derive those bounds consists in carefully splitting the small and large jumps. This approach turns out to be very useful for us to investigate the kernel $H$  and is thoroughly exploited in Appendix \ref{controles_phi}.

From the representation \eqref{density} of the density of $\tilde X_s^{t,T,x,y} $ and \eqref{ASYMP_WATA} we readily get the indicated controls in the stable case. We refer to Appendix \ref{controles_phi} for a thorough discussion on the general case.

\subsubsection{Proof of Lemma \ref{bound}.}\label{SUBSECPROOF}
It is enough to show that there exists $C_{\ref{bound}}:=C_{\ref{bound}}(\H,T_0) $, s.t. for any $\theta \in S^{nd-1}$, 
$
 \int_0^1 | \bar R_v^* \theta |^\alpha dv \geq C_{\ref{bound}}
$.
We define  
$$\bar C:= \inf_{\theta \in S^{nd-1}} \int_0^1 |\bar R_v^* \theta |^\alpha dv.$$
By continuity of the involved functions and compactness of $S^{nd-1}$, the infimum is actually a minimum.
We need to show that this quantity is not zero. We proceed by contradiction. Assume that $\bar C=0$.
Then, there exists $\theta_0 \in S^{nd-1}$ such that for almost all $v \in [0,1]$, 
$| \bar R_v^* \theta_0 |=0$.
But since $\bar R_v^*$ is a continuous function in $v$, the previous statement holds for all $v \in[0,1]$, i.e.
$ \exists \theta_0 \in S^{nd-1}, \forall v \in[0,1],  |\bar R_v^* \theta_0 | =0$, or equivalently, that
$ \exists \theta_0 \in S^{nd-1}, \forall v \in[0,1],  \theta_0 \in Ker(\bar R_v^*) $.
Take now arbitrary $(v_i)_{i\in \leftB 1,n\rightB}$ in $[0,1]$. We have for each $i \in \leftB1,n \rightB $: 
$$
    \begin{pmatrix}
    (\bar R_{v_i}^1)^* & v_i (\bar R_{v_i}^2)^* & \cdots & \frac{v_i^n}{(n-1)!} (\bar R_{v_i}^n)^* \\
    \end{pmatrix}
        \begin{pmatrix}
   \theta_0^1\\  \vdots \\ \theta_0^n \\
    \end{pmatrix}
     = 0_{\R^d}.
$$
This 
equivalently  writes in matrix form:
$$
    \begin{pmatrix}
    (\bar R_{v_1}^1)^* & v_1(\bar  R_{v_1}^2)^* & \cdots & \frac{v_1^n}{(n-1)!} (\bar R_{v_1}^n)^* \\
    \vdots &\vdots&&\vdots\\
       (\bar  R_{v_n}^1)^* & v_n (\bar R_{v_n}^2)^* & \cdots & \frac{v_n^n}{(n-1)!} (\bar  R_{v_n}^n)^* \\
    \end{pmatrix}
        \begin{pmatrix}
   \theta_0^1\\  \vdots \\ \theta_0^n \\
    \end{pmatrix}
     = 0_{\R^{nd}}.
$$
Now,  taking $v_1 \rightarrow 0$ in the first line yields $(\bar R_{v_1}^1)^*\theta_0^1 = 0_{\R^d}$. Since the $(\bar R_v^i)_{i\in \leftB 1,n\rightB}$ are from Lemma \ref{Form of the resolvent} non degenerate, we have that $\theta_0^1 = 0_{\R^d}$.
Hence, the second line becomes:
$$
v_2 (\bar R_{v_2}^2)^*\theta_0^2 + \cdots + \frac{v_2^n}{(n-1)!} (\bar R_{v_2}^n)^*\theta_0^n  = 0_{\R^d}.
$$
Dividing by $v_2$, and taking $v_2 \rightarrow 0$, we get  $(\bar R_{v_2}^2)^*\theta_0^2= 0_{\R^d}$. Hence, $\theta_0^2 = 0_{\R^d}$.
By induction, we have that all components $\theta_0^i = 0_{\R^d}$, but this contradicts $\theta_0 \in S^{nd-1}$.
This yields $\bar C:=C_{\ref{bound}} >0$, which concludes the proof.\\
\phantom{BOUOU}\hfill $\square $

\begin{remark}
In the previous argument, the fact that the powers are increasing plays a key-role.
Indeed, we rely on the multi-scale property reflected by the scale matrix $\T^\alpha $. 
\end{remark}

\subsubsection{Proof of Lemma \ref{convergence_dirac}.}
Let us write:
\begin{eqnarray*}
 \int_{\R^{nd}} f(y) \tilde{p}_{\alpha}^{T,y}(t,T,x,y)dy -f(x) 
  &=&  \int_{\R^{nd}} f(y) \Big(\tilde{p}_{\alpha}^{T,y}(t,T,x,y) -  \tilde{p}_{\alpha}^{T,R_{T,t}x}(t,T,x,y)\Big)dy \\
  &&+  \int_{\R^{nd}} f(y) \Big(\tilde{p}_{\alpha}^{T,R_{T,t}x}(t,T,x,y) \Big)dy-f(x).
 \end{eqnarray*}
From Proposition \ref{FROZEN_DENSITY}, the second term tends to zero as $T$ tends to $t$.
Let us discuss the first term. Define:
\begin{equation}\label{conv}
\Delta= \int_{\R^{nd}} f(y) \Big(\tilde{p}_{\alpha}^{T,y}(t,T,x,y) -  \tilde{p}_{\alpha}^{T,R_{T,t}x}(t,T,x,y)\Big)dy.
\end{equation}
For a given threshold $K>0$ and a certain $\beta >0$ to be specified, we split $\R^{nd}$ into $D_1\cup D_2$ where:
$$
D_1=\{ y \in \R^{nd}; |(\T_{T-t}^\alpha)^{-1}(y-R_{T,t}x)| \leq K (T-t)^{-\beta} \},
$$
$$
D_2=\{ y \in \R^{nd}; |(\T_{T-t}^\alpha)^{-1}(y-R_{T,t}x)| > K (T-t)^{-\beta} \}.
$$
From Propositions \ref{EST_DENS_GEL}, \ref{EST_DENS_GEL_T}, the two densities in \eqref{conv} are upper-bounded by 
$
\dfrac{C\det(\T_{T-t}^\alpha)^{-1}}{K \vee  |(\T_{T-t}^\alpha)^{-1}(y-R_{T,t}x)|^{d+1+\alpha}}$. The idea is that on $D_2$ they are both in the \textit{off-diagonal} regime so that tail estimates can be used. On the other hand, we will explicitly exploit the compatibility between the spectral measures and the Fourier transform on $D_1$. 
Set for  $i\in \{1,2\},\ \Delta_{D_i}:=\int_{D_i} f(y) \Big(\tilde{p}_{\alpha}^{T,y}(t,T,x,y) -  \tilde{p}_{\alpha}^{T,R_{T,t}x}(t,T,x,y)\Big)dy $. We derive:
\begin{eqnarray*}
|\Delta_{D_2}|&\leq & C|f|_\infty \int_{D_2}  \frac{\det(\T_{T-t}^\alpha)^{-1}}{K \vee  |(\T_{T-t}^\alpha)^{-1}(y-R_{T,t}x)|^{d+1+\alpha}} dy = C|f|_\infty \int_{K (T-t)^{-\beta}}^{+\infty} dr \frac{r^{nd-1}}{K \vee r^{d+1+\alpha}}\\ 
&\leq &C (T-t)^{\beta((1-n)d+1+\alpha)}.
\end{eqnarray*}
Thus, for $\beta>0$, $\Delta_{D_2}\underset{T\downarrow t}{\longrightarrow}0 $.
On $D_1$, we will start from the inverse Fourier representation of $\tilde{p}_{\alpha}^{T,w}$ deriving from \eqref{density_lambda}, for $w=y$ or  $R_{T,t}x$.
Namely,
\begin{eqnarray*}
\tilde{p}_{\alpha}^{T,w}(t,T,x,y)=\frac{1}{\det(\M_{T-t})(2\pi)^{nd}} \int_{\R^{nd}} dp e^{-i \langle p ,\M_{T-t}^{-1} (y - R_{T,t}x) \rangle}
\exp\left( 
F_{T-t}(p,w)\right),
\end{eqnarray*}
where the Fourier exponent writes: 
$$\forall (p,w)\in (\R^{nd})^2,\ F_{T-t}(p,w)= -(T-t)\int_0^1\int_{\R^d} \{1-\cos(\langle p,\bar R_v\sigma(u(v),R_{u(v),T}w)z\rangle)g(|z|)\nu(dz) \}.
$$
We thus rewrite:
\begin{eqnarray*}
\Big(\tilde{p}_{\alpha}^{T,y}-  \tilde{p}_{\alpha}^{T,R_{T,t}x}\Big)(t,T,x,y)
=\frac{1}{\det(\M_{s-t})(2\pi)^{nd}} \int_{\R^{nd}}  dpe^{-i \langle p , \M_{s-t}^{-1}(y - R_{T,t}x) \rangle}\\
\int_0^1 d\lambda 
\Big(F_{T-t}(p,y)-F_{T-t}(p,R_{T,t}x)\Big)e^{(\lambda F_{T-t}(p,y) +(1-\lambda)F_{T-t}(p,R_{T,t}x))}.
\end{eqnarray*}
The key point is now to observe that from {\textbf{[H-2]}} the proof of Proposition \ref{LAB_EDFR} and the bound of Lemma \ref{bound}, we have:
\begin{eqnarray*}
\forall (p,w)\in (\R^{nd})^2,\ F_{T-t}(p,w)  
\leq C_{\ref{bound}}(T-t)(-| p|^\alpha+1). 
\end{eqnarray*}
Hence, $\exp(\lambda F_{T-t}(p,y) +(1-\lambda)F_{T-t}(p,R_{T,t}x))\leq \exp(C_{\ref{bound}}(T-t)\{-|  p|^\alpha+1\})$, independently on $\lambda\in [0,1]$.
Now, the smoothness of the tempering function $g$ in \textbf{[T]} yields:
\begin{eqnarray*}
|F_{T-t}(p,y) - F_{T-t}(p,R_{T,t}x)|  
 \\
\le (T-t)\int_0^1 \left|\int_{\R^{d}} \cos(\langle \sigma(u(v),R_{u(v),T}y)^*\bar R_v^* p,z\rangle)-\cos(\langle\sigma(u(v),R_{u(v),t}x)^* \bar R_v^* p,z\rangle)   g(|z|) \nu(dz) \right|dv\\
\le c (T-t)\big\{\int_0^1 \int_{S^{d-1}}\Big| | \langle p, \bar{R}_v \sigma(u(v),R_{u(v),T}y)\varsigma \rangle|^{\alpha} -  | \langle  p, \bar{R}_{v} \sigma(u(v),R_{u(v),t}x)\varsigma \rangle|^{\alpha} \Big|\mu(d\varsigma) dv
 +1\big\},
\end{eqnarray*}
using the notations of the proof of Proposition \ref{LAB_EDFR}.
On the other hand, since $\sigma$ is $\eta $-H\"older continuous in its second variable (see {\textbf{[H-1]}}), we have:
\begin{eqnarray*}
|F(p,y) - F(p,R_{T,t}x)|  
 \\
\le c(T-t)\{\int_0^1  |p|^\alpha |R_{u(v),T}y-R_{u(v),t}x|^{\eta(\alpha\wedge 1)}dv+1\}  
\leq C(T-t)\{| p|^\alpha |y-R_{T,t}x|^{\eta (\alpha\wedge 1)}+1\},
\end{eqnarray*}
using the Lipschitz property of the flow for the last inequality. 

To summarize, we get in all cases:
\begin{eqnarray*}
|\Delta_{D_1}| &\leq& |f|_\infty\int_{D_1}dy \left| \tilde{p}_{\alpha}^{T,y}(t,T,x,y) -  \tilde{p}_{\alpha}^{T,x}(t,T,x,y) \right| \\
&\leq& C  
\frac{1}{\det(\M_{T-t})}\int_{D_1}dy \int_{\R^{nd}}dp    (T-t)\{|  p|^\alpha |y-R_{T,t}x|^{\eta(\alpha\wedge 1)}+1\} e^{-C_{\ref{bound}} (T-t)|  p|^\alpha}.
\end{eqnarray*}
Changing variables, and integrating over $p$ yields
\begin{eqnarray*}
|\Delta_{D_1}| &\leq& \frac{C}{\det(\T_{T-t}^\alpha)}
 \int_{\{  |(\T_{T-t}^\alpha)^{-1}(y-R_{T,t}x)| \leq K (T-t)^{-\beta} \}}dy \{|y-R_{T,t}x|^{\eta(\alpha\wedge 1)}+(T-t)\}\\
&\leq&C \int_{\{  |Y| \leq K (T-t)^{-\beta} \}}dY \{|\T_{T-t}^\alpha Y|^{\eta(\alpha\wedge 1)}+(T-t)\} \leq C(T-t)^{\eta(\frac 1\alpha\wedge 1)  -\beta (nd+\eta(\alpha \wedge 1))}.
\end{eqnarray*}
Choosing now $\frac{\eta(\frac 1\alpha \wedge 1)}{nd+\eta(\alpha \wedge 1)} > \beta>0 $ gives that $|\Delta_{D_1}|\underset{T\downarrow t}{\longrightarrow} 0$, which concludes the proof. \hfill $\square $

\subsection{Estimates on the convolution kernel $H$.}
\label{Estimates on the kernel} 



In order to derive pointwise bounds on the kernel $H(t,T,x,y):=(L_t-\tilde L_t^{T,y})\tilde p_\alpha^{T,y}(t,T,x,y)$, it is convenient, since $\tilde p_\alpha^{T,y} $ is given in terms of Fourier inversion, to compute the symbols of the operators $ L_t,\tilde L_t^{T,y}$. Precisely, we denote by $l_t(p,x) $ (resp. $\tilde l_t^{T,y}(p,x) $) the functions of $(p,x)\in (\R^{nd})^2 $ s.t.
\begin{eqnarray*}
\forall \varphi \in C_0^2(\R^{nd}), \ \forall x\in \R^{nd},\ L_t\varphi(x)&=& \frac{1}{(2\pi)^{nd}}\int_{\R^{nd}}dp\exp(-i\langle p,x\rangle)l_t(p,x)\hat \varphi(p),\\
\tilde L_t^{T,y}\varphi(x)&=& \frac{1}{(2\pi)^{nd}}\int_{\R^{nd}}dp\exp(-i\langle p,x\rangle)\tilde l_t^{T,y}(p,x)\hat \varphi(p).
\end{eqnarray*} 
We refer to Jacob \cite{jacob} for further properties of the symbols associated with an integro-differential operator.
From usual properties of the (inverse) Fourier transform, 
we derive the following expressions.
\begin{lemma}\label{expression des symboles}
Let $(p,x)\in(\R^{nd})^2 $ be given. Recalling that 
 $B$ stands for the injection matrix of $\R^d$ into $\R^{nd}$, we have:
\begin{eqnarray*}
l_t(p,x) &=& \langle p,A_tx \rangle + \int_{\R^{d}} \{\cos( \langle p, B\sigma(t,x)z \rangle) -1 \}g(|z|) \nu(dz),
\\
\tilde{l}_t^{T,y}(p,x) &=& \langle p,A_t x \rangle + \int_{\R^{d}} \{\cos( \langle p, B\sigma(t,R_{t,T}y)z \rangle) -1 \}g(|z|) \nu(dz). \label{symbole tilde}
\end{eqnarray*}
\end{lemma}
From Lemma \ref{expression des symboles} we rewrite:
\begin{eqnarray*}
H(t,T,x,y) &=& \frac{1}{(2\pi)^{nd}} \int_{\R^{nd}} dp e^{-i\langle p, y- R_{T,t}x \rangle}
 \left\{ 
 \int_{\R^{d}} \{\cos( \langle p, B\sigma(t,x)z \rangle) -
 \cos( \langle p, B\sigma(t,R_{t,T}y)z \rangle)
  \}g(|z|) \nu(dz)
  \right\}\\
&&\times \exp \left( -\int_t^T du\int_{\R^d} \{1- \cos(\langle p, R_{T,u}^{1,\cdot}\sigma(u,R_{u,T}(y)\tilde z \rangle)\} g(|\tilde z|) \nu(d\tilde z) \right).
\end{eqnarray*}
\begin{remark}
Observe the interesting fact that since the drift is linear, it disappears in the difference of the generators.
\end{remark}
Let us now derive the diagonal bounds on the kernel, i.e. when $|(\T_{T-t}^\alpha)^{-1}(y-R_{T,t}x)|\le K $.
Observe first from the proof of Proposition \ref{LAB_EDFR} that we can write:
\begin{eqnarray*}
|H(t,T,x,y)|
\le C
\int_{\R^{nd}} dp 
 \left| \int_{\R^{d}} \{\cos( \langle p, B\sigma(t,x)z \rangle) -
 \cos( \langle p, B\sigma(t,R_{t,T}y)z \rangle)
  \}g(|z|) \nu(dz)
  \right|
 \exp\left( -c |\T_{T-t}^{\alpha} p|^\alpha  \right).
\end{eqnarray*}
Assume first that $\alpha \in (0,1)$. We then perform a first order Taylor expansion in the variable $z=\rho\varsigma$ associated with a radial cut-off at threshold $1/\{|p^1|\Delta \sigma(t,x,R_{t,T}y)\},\ \Delta \sigma(t,x,R_{t,T}y):= |\sigma(t,x)-\sigma(t,R_{t,T}y)| $. Recalling that $\sigma $ is $\eta $-H\"older continuous, we obtain:
\begin{eqnarray*}
|H(t,T,x,y)|\le C \int_{\R^{nd}} dp \left\{ \int_{|z|\le 1/\{|p^1|\Delta \sigma(t,x,R_{t,T}y)\}} |p^1| \Delta \sigma(t,x,R_{t,T}y) 
\rho \tilde \mu(d\varsigma)\frac{d\rho}{\rho^{1+\alpha}}\right.\\
+\left.2\int_{\rho>1/\{|p^1|\Delta \sigma(t,x,R_{t,T}y)\}}\frac{d\rho}{\rho^{1+\alpha}}\right\}
 \exp\left( -c |\T_{T-t}^{\alpha} p|^\alpha  \right)\\
 \le C\int_{\R^{nd}}dp |p^1|^\alpha \{\Delta \sigma(t,x,R_{t,T}y)\}^\alpha\exp\left(-c |\T_{T-t}^{\alpha} p|^\alpha\right)\\
 \le C\frac{\delta\wedge |x-R_{t,T}y|^{\alpha\eta}}{T-t}\int_{\R^{nd}}dp (T-t)|p^1|^{\alpha}\exp\left(-c |\T_{T-t}^{\alpha} p|^\alpha\right)\\
 \le C\frac{\delta\wedge |x-R_{t,T}y|^{\alpha\eta}}{T-t} \det(\T_{T-t}^\alpha)^{-1}=C\frac{\delta\wedge |x-R_{t,T}y|^{\alpha\eta}}{T-t}\bar p_\alpha(t,T,x,y).
\end{eqnarray*}
The case $\alpha\in (1,2) $ can be handled as above performing a Taylor expansion at order $2$ for the small jumps and $1$ for the large ones if   for the threshold $1/|p^1| $. The case $\alpha=1 $ is direct in the stable case and can be extended to the tempered one performing a first order Taylor expansion for the small jumps using the smoothness of $g$ around the origin.  

This gives 
the claim of Lemma \ref{CTR_KER_PTW} in the diagonal regime. The off-diagonal case is much more involved and leads to consider a quite tricky phenomenon of \textit{rediagonalization}. These aspects are considered in Appendix \ref{controles_phi}. 

\begin{remark}\label{drift borne}
We emphasize here that we could also consider an additional bounded drift term in the first $d $ components when $\alpha>1 $. Denoting this term by $b:\R^+\times\R^{nd}\rightarrow \R^d $, we could still use the previous frozen process as proxy. Exploiting the above symbol representation,  the additional term coming from the difference of the generators would write 
\begin{eqnarray*}
\langle b(t,x),\nabla_{x^1}\tilde p_\alpha(t,T,x,y)\rangle =\frac{1}{(2\pi)^{nd}} \int_{\R^{nd}} dp e^{-i\langle p, y- R_{T,t}x \rangle}
   \langle b(t,x),p^1\rangle \\
 \times\exp
\left( -\int_t^T du \int_{\R^{nd}} \{1-\cos(\langle p, R_{T,u}^{1,\cdot}\sigma(u,R_{u,T}y)z \rangle)\} g(|z|)\nu(dz) \right),
\end{eqnarray*} 
where $\nabla_{x^1} $ stands for the derivative w.r.t. to the first $d$ components. Observe that $|p^1|(T-t)^{1/\alpha} $ is homogeneous to the the contributions associated with $p^1$ in the exponential. This actually yields:
$$|\langle b(t,x),\nabla_{x^1}\tilde p_\alpha(t,T,x,y)\rangle|\le \frac{|b|_\infty}{(T-t)^{1/\alpha}}\bar p_\alpha (t,T,x,y), $$
on the diagonal which for $\alpha>1$ gives an integrable singularity in time. The off-diagonal case can be handled as in Appendix \ref{controles_phi}.
\end{remark}

\mysection{Controls of the convolutions.
}
\label{SECTION_TECH}
In this section we assume w.l.o.g. that $T\le T_0 = T_0(\H)\le 1$, as in Lemma \ref{Form of the resolvent}.
We first prove
Lemma \ref{lemme} that 
 emphasizes how the spatial contribution in the r.h.s. of \eqref{majoration H} yields, once integrated, a regularizing effect in time. 

\subsection{Proof of Lemma \ref{lemme}.}

We prove the first estimate only, the other one is obtained similarly.
Let us  naturally split the space according to the regimes of $\bar {p}_\alpha$ and $\breve p_\alpha $. With the notations of Proposition \ref{EST_DENS_GEL} we introduce the partition:
$$
D_1=\{ z \in \R^{nd} ; | (\T_{T-\tau}^\alpha)^{-1} (y-R_{T,\tau}z)| \leq K  \},\
D_2=\{ z \in \R^{nd} ; |(\T^{\alpha}_{T-\tau})^{-1} (y-R_{T,\tau}z)| > K  \}.
$$
On $D_1$, the diagonal control holds for $\bar {p}_\alpha+\breve{p}_\alpha$, that is, for $z\in D_1$ and recalling the definition of $\T_{T-\tau}^\alpha $ in Theorem \ref{MTHM}:
$$
(\bar {p}_\alpha+\breve{p}_\alpha)(\tau,T,z,y) \leq C_{\ref{EST_DENS_GEL}} \det ({  \T^{\alpha}_{T-\tau}  })^{-1}  =C_{\ref{EST_DENS_GEL}} (T-\tau)^{-d(\frac n \alpha+\frac{n(n-1)}2)}.
$$
On the other hand, denoting by $\| \cdot \| $ the matricial norm, we have from the scaling Lemma \ref{ScalingLemma}:
\begin{eqnarray*}
|z-R_{\tau,T}y|^{\eta(\alpha \wedge 1)} &\leq &
\|R_{\tau,T}\|^{\eta(\alpha \wedge 1)} \|\T_{T-\tau}^\alpha\|^{\eta(\alpha \wedge 1)} |(  \T^{\alpha}_{T-\tau}  )^{-1} (y- R_{T,\tau}z)|^{\eta(\alpha \wedge 1 )} 
\leq C (T-\tau)^{\eta(\frac 1 \alpha \wedge 1)},
\end{eqnarray*}
where the last inequality follows from the boundedness of the resolvent on compact sets and the definition of  $\T_{T-\tau}^\alpha $.
 
Besides, the Lebesgue measure of the set $D_1$ is bounded by $C\det (\T_{T-\tau}^\alpha)$, compensating exactly the time singularity appearing in the bound of $\tilde{p}_\alpha+\breve{p}_\alpha$.
In conclusion, we obtained on $D_1$:
$$
\int_{D_1} \delta \wedge |z-R_{\tau,T}y|^{\eta(\alpha \wedge 1)}(\bar{p}_\alpha+\breve{p}_\alpha)(\tau,T,z,y) dz \le  C(T-\tau)^{\eta(\frac{1}{\alpha}\wedge 1)}.
$$
Similarly, for $z\in D_2$, the off-diagonal bound holds for $\bar{p}_\alpha$ and $\breve p_\alpha $, i.e.:
\begin{eqnarray*}
(\bar{p}_\alpha+\breve{p}_\alpha)(\tau,T,z,y) \leq  C\bigg\{ \frac{\det( \T^{\alpha}_{T-\tau}  )^{-1} }{|(\T^{\alpha}_{T-\tau} )^{-1} (y-R_{T,\tau}z)|^{d+1+\alpha}}\\
+\frac{\I_{|(z-R_{\tau,T}y)^1|/(T-\tau)^{1/\alpha}| \asymp |\T_{(T-\tau)}^{-\alpha}(z-R_{\tau,T}y)|}}{(T-\tau)^{d/\alpha}(1+\frac{|(z-R_{\tau,T}y)^1|}{(T-\tau)^{1/\alpha}})^{d+\alpha}}\times \frac{1}{(T-\tau)^{\frac{(n-1)d}\alpha+\frac{n(n-1)d}2}(1+|(\T_{T-\tau}^{-\alpha}(z-R_{\tau,T}y)^{2:n}|)^{1+\alpha}} \bigg\}.
\end{eqnarray*}

From the scaling Lemma \ref{ScalingLemma} we derive $|z-R_{\tau,T}y|^{\eta(\alpha \wedge 1)} \le C|y-R_{T,\tau}z|^{\eta(\alpha \wedge 1)}\le C (T-\tau)^{\eta(\frac{1}{\alpha}\wedge 1)}|(\T^{\alpha}_{T-\tau})^{-1} (y-R_{T,\tau}z)|^{\eta(\alpha \wedge 1)}$. Hence setting $\xi:=|(\T^{\alpha}_{T-\tau})^{-1} (y-R_{T,\tau}z)|$ we first derive
\begin{eqnarray}
\label{INTERM_D2}
\int_{D_2} \delta \wedge |z-R_{\tau,T}y|^{\eta(\alpha \wedge 1)}\bar{p}_\alpha(\tau,T,z,y) dz\le C\int_{\xi>K}\bigl(\delta \wedge [ (T-\tau)^{\eta(\frac{1}{\alpha}\wedge 1)} \xi^{\eta(\alpha \wedge 1)} ]\bigr) \xi^{nd-1} \frac{d\xi}{\xi^{d+1+\alpha}}.
\end{eqnarray}
Now if $\beta:=(1-n)d+2+\alpha-\eta (\alpha\wedge 1)>1$, we directly get $\int_{\xi>K}\bigl(\delta \wedge[(T-\tau)^{\eta(\frac{1}{\alpha}\wedge 1)} \xi^{\eta(\alpha \wedge 1)}]\bigr) \frac{d\xi}{\xi^{(1-n)d+2+\alpha}}\le (T-\tau)^{\eta(\frac{1}\alpha \wedge 1)}\int_{\xi>K}\frac{d\xi}{\xi^{\beta
}}\le C (T-\tau)^{\eta(\frac{1}\alpha \wedge 1)}$.
When $\beta\le1$ 
we have to be more subtle.
We refine the partition introducing:
\begin{eqnarray*}
D_{2,1} = \{ \xi \in \R; K\leq \xi \leq K(T-\tau)^{-1/\alpha}\},\
D_{2,2} = \{ \xi \in \R; \xi > K(T-\tau)^{-1/\alpha}\}.
\end{eqnarray*}
On $D_{2,1}$, writing $\delta \wedge[(T-\tau)^{\eta(\frac{1}{\alpha}\wedge 1)}\xi^{\eta(\alpha \wedge 1)}]\leq  [(T-\tau)^{\eta(\frac{1}{\alpha}\wedge 1)}\xi^{\eta(\alpha \wedge 1)}]$ we get:
$(T-\tau)^{\eta(\frac{1}{\alpha}\wedge 1)}\int_{\xi \in D_{2,1}} d\xi \xi^{-\beta}\le C\{ (T-\tau)^{((1-n)d+1+\alpha)/\alpha}\I_{\beta<1}+(T-t)^{\eta(\frac{1}{\alpha}\wedge 1)}|\log(T-\tau)|\I_{\beta=1}\}.$
On $D_{2,2}$, using $\delta \wedge [(T-\tau)^{\eta(\frac 1\alpha \wedge 1)}\xi^{\eta(\alpha \wedge 1)}]\leq\delta$ we derive $\int_{\xi\in D_{2,2}}\frac{d\xi}{\xi^{(1-n)d+2+\alpha}}\le C_\delta(T-\tau)^{((1-n)d+1+\alpha)/\alpha} $.
Plugging the above controls in \eqref{INTERM_D2} yields the stated control. Let us now turn to:
\begin{eqnarray*}
\int_{D_2} \delta \wedge |z-R_{\tau,T}y|^{\eta(\alpha \wedge 1)}\breve{p}_\alpha(\tau,T,z,y) dz\le C\int_{|\zeta|>K}\bigl(\delta \wedge [ (T-\tau)^{\eta(\frac{1}{\alpha}\wedge 1)} |\zeta|^{\eta(\alpha \wedge 1)} ]\bigr) \frac{\I_{|\zeta^1|\asymp |\zeta|}}{(1+|\zeta^1|)^{d+\alpha}} \frac{d\zeta}{(1+|\zeta^{2:n}|)^{1+\alpha}},
\end{eqnarray*} 
where we have set $\zeta:=(\T^{\alpha}_{T-\tau})^{-1} (y-R_{T,\tau}z) $. We can now somehow \textit{tensorize} the two contributions. We obtain on the considered events:
\begin{eqnarray*}
\int_{D_2} \delta \wedge |z-R_{\tau,T}y|^{\eta(\alpha \wedge 1)}\breve{p}_\alpha(\tau,T,z,y) dz\le C\big\{\int_{|\zeta^1|>cK}\bigl(\delta \wedge [ (T-\tau)^{\eta(\frac{1}{\alpha}\wedge 1)} |\zeta^1|^{\eta(\alpha \wedge 1)} ]\bigr)  \frac{d\zeta^1}{|\zeta^1|^{d+\alpha}}\\
+\int_{|\zeta|>K}\bigl(\delta \wedge [ (T-\tau)^{\eta(\frac{1}{\alpha}\wedge 1)} |\zeta^{2:n}|^{\eta(\alpha \wedge 1)} ]\bigr) \frac{\I_{|\zeta^1|\asymp |\zeta|}}{(1+|\zeta^1|)^{d+\alpha}} \frac{d\zeta}{(1+|\zeta^{2:n}|)^{1+\alpha}}\big\}:=\breve{T}_1+\breve{T}_2.
\end{eqnarray*} 
For the term $\breve{T}_1 $, we directly have $\breve{T}_1\le C (T-\tau)^{\eta(\frac 1\alpha \wedge 1)} $ provided $\alpha>\eta(\alpha\wedge 1) $. Otherwise, i.e. the only possible case is $\alpha=\eta(\alpha\wedge 1) $, considering the partition $|\zeta^1|\in D_{2,1}\cup D_{2,2} $ as above replacing $K$ by $cK$, one can reproduce the previous arguments. Namely, on $D_{2,1}$, $(T-\tau)^{\eta(\frac 1\alpha \wedge 1)}\int_{D_{2,1}}r^{-(1+\alpha)+\eta(\alpha\wedge 1)}dr\le  C\{(T-\tau)^{\eta(\frac{1}{\alpha}\wedge 1)}|\log(T-\tau) |\}$. On the other hand, on $D_{2,2} $, $ \int_{D_{2,2}} (\delta \wedge [ (T-\tau)^{\eta(\frac{1}{\alpha}\wedge 1)} |\zeta^1|^{\eta(\alpha \wedge 1)}])\frac{d\zeta^1}{|\zeta^1|^{d+\alpha}}\le \delta\int_{r>(T-\tau)^{-1/\alpha}Kc}\frac{dr}{r^{1+\alpha}} \le C(T-\tau)$. For $\breve{T}_2 $, on $\{ |\zeta^{2:n}|\le K\} $ we directly get the estimate. Now, for  $\{ |\zeta|^{2:n}>K\} $ we get:
\begin{eqnarray*}
\int_{|\zeta^{2:n}|>K\cap |\zeta|>K}\bigl(\delta \wedge [ (T-\tau)^{\eta(\frac{1}{\alpha}\wedge 1)} |\zeta^{2:n}|^{\eta(\alpha \wedge 1)} ]\bigr) \frac{\I_{|\zeta^1|\asymp |\zeta|}}{(1+|\zeta^1|)^{d+\alpha}} \frac{d\zeta}{(1+|\zeta^{2:n}|)^{1+\alpha}}\big\}\\
\le (T-\tau)^{\eta(\frac 1\alpha \wedge 1)}\int_{|\zeta^1|>\bar c K}\frac{d\zeta^1}{(1+|\zeta^1|)^{d+\alpha}}\int_{ c |\zeta^1|\ge  |\zeta^{2:n}|\ge K}|\zeta^{2:n}|^{\eta(\alpha\wedge 1)}
\frac{d\zeta^{2:n}}{(1+|\zeta^{2:n}|)^{1+\alpha}}\\
\le C(T-\tau)^{\eta(\frac 1\alpha \wedge 1)}\int_{|\zeta^1|>\bar c K}\frac{d\zeta^1}{(1+|\zeta^1|)^{d+\alpha}} \{|\zeta^1|^{\eta(\alpha\wedge 1)+(n-1)d-1-\alpha}\I_{ \beta<1}+\log(|\zeta^1|)\I_{ \beta=1}\},
\end{eqnarray*}
for $\beta$ as above. Thus 
\begin{eqnarray*}
\breve T_2\le C(T-\tau)^{\eta(\frac 1\alpha \wedge 1)}\{ 1+\int_{r>\bar c K}dr \{r^{- (d(1-n)+2+2\alpha-\eta(\alpha\wedge 1))}\I_{ \beta<1}+r^{-(1+\alpha)}\log(r)\I_{ \beta=1}\}\le C(T-\tau)^{\eta(\frac 1\alpha \wedge 1)}, 
\end{eqnarray*}
using again the condition $d(1-n)+1+\alpha>0 $ for the last inequality.
The smoothing bounds of equations \eqref{lemme backward spec}, \eqref{lemme forward spec}
for $d=1,n=2 $ when the fast component is considered 
can be derived similarly.
\hfill $\square $

A useful extension of the previous result is the following lemma involving an additional logarithmic contribution which is \textit{explosive} in the off-diagonal regime. This anyhow does not affect \textit{much} the smoothing effect.

\begin{lemma}\label{lemme2}
There exists $C_{\ref{lemme2}}:=C_{\ref{lemme2}}(\H,T_0)>0 $ s.t. for all $T\in (0,T_0], (x,y)\in (\R^{nd})^2$, $\tau \in (t,T)$:
\begin{eqnarray*}
\int_{\R^{nd}} \log(K \vee |(\T_{T-\tau}^{\alpha})^{-1}(y-R_{T,\tau}z)|)\bigl\{ \delta \wedge |(z-R_{\tau,T}y)^2|^{\eta(\alpha \wedge 1)}\bigr\}(\bar{p}_\alpha+\breve{p}_\alpha) (\tau,T,z,y) dz \\
\le  C_{\ref{lemme2}} (T-\tau)^{(1+\frac1\alpha)\eta(\alpha
 \wedge1)},\\
\int_{\R^{nd}} \log(K \vee |(\T_{\tau-t}^{\alpha})^{-1}(z-R_{\tau,t}x)|) \bigl\{\delta \wedge [(\tau-t)|(z-R_{\tau,t}x)^1|+|(z-R_{\tau,t}x)^2|]^{\eta(\alpha\wedge 1)}\bigr\}
\bar{p}_{\alpha,\Theta}
(t,\tau,x,z) dz \\
\le  C_{\ref{lemme2}} (\tau-t)^{(1+\frac 1 \alpha)\eta(\alpha
 \wedge1)}. 
\end{eqnarray*}
\end{lemma}

\begin{proof}
The proof does not change much from the previous one.
Observe first that, from the supremum in the logarithm, the only difference arises for \textit{off-diagonal}  regimes, that is, for $z\in D_2$ referring to the partition in the previous proof.
The argument in the logarithm is however the same as the denominator of the \textit{off-diagonal estimate}. After changing variables to $\xi $ or $\zeta $ with the notations of the previous proof, it suffices to observe that  
 for any $\varepsilon\in (0,\alpha)$, there exists $C_\varepsilon>0$ s.t. for all $ r>K$:
$\log(K \vee r)\leq C_\varepsilon r^\varepsilon$. 
Taking $\varepsilon>0 $ s.t. $d(1-n)+1+\alpha-\varepsilon>0 $  allows to proceed as in the proof of Lemma \ref{lemme}.
\end{proof}

We now state a key lemma for our analysis. It gives a control for the first convolution between the frozen density $\tilde p_\alpha $ and the parametrix kernel $H $.
The result differs here from the expected one: we get an additional logarithmic factor, w.r.t. 
the bounds established for this quantity in \cite{D&M} for the Gaussian degenerate case, or \cite{kolo:00} for the stable non-degenerate case, as well as another contribution coming from the \textit{rediagonalization} phenomenon.




 \begin{lemma}[\textbf{First Step Convolution}.]\label{premier coup} Assume $d=1,n=2$.
There exist $C_{\ref{premier coup}}:=C_{\ref{premier coup}}(\H)>0,\ \omega:=\omega(\H)\in (0,1]$ s.t.
for all $T\in (0,T_0],\ T_0:=T_0(\H)\le 1, (x,y)\in (\R^{nd})^2$, $t\in [0,T)$,
 \begin{eqnarray*}
|\tilde p_\alpha\otimes H |(t,T,x,y) \le C_{\ref{premier coup}}  \bigg( \bar p_\alpha (t,T,x,y)
\Big ((T-t)^\omega+\\
 \delta \wedge |x-R_{t,T}y|^{\eta(\alpha \wedge 1)}(1 
+\log(K \vee |(\T_{T-t}^\alpha)^{-1}(y-R_{T,t}x)|))\Big) 
+[\delta \wedge |x-R_{t,T}y|^{\eta(\alpha \wedge 1)}] \check p_\alpha (t,T,x,y) \bigg), 
\end{eqnarray*}
with $\check p $ as in \eqref{CHECK_P}.
Suppose now that 
\textbf{[HT]} holds, that $\sigma(t,x)=\sigma(t,x^2) $ and $ \eta>1/[(\alpha \wedge 1)(1+\alpha)]$. We can then improve the previous bound and derive:
%
\begin{eqnarray}\label{pxH}
|\tilde p_\alpha\otimes H|(t,T,x,y) 
&\le& C_{\ref{premier coup}} \Big( (T-t)^\omega \bar{p}_{\alpha,\Theta}(t,T,x,y)  + \bar{q}_{\alpha,\Theta}(t,T,x,y) \Big) ,\label{SPATIAL_CONV}
\end{eqnarray}
where we denote:
\begin{eqnarray*}
\bar{q}_{\alpha,\Theta}(t,T,x,y) &=& \delta \wedge \{(T-t)|(x-R_{t,T}y)^1|+|(x-R_{t,T}y)^2|\}^{\eta(\alpha \wedge 1)} \\
&&\times \big[\bar{p}_{\alpha,\Theta}(t,T,x,y)
\left(  1 + \log \Big[K \vee |(\T_{T-t}^\alpha)^{-1}(y-R_{T,t}x)| \Big]  \right)
\big].
\end{eqnarray*}
\end{lemma}
\begin{remark}
The first part of the Lemma gives the bound of Lemma \ref{CTRL_PRELIM_NON_FONCTIONNEL}. Let us emphasize
that this bound is not sufficient to derive the convergence of the parametrix series \eqref{parametrix}. 
The difficulty comes from the term in $\check p $ deriving from the rediagonalization phenomenon that induces a possible loss of concentration in the stable case and also prevents from a regularizing property in the tempered one if $\sigma $ depends on both variables. Namely, the additional time singularity in $\check p $ can be compensated if $\sigma $ only depends on the fast variable, which gives a higher order smoothing effect, but does not seem to be easily handleable in the general setting.
The control \eqref{pxH} is actually sufficient to imply the convergence of the parametrix series when $d=1, n=2,\ \sigma(t,x)=\sigma(t,x^2)$ under the indicated condition on $\eta $. It gives the first statement in Lemma \ref{LEMME_IT_KER}.
\end{remark}
\begin{proof}
 
 
To perform the analysis, we first bound $H$ using \eqref{majoration H}. We thus obtain:
\begin{equation}
\label{BD_PREAL}
|\tilde p_\alpha\otimes H|(t,T,x,y)\le C\int_t^T d\tau \int_{\R^{nd}}\bar p_\alpha(t,\tau,x,z)\frac{\delta \wedge |z-R_{\tau,T}y|^{\eta(\alpha \wedge 1)}}{T-\tau}(\bar p_\alpha+\breve p_\alpha)(\tau,T,z,y)dz.
\end{equation}
For the proof it will be convenient to split the time interval $[t,T] $ into two subintervals $I_1:=[t,\frac{t+T}2],I_2:=[\frac{t+T}2,T] $.
We observe that for  $\tau \in I_1$, $T-\tau \asymp T-t $ whereas for $\tau\in I_2,\ \tau-t\asymp T-t $.\\

The leading idea for the proof is to partition the space in order to say that one of the 
densities involved in \eqref{BD_PREAL} is homogeneous to the \textit{global one} $\bar p_\alpha(t,T,x,y) $, and to get some \textit{regularization} from the other contribution, using thoroughly Lemma \ref{lemme}.\\
\\
\textbf{Diagonal Estimates.}
When the global diagonal regime holds, i.e. $|(\T_{T-t}^{\alpha})^{-1}(R_{T,t}x-y)|\le K $, we will prove the following global diagonal estimate:
\begin{equation}
\label{EST_CONV_DIAG}
|\tilde p_\alpha \otimes H|(t,T,x,y) 
\le C \Big( (T-t)^\omega+\delta \wedge |x-R_{t,T}y|^{\eta(\alpha \wedge 1)} \Big)\bar p_\alpha(t,T,x,y).
\end{equation} 
Indeed, on $I_1$,  if $|(  \T^{\alpha}_{T-\tau}  )^{-1} (y-R_{T,\tau}z) | \leq K$, from Proposition \ref{EST_DENS_GEL} the diagonal estimate holds for $\bar{p}_\alpha(\tau,T,z,y)$. Since $T-\tau \asymp T-t$, we have:
$$\bar{p}_\alpha(\tau,T,z,y) \leq C\det(\T_{T-\tau}^\alpha)^{-1}\le C \det(\T_{T-t}^\alpha)^{-1}\le C\bar p_\alpha(t,T,x,y).$$
On the other hand, if  $|(  \T^{\alpha}_{T-\tau}  )^{-1} (y-R_{T,\tau}z) | > K$, the off-diagonal expansion holds for $\bar{p}_\alpha(\tau,T,z,y)$ and from Proposition \ref{EST_DENS_GEL}:
$$
\bar{p}_\alpha(\tau,T,z,y) \le C \frac{\det(\T^{\alpha}_{T-\tau} )^{-1} }{|(\T_{T-\tau}^\alpha)^{-1}(y-R_{T,\tau}z)| ^{d+1+\alpha}}\leq C \det( \T^{\alpha}_{T-\tau}  )^{-1}   \leq C \det (\T^{\alpha}_{T-t}  )^{-1}  \leq C \bar{p}_\alpha(t,T,x,y)\footnote{Observe that we could have used here that the diagonal control is a global bound. We introduced the dichotomy on the regime to emphasize that it is a crucial argument in this section.}.
$$
Additionally, the boundedness of the resolvent yields:
\begin{equation}\label{InegTriangRes}
|z- R_{\tau,T}y| \leq |z- R_{\tau,t}x|+ |R_{\tau,t}x- R_{\tau,T}y|\leq C \Big( |z- R_{\tau,t}x|+ |x- R_{t,T}y| \Big).
\end{equation}
On the other hand, on $I_1$:
\begin{equation}
\label{DIAG_BREVE_P}
\breve p(\tau,T,z,y)\le C\det(\T_{T-t}^\alpha)^{-1}\le C\bar p_\alpha(t,T,x,y).
\end{equation}
Denoting by  $\otimes_{|I_1}$ the time-space convolution, where the time parameter is restricted to the interval $I_1$, we have from  \eqref{BD_PREAL}, \eqref{InegTriangRes}, \eqref{DIAG_BREVE_P} and Lemma \ref{lemme}:
\begin{eqnarray}
|\tilde p_\alpha \otimes_{|I_1} H|(t,T,x,y)&\le&  C\bar{p}_\alpha(t,T,x,y) \int_{I_1}d\tau \int_{\R^{nd}} \bar{p}_\alpha(t,\tau,x,z)
\left( \frac{\delta \wedge |z- R_{\tau,t}x|^{\eta(\alpha \wedge 1)}}{\tau-t}\right.\nonumber\\
&&\left.+ \frac{\delta \wedge |x- R_{t,T}y|^{\eta(\alpha \wedge 1)}}{T-t}+1 \right)dz\nonumber \\
 &\le&  C \bar{p}_\alpha(t,T,x,y) \int_{I_1}d\tau \left( (\tau-t)^{\omega-1} +
\frac{\delta \wedge |x-R_{t,T}y|^{\eta(\alpha\wedge 1)}}{T-t} +1\right)\nonumber\\
 &\le&  C \bar{p}_\alpha(t,T,x,y) ( (T-t)^\omega + \delta \wedge |x-R_{t,T}y|^{\eta(\alpha \wedge 1)}).\label{INEG_TRIANG_TEST}
\end{eqnarray}

Now, when $\tau \in I_2$, we have
$\bar{p}_\alpha(t,\tau,x,z) \leq \bar{p}_\alpha(t,T,x,y)$, so that from Lemma \ref{lemme}:
\begin{eqnarray*}
|\tilde p_\alpha \otimes_{|I_2} H|(t,T,x,y)&\le&  C\bar{p}_\alpha(t,T,x,y) \int_{I_2} d\tau \int_{\R^{nd}} \frac{\delta \wedge |z-R_{\tau,T}y|^{\eta(\alpha \wedge 1)}}{T-\tau}(\bar{p}_\alpha+\breve p_\alpha)(\tau,T,z,y) dz\\
 &\le&  C \bar{p}_\alpha(t,T,x,y) \int_{I_2} d\tau(T-\tau)^{\omega-1} \le C (T-t)^\omega \bar{p}_\alpha(t,T,x,y).
\end{eqnarray*}

\textbf{Off-Diagonal Estimates.}
We consider here the case $|  (\T^{\alpha}_{T-t}  )^{-1} (y-R_{T,t}x) | \geq K$. Since we will need in the proof to exploit the \textit{semigroup} property of Corollary \ref{SG_PROP} we restrict for the off-diagonal estimates to the case $d=1,n=2 $.\\

\textit{Contributions involving $\bar p_\alpha(t,T,x,y) $.} 

We first consider the contributions involving  $\bar{p}_\alpha(t,T,x,y)$ which is in the \textit{off-diagonal} regime.
In our current degenerate setting, several scales are involved in the term  $|( \T^{\alpha}_{T-t}  )^{-1} (y-R_{T,t}x) |$. The \textit{slow} time scale,  associated with the first component of the process, induces in the off-diagonal regime additional time singularities in the density w.r.t. to the non-degenerate case. We thus need to be very careful when comparing the two contributions in $\bar p_\alpha $ appearing  in the convolution $\tilde p_\alpha \otimes H$. 
Observe anyhow from the scaling Lemma \ref{ScalingLemma} that:
\begin{eqnarray}
|(\T^\alpha_{T-t})^{-1}(y-R_{T,t}x)| &\leq&  |(\T^\alpha_{T-t})^{-1}(y-R_{T,\tau}z)| + |(\T^\alpha_{T-t})^{-1}(\T_{T-t}^\alpha \hat R_{\frac{\tau-t}{T-t}}^{t,T}(\T_{T-t}^{\alpha})^{-1}\{z-R_{\tau,t}x\})| \nonumber \\
&\le &  |(\T^\alpha_{T-t})^{-1}(y-R_{T,\tau}z)|+C |(\T_{T-t}^{\alpha})^{-1}(z-R_{\tau,t}x)|\nonumber \\
&\le &  |(\T^\alpha_{T-\tau})^{-1}(y-R_{T,\tau}z)|+C |(\T_{\tau-t}^{\alpha})^{-1}(z-R_{\tau,t}x)|,\ C:=C(\H,T_0) . \label{INEG_TRIANG_MS}
\end{eqnarray}
Hence, at least one of the two densities involved in the convolution is off-diagonal. As emphasized below, the main difficulty w.r.t.  the non degenerate case consists in suitably controlling the multi-scale effects that prevent from  handling directly the time singularity of $H$ in the convolution $\tilde p_\alpha \otimes H $, see e.g. Proposition 3.2  in  Kolokoltsov \cite{kolo:00}. 
Assume now that the component number $k\in \{1,2\}$ dominates in  $\bar p_\alpha(t,T,x,y) $ when considering the flow at the current time $ \tau$ of the convolution, the off-diagonal estimate becomes:

\begin{eqnarray*}
\bar{p}_\alpha(t,T,x,y) 
&\le& C \frac{({\rm det}(\T_{T-t}^\alpha))^{-1}}{|(\T_{T-t}^{\alpha})^{-1}(R_{\tau,t}x-R_{\tau,T}y)|^{2+\alpha}} \le C\frac{(T-t)^{-\zeta(k)} }{|R^k_{\tau,t}x-R^k_{\tau,T}y|^{2+\alpha}},\\
\zeta(k)&=&(\frac {2}\alpha+1)-((k-1)+\frac{1}{\alpha})(2+\alpha).
\end{eqnarray*}
According to the sign of the power of $T-t$, two cases arise. 
Set for $k\in \{1,2\},\ \gamma(k):=\zeta(k)-1 $. For the second, or \textit{fast}, component,  the exponent $\gamma(2)=  \zeta(2)-1=1+\alpha$ is non negative. For the first, \textit{slow} component $\gamma(1)=  -1$. This  is the aforementioned \textit{slow/fast} dichotomy.
 
 \begin{trivlist}
\item[-]When the fast component dominates, as the off-diagonal estimates are not singular in time anymore, no major problem arises. We refine \eqref{INEG_TRIANG_MS} in the following sense:
$$
K(T-t)^{1+\frac{1}{\alpha}} \leq |R^2_{\tau,T}y-R^2_{\tau,t}x| \leq  |R^2_{\tau,T}y-z^2| + |z^2-R^2_{\tau,t}x|.
$$
Thus, at least one of the two densities in \eqref{BD_PREAL} is off-diagonal through a fast component.
On the one hand, if $ 1/2 |R^2_{\tau,T}y-R^2_{\tau,t}x| \leq  |z^2-R^2_{\tau,t}x|$,
\begin{eqnarray*}
\bar{p}_\alpha (t,\tau, x,z) &\le& C \frac{\det(\T_{\tau-t}^\alpha)^{-1}}{| (\T_{\tau-t}^\alpha)^{-1}(z- R_{\tau,t}x)|^{2+\alpha}} \leq C\frac{(\tau-t)^{\gamma(2)+1}}{|z^k-R^k_{\tau,t}x|^{2+\alpha}} \\
&\leq& C
\frac{(T-t)^{\gamma(2) +1}}{|R^2_{\tau,t}x-R^2_{\tau,T}y|^{2+\alpha}}.
\end{eqnarray*}
On the other hand, if $ 1/2 |R^2_{\tau,T}y-R^2_{\tau,t}x| \leq |z^2-R^2_{\tau,T}y| $,
$$
\frac{1}{T-\tau}\bar {p}_\alpha (\tau,T, z,y) \leq C\frac{(T-\tau)^{\gamma(2)}}{|z^2-R^2_{\tau,T}y|^{2+\alpha}} \leq 
\frac{C}{T-t}\frac{(T-t)^{\gamma(2) +1}}{|R^2_{\tau,T}y-R^2_{\tau,t}x|^{2+\alpha}}.
$$
 In both cases, we are in position to apply Lemma \ref{lemme}, directly in the first case, similarly to \eqref{INEG_TRIANG_TEST} in the second one. The proof is then the same as in Kolokoltsov \cite{kolo:00}. 
Observe that in the second case, we have compensated the singularity associated with the contribution 
$\bar p_\alpha $ in the kernel $H$, independently of the position of the time parameter $\tau$. 

\item[-]We now focus on the second case, that is when the slow component dominates so that $\gamma(1)$ is negative.
We consider the partition $[t,T]= I_1\cup I_2$ and  start with $\tau \in I_2$. In this case, we have $T-t \asymp \tau-t$.
In other words, this is the case where the singularity induced by the kernel $H$ is the worst.
\end{trivlist}
We split  $\R^{2}$ into 
\begin{eqnarray}
D_1:= \{z \in \R^{2} ; (T-\tau)^\beta |(\T_{T-t}^\alpha)^{-1}( y-R_{T,t}x)| \leq 
 |(\T_{\tau-t}^\alpha)^{-1}(z-R_{\tau,t}x)| \},\nonumber\\
D_2:= \{z \in \R^{2} ; (T-\tau)^\beta |(\T_{T-t}^\alpha)^{-1}(y-R_{T,t}x)| >
  |(\T_{\tau-t}^\alpha)^{-1}( z-R_{\tau,t}x)| \},\label{SPATIAL_PART}
\end{eqnarray}
for a parameter $ \beta>0$ to be specified later on.
We define accordingly, for $i\in \{1,2\}$:
\begin{eqnarray}
\label{I2DI}
\bar A_{\alpha,I_2,D_i}	(t,T,x,y)  :=  \int_{I_2}d\tau \int_{D_i}\bar{p}_\alpha(t,\tau,x,z) \frac{\delta \wedge |z-R_{\tau,T}y|^{\eta(\alpha \wedge 1)}}{T-\tau} \bar{p}_\alpha(\tau,T,z,y)dz.
\end{eqnarray}
Let us first deal with $z\in D_1$. Since $\tau \in I_2$, we have:
$$
\bar{p}_\alpha(t,\tau,x,z)\leq C\frac{\det(\T^\alpha_{\tau-t})^{-1} }{ |(\T_{\tau-t}^\alpha)^{-1}(z-R_{\tau,t}x)|^{2+\alpha} } \leq C\frac{\det(\T_{T-t}^\alpha)^{-1}}{ (T-\tau)^{\beta(2+\alpha)}|(\T_{T-t}^\alpha)^{-1}(y-R_{T,t}x)|^{2+\alpha}}.
$$
Hence, as we did in the first part of the proof, we take out $\bar{p}_\alpha(t,\tau,x,z)$ off the integral \eqref{I2DI}. This is done here up to the additional singular coefficient $(T-\tau)^{-\beta(2+\alpha)}$.
Still from Lemma \ref{lemme}, we get:
$$
\bar A_{\alpha,I_2,D_1} (t,T,x,y) \leq C\bar{p}_\alpha(t,T,x,y) \int_{I_2} d\tau(T-\tau)^{\omega - \beta(2+\alpha)-1}.
$$
Then, in order to get an integrable bound, we must take:
\begin{equation}\label{def_beta}
0<\beta < \frac{\omega}{2+\alpha}.
\end{equation}

On $D_2$, we have to be more subtle. From the previous partition, the idea is to say that if $\tau \in [\tau_0,T]$ for $\tau_0 $ \textit{close enough} to $T$, then the diagonal bound holds for the first density on $D_2$. In such cases we manage to get the global expected bound in the convolution. However, the previous $\tau_0 $
will highly depend on the global off-diagonal estimate $|(\T_{T-t}^{\alpha})^{-1}(R_{T,t}x-y)| $, and for  $\tau \in I_2,\ \tau\le \tau_0 $, we did not succeed to do better than integrating the singularity in $(T-\tau)^{-1} $ yielding the logarithmic contribution.
%
\begin{trivlist}
\item[$\bullet$] Let us fix $\delta_0\in (0,K) $. Observe that for  fixed  $ (t,T,x,y)$, if $\tau\ge  \tau_0 := T-\left( \frac{\delta_0}{|(\T^\alpha_{T-t})^{-1}(y-R_{T,t}x)|}\right)^{\frac{1}{\beta}}$ then $\delta_0 \geq (T-\tau)^\beta |(\T^\alpha_{T-t})^{-1}(y-R_{T,t}x)|$.
Then, since $z\in D_2$, we have $\delta_0 \geq |(\T^\alpha_{\tau-t})^{-1}(z-R_{\tau,t}x)|$, and the diagonal estimate holds for $\bar{p}_\alpha(t,\tau,x,z)$. We write:
\begin{eqnarray*}
\bar A_{\alpha ,I_2\cap\{\tau\ge \tau_0\}, D_2} (t,T,x,y) &\leq& C\int_{I_2\cap \{\tau\ge \tau_0 \}}\hspace*{-.5cm} d\tau \det(\T_{\tau-t}^\alpha)^{-1}  \int_{D_2}  \frac{\delta \wedge |z-R_{\tau,T}y|^{\eta(\alpha \wedge 1)} }{T-\tau}\bar{p}_\alpha(\tau,T,z,y)dz \\
&\overset{Lemma \ \ref{lemme}}{\le}& C \int_{I_2\cap \{\tau\ge \tau_0 \}}\hspace*{-.5cm} d\tau\det(\T_{\tau-t}^{\alpha})^{-1} (T-\tau)^{\omega -1 } .
\end{eqnarray*}
Now $\delta_0^{2+\alpha} \geq (T-\tau)^{\beta(2+\alpha)}  |(\T^\alpha_{T-t})^{-1}(y-R_{T,t}x)|^{2+\alpha}$, so that:
$$
\bar A_{\alpha,I_2\cap\{ \tau\ge \tau_0\}, D_2} (t,T,x,y) \leq  \int_{I_2} d\tau\det(\T_{T-t}^{\alpha})^{-1} (T-\tau)^{\omega-\beta(2+\alpha) -1}  \frac{\delta_0^{2+\alpha} }{  |(\T^\alpha_{T-t})^{-1}(y-R_{T,t}x)|^{2+\alpha}}.
$$
Thus, as long as $\beta$ satisfies \eqref{def_beta}, $\bar A_{\alpha,I_2\cap\{\tau\ge \tau_0 \},D_2} (t,T,x,y) \leq (T-t)^{\bar \omega}\bar{p}_\alpha(t,T,x,y),\ \bar \omega:=\omega-\beta(2+\alpha).$

\item[$\bullet$] Assume now that $\tau<  \tau_0 = T-\left( \frac{\delta_0}{|(\T^\alpha_{T-t})^{-1}(y-R_{T,t}x)|}\right)^{\frac{1}{\beta}}$.
The singularity induced by $H$ is then integrable, and yields the logarithmic contribution.
Specifically:
\begin{eqnarray*}
\bar A_{\alpha,I_2\cap\{\tau<\tau_0 \},D_2} (t,T,x,y) &\leq&  C\int_{I_2} d\tau  \ind_{\tau \leq \tau_0} \int_{D_2}\bar p_\alpha(t,\tau,x,z)\frac{\delta \wedge |z-R_{\tau,T}(y)|^{\eta(\alpha \wedge 1)  }}{T-\tau}\bar p_\alpha(\tau,T,z,y)dz.\\
\end{eqnarray*}


Now, the key-point to get a smoothing effect is to keep the $\delta \wedge |x-R_{t,T}y|^{\eta(\alpha \wedge 1)}$ part in the control of the convolution. In order to keep track of this term, we need to determine which component dominates in $|x-R_{t,T}y|$. This can be rather intricate
in the multi-scale setting.
In the case $n=2$, the only \textit{slow} component is the first one.
Saying that it dominates at a given integration time $\tau$ is asking:
\begin{equation}
\label{THE_FIRST_COM_TAU}
|R^2_{\tau,T}y-R^2_{\tau,t}x| \leq  (T-t) |R^1_{\tau,T}y-R^1_{\tau,t}x|.
\end{equation}
Furthermore, we can write:
\begin{eqnarray*}
|R_{T,t}^1x-y^1|&\ge &|R_{\tau,t}^1x-R_{\tau,T}^1y|-\|R_{T,\tau}-I\||R_{\tau,t}x-R_{\tau,T}y|.
\end{eqnarray*}
From Lemma \ref{Form of the resolvent}, and observing from its proof that we could also establish that $\sum_{j=1}^2\|(R_{T,\tau}-I)^{j,2}\|+\|(R_{T,\tau}-I)^{1,1}\|\le C(T-\tau),\ C:=C(\H,T_0),\ T_0\le 1 $ we get using \eqref{THE_FIRST_COM_TAU}:
\begin{eqnarray*}
|R_{T,t}^1x-y^1|&\ge & |R_{\tau,t}^1x-R_{\tau,T}^1y|(1-C(T-\tau)).
\end{eqnarray*}
Thus, for $T$ small enough we get: $(T-t)|R_{T,t}^1x-y^1|\ge \frac {T-t}2 |R_{\tau,t}^1x-R_{\tau,T}^1y|\overset{\eqref{THE_FIRST_COM_TAU}}{\ge}\frac{1}2|R_{\tau,t}^2x-R_{\tau,T}^2y|  $. We then derive similarly that:
\begin{eqnarray*}
|R_{\tau,t}^2x-R_{\tau,T}^2y|&\ge& |R_{T,t}^2x-y^2|-\|R_{\tau,T}-I\||R_{T,t}x-y|\\
&\ge & \frac{|R_{T,t}^2x-y^2|}{2}-C(T-\tau)|R_{T,t}^1x-y^1|.
\end{eqnarray*}
This finally yields that 
\begin{equation}
\label{EQUIV_FIRST_COMP}
(T-t)|R_{T,t}^1x-y^1|\ge \frac{|R_{T,t}^2x-y^2|}{4(1+C)}, 
\end{equation}
that is, the first component dominates in the contribution $|(\T_{T-t}^\alpha)^{-1}(R_{T,t}x-y)| $ appearing in $D_2$.
Write now:
\begin{equation}\label{cond_prealable}
|z-R_{\tau,T}y| \leq |z^1-R^1_{\tau,t}x|+|z^2-R^2_{\tau,t}x| + |R_{\tau,t}x-R_{\tau,T}y|.
\end{equation}
\begin{trivlist}
\item[$\diamond$] Suppose first that $(\tau-t)|z^1-R^1_{\tau,t}x|\leq |z^2-R^2_{\tau,t}x|$. Since $z \in D_2$, we have from \eqref{EQUIV_FIRST_COMP}: 
$$
|z^2 -R^2_{\tau,t}x| \leq C(\tau-t)(T-\tau)^\beta |R_{T,t}^1x -y^1|.
$$
Consequently, plugging the last two inequalities into \eqref{cond_prealable}, we get:
\begin{eqnarray*}
|z-R_{\tau,T}y| &\leq& \left( \frac{1}{\tau-t}+1 \right)|z^2-R^2_{\tau,t}x| + |R_{\tau,t}x-R_{\tau,T}y|\\
&\leq& \Big(1+(\tau-t) \Big)(T-\tau)^\beta|R^1_{T,t}x-y^1| + |R_{\tau,t}x-R_{\tau,T}y|\\
&\leq& C|x-R_{t,T}y|,
\end{eqnarray*}
using the Lipschitz property of the flow for the last inequality.

\item[$\diamond$] Assume now that $|z^2-R^2_{\tau,t}x|\leq (\tau-t)|z^1-R^1_{\tau,t}x|\leq |z^1-R^1_{\tau,t}x|$. We exploit that $z \in D_2$ and \eqref{EQUIV_FIRST_COMP} to write: 
$$|z^1 -R^1_{\tau,t}x| \leq C(T-\tau)^\beta |R_{T,t}^1x -y^1|. $$
Plugging the last two inequalities into \eqref{cond_prealable} yields:
\begin{eqnarray*}
|z-R_{\tau,T}y| &\leq& 2|z^1-R^1_{\tau,t}x| + |R_{\tau,t}x-R_{\tau,T}y|\\
&\leq& 2C(T-\tau)^\beta|R_{T,t}^1x-y^1| + |R_{\tau,t}x-R_{\tau,T}y|
\leq
C|x-R_{t,T}y|,
\end{eqnarray*}
using again the Lipschitz property of the flow for the last inequality.
\end{trivlist}
Thus, in both cases, 
\begin{equation}
\label{CTR_DEL_TAULETAU0_FIRSTDOM}
 |z-R_{\tau,T}y| \leq C|x-R_{t,T}y| \Rightarrow \delta\wedge |z-R_{\tau,T}y|^{\eta(\alpha \wedge 1)} \leq C\delta\wedge |x-R_{t,T}y|^{\eta(\alpha \wedge 1)}.
 \end{equation}
 It could similarly be shown that when $\sigma(t,x):=\sigma(t,x^2) $:
\begin{eqnarray}
\label{CTR_DEL_TAULETAU0_FIRSTDOM_PARTIAL_DEP}
 \delta\wedge |(z-R_{\tau,T}y)^2|^{\eta(\alpha \wedge 1)} &\leq & C\delta\wedge \{(T-t)|(x-R_{t,T}y)^1| +|(x-R_{t,T}y)^2|\}^{\eta(\alpha \wedge 1)}\nonumber\\
 &\le &C\delta\wedge \{(T-t)|(R_{T,t}x-y)^1| +|(R_{T,t}x-y)^2|\}^{\eta(\alpha \wedge 1)},
 \end{eqnarray}
 using a direct modification of Lemma \ref{ScalingLemma} for the last inequality.
Taking out this contribution from the spatial integral we get:
\begin{eqnarray*}
\bar A_{\alpha,I_2\cap\{\tau\le \tau_0 \}, D_2} (t,T,x,y)  &\leq& C \int_{I_2}d\tau \frac{\delta\wedge |x-R_{t,T}y|^{\eta(\alpha \wedge 1)} }{T-\tau} \ind_{\tau \leq \tau_0} \int \bar{p}_\alpha(t,\tau,x,z )\bar{p}_\alpha(\tau,T,z,y) dz\\ 
&\leq& C \delta\wedge |x-R_{t,T}y|^{\eta(\alpha \wedge 1)}  \log \Big(K \vee |(\T_{T-t}^\alpha)^{-1}(y-R_{T,t}x)| \Big) \bar{p}_\alpha(t,T,x,y),
\end{eqnarray*}
using the semigroup property of Corollary \ref{SG_PROP} for the last inequality.\\

To complete the analysis for this contribution, it remains to consider the case $\tau \in I_1$.
In this case, $T-t\asymp T-\tau$, and we have by triangle inequality:
$$
\delta \wedge |z-R_{\tau,T}y|^{\eta(\alpha \wedge 1)} \leq 
C \left( \delta \wedge |z-R_{\tau,t}x|^{\eta(\alpha \wedge 1)} + \delta \wedge |x-R_{t,T}y|^{\eta (\alpha \wedge 1)} \right).
$$
Recalling that  $T-\tau$ is not singular and splitting the integrals 
accordingly yields: 
\begin{eqnarray*}
\bar A_{\alpha, I_1}(t,T,x,y) \leq C\int_{I_1} d\tau \int_{\R^{nd}} dz \bar{p}_\alpha(t,\tau,x,z) \frac{\delta \wedge |x-R_{t,\tau}z|^{\eta(\alpha \wedge 1)}}{\tau-t} \bar{p}_\alpha(\tau,T,z,y)\\
 + C\delta \wedge |x-R_{t,T}y|^{\eta(\alpha \wedge 1)} \bar{p}_\alpha(t,T,x,y),
\end{eqnarray*}
where we used the semigroup property of Corollary \ref{SG_PROP} for the last term in the r.h.s. 
Now, for the first term in the above r.h.s., the previous arguments apply. Similarly to \eqref{INEG_TRIANG_MS} one of the two terms $|(\T_{\tau-t}^{\alpha})^{-1}(R_{\tau,t}x-z)| $, $|(\T_{T-\tau}^{\alpha})^{-1}(R_{T,\tau}z-y)| $ is in the \textit{off-diagonal} regime. If it is the second one, then $\bar  p_\alpha(\tau,T,z,y)\le C \bar p_\alpha(t,T,x,y)$ and we conclude using Lemma \ref{lemme}. If it is the first term, then we can still perform the previous dichotomy along the dominating component in $|(\T_{\tau-t}^\alpha)^{-1}(R_{\tau,t}x-z)|$. If the fast component dominates, the density is not singular. When the first component dominates, we modify the previous partition $(D_i)_{i\in\{1,2\}}$, considering:
 $$
D_1= \{z \in \R^{nd} ; (\tau-t)^\beta |(\T_{T-t}^\alpha)^{-1}( y-R_{T,t}x)| \leq 
 |(\T_{T-\tau}^\alpha)^{-1}(z-R_{\tau,T}y)| \},
$$
$$
D_2= \{z \in \R^{nd} ; (\tau-t)^\beta|(\T_{T-t}^\alpha)^{-1}(y-R_{T,t}x)| >
  |(\T_{T-\tau}^\alpha)^{-1}( z-R_{\tau,T}y)| \}.
$$
From this point on, the proof is similar: on $D_1$, we compensate the singularity, as long as $\beta$ is like in \eqref{def_beta}.
When $z\in D_2$, we subdivide along 
$\delta_0$ $\leq$ or $>$ $(\tau-t)^\beta |(\T^\alpha_{T-t})^{-1}(y-R_{T,t}x)|$.
The first case is dealt as above. In the second case, we can integrate the time singularity.\\

\textit{Contributions involving $\breve p_\alpha(t,T,x,y) $.} We first  focus on the contribution
\begin{eqnarray*}
\breve A_{\alpha,I_2}&:=&\int_{I_2} d\tau \int \bar p_\alpha(t,\tau,x,z) \frac{\delta\wedge |z-R_{\tau,T }y|^{\eta(\alpha \wedge 1)}}{T-\tau} \breve p_\alpha(\tau,T,z,y)dz\\
&\le &\int_{I_2}d\tau \int \bar p_\alpha(t,\tau,x,z)\delta\wedge |z-R_{\tau,T }y|^{\eta(\alpha \wedge 1)} \I_{\frac{|(z-R_{\tau,T}y)^1|}{(T-\tau)^{1/\alpha}}\asymp |(\T_{T-\tau}^\alpha)^{-1} (z-R_{\tau,T}y)|\ge K}\\
&&\times \frac1{|(z-R_{\tau,T}y)^1|^{1+\alpha}} \theta(|\M_{T-t}^{-1}(z-R_{\tau,T}y)|)
\frac{dz}{(T-\tau)^{1+\frac1\alpha}(1+\frac{|(z-R_{\tau,T}y)^2|}{(T-\tau)^{1+\frac 1 \alpha}})^{1+\alpha}}.
\end{eqnarray*}
Using again the partition in equation \eqref{SPATIAL_PART}, we readily get from Lemma \ref{lemme}, similarly to the previous paragraph, that:
\begin{eqnarray*}
\breve A_{\alpha,I_2,D_1}&:=&\int_{I_2} d\tau \int_{D_1} \bar p_\alpha(t,\tau,x,z) \frac{\delta\wedge |z-R_{\tau,T }y|^{\eta(\alpha \wedge 1)}}{T-\tau} \breve p_\alpha(\tau,T,z,y)dz\\
&\le & C(T-\tau)^\omega \bar p_\alpha(t,T,x,y).
\end{eqnarray*}
On $D_2$ the previous arguments also apply for $\{\tau\ge \tau_0\} $, with the same definition of $\tau_0 $. Hence:
\begin{eqnarray*}
\breve A_{\alpha,I_2\cap \{ \tau>\tau_0\},D_2}\le C(T-\tau)^\omega \bar p_\alpha(t,T,x,y).
\end{eqnarray*}
The only remaining case to handle is when the slow component dominates at the current time $\tau$, i.e. $|(R_{\tau,t}x-R_{\tau,T}y)^1|\ge c_0(T-t)|(R_{\tau,t}x-R_{\tau,T}y)^2| $.

On the considered set, it has previously been proven on $D_2$ (see \eqref{CTR_DEL_TAULETAU0_FIRSTDOM}) that $\delta\wedge |z-R_{\tau,T }y|^{\eta(\alpha \wedge 1)}\le C \delta\wedge |x-R_{t,T }y|^{\eta(\alpha \wedge 1)}  $ which can be taken out of the integral. 
Thus on the considered set, recalling that $ |(z-R_{\tau,T}y)^1|\ge c |(R_{\tau,t}x-R_{\tau,T}y)^1|$:
\begin{eqnarray*}
\breve A_{\alpha,I_2\cap \{ \tau\le \tau_0\},D_2}\le C\delta\wedge |x-R_{t,T }y|^{\eta(\alpha \wedge 1)} \int_{I_2}d\tau \I_{\tau\le \tau_0}\int_{D_2}\bar p_\alpha(t,\tau,x,z)\times \I_{\frac{|(z-R_{\tau,T}y)^1|}{(T-\tau)^{1/\alpha}} \asymp |(\T_{T-\tau}^\alpha)^{-1} (z-R_{\tau,T}y)|\ge K}\\
\times \frac1{|(z-R_{\tau,T}y)^1|^{1+\alpha}} 
\frac{dz}{(T-\tau)^{1+\frac1\alpha}(1+\frac{|(z-R_{\tau,T}y)^2|}{(T-\tau)^{1+\frac 1 \alpha}})^{1+\alpha}}\theta(|\M_{T-\tau}^{-1}(R_{\tau,T}y-z)|)\\
\le C \delta\wedge |x-R_{t,T }y|^{\eta(\alpha \wedge 1)} \theta(|\M_{T-t}^{-1}(R_{t,T}y-x)|)\int_{I_2} d\tau \frac{\I_{\tau\le \tau_0}}{|(R_{\tau,t}x-R_{\tau,T}y)^1|^{1+\alpha}}\\
\times \int_{D_2} \frac{dz}{(\tau-t)^{\frac{2}\alpha+1} (1+\frac{|(R_{\tau,t}x-z)^1|}{(\tau-t)^{1/\alpha}}+\frac{|(R_{\tau,t}x-z)^2|}{(\tau-t)^{1+1/\alpha}} )^{2+\alpha}}\times 
\frac{\I_{\frac{|(z-R_{\tau,T}y)^1|}{(T-\tau)^{1/\alpha}} \asymp |(\T_{T-\tau}^\alpha)^{-1} (z-R_{\tau,T}y)|\ge K}}{(T-\tau)^{1+\frac1\alpha}(1+\frac{|(z-R_{\tau,T}y)^2|}{(T-\tau)^{1+\frac 1 \alpha}})^{1+\alpha}}\\
\le C \delta\wedge |x-R_{t,T }y|^{\eta(\alpha \wedge 1)}\frac{\theta(|\M_{T-t}^{-1}(R_{t,T}y-x)|)}{(T-t)^{1/\alpha}(1+|(\T_{T-t}^\alpha)^{-1}(R_{T,t}x-y)|)^{1+\alpha}}\\
\int_{I_2} d\tau \frac{\I_{\tau \le \tau_0}}{T-t}\int  \frac{dz_2}{(\tau-t)^{\frac{1}\alpha+1} (1+\frac{|(R_{\tau,t}x-z)^2|}{(\tau-t)^{1+1/\alpha}} )^{1+\alpha}}\frac{1}{(T-\tau)^{1+\frac1\alpha}(1+\frac{|(z-R_{\tau,T}y)^2|}{(T-\tau)^{1+\frac 1 \alpha}})^{1+\alpha}}\\
\le C \delta\wedge |x-R_{t,T }y|^{\eta(\alpha \wedge 1)}\frac{\theta(|\M_{T-t}^{-1}(R_{t,T}y-x)|)}{(T-t)^{1/\alpha}(1+|(\T_{T-t}^\alpha)^{-1}(R_{T,t}x-y)|)^{1+\alpha}}\\
\int_{I_2} d\tau \frac{\I_{\tau \le \tau_0}}{T-t} \frac{1}{(T-t)^{\frac{1}\alpha+1} (1+\frac{|(R_{\tau,t}x-R_{\tau,T}y)^2|}{(T-t)^{1+1/\alpha}} )^{1+\alpha}}.
\end{eqnarray*}
From this last inequality we deduce that if $|(R_{\tau,t}x-R_{\tau,T}y)^2|\ge  c_1 (T-t)|(R_{\tau,t}x-R_{\tau,T}y)^1|$, i.e. the components are equivalent, we get the expected control, which could have already been deduced from the fact that the fast component is equivalent to the global energy.
If such an equivalence does not hold, the natural control is:
\begin{eqnarray*}
\breve A_{\alpha,I_2\cap \{ \tau\le \tau_0\},D_2}\le
C \delta\wedge |x-R_{t,T }y|^{\eta(\alpha \wedge 1)}\frac{\theta(|\M_{T-t}^{-1}(R_{t,T}y-x)|)}{(T-t)^{1/\alpha}(1+|(\T_{T-t}^\alpha)^{-1}(R_{T,t}x-y)|)^{1+\alpha}}\\
\times \frac{1}{(T-t)^{\frac{1}\alpha+1} (1+\inf_{\tau\in [t,T]}\frac{|(R_{\tau,t}x-R_{\tau,T}y)^2|}{(T-t)^{1+1/\alpha}} )^{1+\alpha}}.
\end{eqnarray*}
Now in the stable case \textbf{[HS]}, we obtain:
\begin{eqnarray*}
\breve A_{\alpha,I_2\cap \{ \tau\le \tau_0\},D_2}
\le
C \delta\wedge |R_{\tau^*,t}x-R_{\tau^*,T }y|^{\eta(\alpha \wedge 1)}\frac{1}{(T-t)^{1/\alpha}(1+|(\T_{T-t}^\alpha)^{-1}(R_{\tau^*,t}x-R_{\tau^*,T}y)|)^{1+\alpha}}\\
\times \frac{1}{(T-t)^{\frac{1}\alpha+1} (1+\frac{|(R_{\tau^*,t}x-R_{\tau^*,T}y)^2|}{(T-t)^{1+1/\alpha}} )^{1+\alpha}},
\end{eqnarray*}
where $\tau^* $ achieves the minimum.
In the tempered case \textbf{[HT]}, the control reads:
\begin{eqnarray*}
\breve A_{\alpha,I_2\cap \{ \tau\le \tau_0\},D_2}\le 
 C \delta\wedge |x-R_{t,T }y|^{\eta(\alpha \wedge 1)}\bar p_{\alpha,\Theta}(t,T,x,y)
\times \frac{1}{(T-t)^{\frac{1}\alpha} (1+\frac{|(R_{\tau^*,t}x-R_{\tau^*,T}y)^2|}{(T-t)^{1+1/\alpha}} )^{1+\alpha}}.
\end{eqnarray*}
Similar controls could be established by symmetry for $\breve A_{\alpha,I_1} $.
These bounds thus yield in both cases an additional time-singularity in $(T-t)^{-1/\alpha}$ if $|(R_{\tau^*,t}x-R_{\tau^*,T}y)^2|\le K (T-t)^{1+1/\alpha}$ and a possible loss of concentration in the stable case. They also turn out to be difficult to exploit in order to iterate in the series to establish the existence of the density and related bounds.

\end{trivlist}

Now if $\sigma(t,x):=\sigma(t,x_2) $ we can get rid of the additional singularity in the tempered case, writing:
\begin{eqnarray*}
\breve A_{\alpha, \{ \tau\le \tau_0\},D_2}\le C \int_t^{T-\tau_0}d\tau \int_{D_2}\bar p_\alpha(t,\tau,x,z)  \frac{\delta\wedge |(z-R_{\tau,T}y)^2|^{\eta(\alpha\wedge 1)}}{|(z-R_{\tau,T}y)^1|^{1+\alpha}}\theta(|\M_{T-\tau}^{-1}(z-R_{\tau,T}y)|)\times\\
\frac{1}{(T-\tau)^{1+\frac 1\alpha}}\frac{dz}{(1+\frac{|(z-R_{\tau,T}y)^2|}{(T-\tau)^{1+\frac1 \alpha }})^{1+\alpha} }\\
\le C\frac{\theta(|\M_{T-t}^{-1}(R_{T,t}x-y)|)}{|(R_{T,t}x-y)^1|^{1+\alpha}}
\times C \int_t^{T-\tau_0}d\tau \int \frac{1}{(\tau-t)^{1+\frac 1 \alpha}(1+\frac{|(R_{\tau,t}x-z)^2|}{(\tau-t)^{1+\frac 1 \alpha}})^{1+\alpha}} 
\\\delta\wedge |(z-R_{\tau,T}y)^2|^{\eta(\alpha\wedge 1)} 
\frac{1}{(T-\tau)^{1+\frac 1\alpha}}\frac{dz^2}{(1+\frac{|(z-R_{\tau,T}y)^2|}{(T-\tau)^{1+\frac1 \alpha }})^{1+\alpha} }.
\end{eqnarray*}
Observe now from \textbf{[HT]} that we have the control:
\begin{eqnarray*}
\frac{\theta(|\M_{T-t}^{-1}(R_{T,t}x-y)|)}{|R_{T,t}x-y|^{1+\alpha}}\le \frac{|(R_{T,t}x-y)^1|\theta(|\M_{T-t}^{-1}(R_{T,t}x-y)|)}{|(R_{T,t}x-y)^1|^{2+\alpha}}\le \frac{\Theta(|\M_{T-t}^{-1}(R_{T,t}x-y)|)}{|R_{T,t}x-y|^{1+\alpha}}=\bar p_{\alpha,\Theta}(t,T,x,y),
\end{eqnarray*}
on the considered case (i.e. the first component dominates in the off-diagonal regime). Hence:
\begin{eqnarray}
\breve A_{\alpha, \{ \tau\le \tau_0\},D_2}
\le C\bar p_{\alpha,\Theta}(t,T,x,y)\int_{t}^{T-\tau_0}d\tau\left\{ (\tau-t)^{-(1+\frac 1\alpha)(q-1)}\int dz^2 \frac{1}{(\tau-t)^{1+\frac 1 \alpha}} \frac{1}{(1+ \frac{|(z-R_{\tau,t}x)^2|}{(\tau-t)^{1+\frac 1\alpha}})^{(1+\alpha)q}} \right\}^{1/q}\nonumber \\
\times \left\{(T-\tau)^{-\{(1+\frac 1\alpha)(p-1)\}} \int [\delta\wedge |(z-R_{\tau,T}y)^2|^{\eta(\alpha\wedge 1)}]^p
\frac{1}{(T-\tau)^{1+\frac 1\alpha}}
\frac{dz^2}{(1+\frac{|(z-R_{\tau,T}y)^2|}{(T-\tau)^{1+\frac1 \alpha }})^{(1+\alpha)p} }
\right\}^{1/p}\label{HOLDER}\\
\le C\bar p_{\alpha,\Theta}(t,T,x,y)\int_{t}^{T-\tau_0}d\tau (\tau-t)^{-(1+\frac 1\alpha)\frac{1}{p}}\times (T-\tau)^{-\{(1+\frac 1\alpha)\frac{1}{q}\}}\{(T-t)^{(1+\frac 1\alpha) \eta (\alpha\wedge 1)}\}\nonumber,
\end{eqnarray}
where $p,q>1,p^{-1}+q^{-1}=1$ and s.t.  $p>1+\frac 1\alpha $ for $\tau\in [t,\frac{t+T}2]$ and $q>1+\frac 1\alpha$ for $\tau\in [\frac{t+T}{2},T] $. 
Also, the regularizing term $(T-t)^{(1+\frac1\alpha)\eta(\alpha \wedge 1)} $ in the last control can be derived following the proof of Lemma \ref{lemme}.
We thus derive:
\begin{eqnarray*}
\breve A_{\alpha, \{ \tau\le \tau_0\},D_2} \le C\bar p_{\alpha,\Theta}(t,T,x,y)(T-t)^{(1+\frac 1 \alpha)\eta(\alpha \wedge 1)}(T-t)^{1-(1+\frac1\alpha)}
\le C\bar p_{\alpha,\Theta}(t,T,x,y)(T-t)^{1+(1+\frac{1}{\alpha})(\eta(\alpha\wedge 1)-1)}.
\end{eqnarray*}
Therefore, the last contribution gives a smoothing effect provided:
\begin{eqnarray*}
(1+\frac{1}{\alpha})(\eta(\alpha\wedge 1)-1)>-1 \iff \eta>\frac{1}{(\alpha \wedge 1)(1+\alpha)}.
\end{eqnarray*}
The controls associated with $\bar p_{\alpha} $, when $\sigma(t,x)=\sigma(t,x^2) $, yielding the contribution in $\bar q_{\alpha,\Theta} $   in the Lemma, could be easily deduced in the current case from the previous analysis, exploiting \eqref{CTR_DEL_TAULETAU0_FIRSTDOM_PARTIAL_DEP} instead of \eqref{CTR_DEL_TAULETAU0_FIRSTDOM}. 

\end{proof}


The convergence of the parametrix series \eqref{parametrix} will now follow from controls involving the convolutions of $H$ with the last term $\bar{q}_{\alpha,\Theta}(t,T,x,y)$. The following lemma completes the proof of Lemma \ref{LEMME_IT_KER}.

\begin{lemma}\label{DeuxiemeCoup} Assume $d=1,n=2,\sigma(t,x)=\sigma(t,x_2)$ and $ \eta>\frac{1}{(\alpha \wedge 1)(1+\alpha)}$.
There exist $C_{\ref{DeuxiemeCoup}}:=C_{\ref{DeuxiemeCoup}}(\H)>0,\ \omega:=\omega(\H)\in (0,1]$ s.t.
for all $T\in (0,T_0],\ T_0:=T_0(\H)\le 1,\ (x,y)\in (\R^{nd})^2$, $t\in [0,T)$,

\begin{eqnarray*}
| \bar{q}_{\alpha,\Theta} \otimes H|(t,T,x,y)\le
C (T-t)^\omega\Big( \bar{p}_{\alpha,\Theta}(t,T,x,y) \\
+ \delta \wedge \{ (T-t)| (x-R_{t,T}y)^1|+|(x-R_{t,T}y)^2|\}^{\eta(\alpha\wedge 1)} \log \Big(K \vee |(\T_{T-t}^\alpha)^{-1}(y-R_{T,t}x)| \Big)\bar{p}_{\alpha,\Theta}(t,T,x,y) \Big).
\end{eqnarray*}
\end{lemma}

\begin{proof}

Recall that $\bar{q}_{\alpha,\Theta}(t,T,x,y)$ writes as the sum of 
$$q_{\alpha,\Theta} (t,T,x,y) := \delta \wedge \{ (T-t)| (x-R_{t,T}y)^1|+|(x-R_{t,T}y)^2|\}^{\eta(\alpha \wedge 1)} \bar{p}_{\alpha,\Theta}(t,T,x,y)$$
and 
$$\rho_{\alpha,\Theta}(t,T,x,y) := \delta \wedge \{(T-t)| (x-R_{t,T}y)^{1}|+| (x-R_{t,T}y)^{2}|\}^{\eta(\alpha \wedge 1)} \log \Big(K \vee |(\T_{T-t}^\alpha)^{-1}(y-R_{T,t}x)| \Big)\bar{p}_{\alpha,\Theta}(t,T,x,y).$$

Though the lines of the proof are similar to those of Lemma \ref{premier coup}, we treat the two convolutions separately, to emphasize the difficulties induced by the rediagonalization and the logarithmic factor.
First, for $|q_{\alpha,\Theta} \otimes H|(t,T,x,y)$, we bound $|H|$ using  Lemma \ref{CTR_KER_PTW}, to get:
 \begin{eqnarray*}
|q_{\alpha,\Theta} \otimes H|(t,T,x,y)\leq C\int_t^T d\tau \int_{\R^{nd}} \delta\wedge\{(\tau-t)|(z-R_{\tau,t}x)^1|+|(z-R_{\tau,t}x)^2|\}^{\eta(\alpha \wedge 1)} \bar{p}_{\alpha,\Theta}(t,\tau,x,z) \\
\times \frac{\delta\wedge|(z-R_{\tau,T}y)^2|^{\eta(\alpha \wedge 1)}}{T-\tau} (\bar{p}_{\alpha}+\breve p_\alpha)(\tau,T,z,y).
 \end{eqnarray*}
 The above contribution can be handled as in Lemma \ref{premier coup}, in the \textit{diagonal} case $|(\T_{T-t}^\alpha)^{-1}(y-R_{T,t}x)|\leq K$, or in the \textit{off-diagonal} case 
$|(\T_{T-t}^\alpha)^{-1}(y-R_{T,t}x)|> K$ when for a given integration time $\tau\in [t,T] $ the \textit{fast component} dominates, i.e. 
$|R^2_{\tau,T}y-R^2_{\tau,t}x|\ge (T-t)|R^1_{\tau,T}y-R^1_{\tau,t}x|$.
The only difference is that we do not need to use the triangle inequality in order to apply Lemma \ref{lemme}. Indeed,  regularizing terms $\delta \wedge |(z-R_{\tau,T}y)^2|^{\eta(\alpha \wedge 1)}$,  $\delta \wedge \{(\tau-t) |(z-R_{\tau,t}x)^1|+|(z-R_{\tau,t}x)^2| \}^{\eta (\alpha \wedge 1)}$  already appear for both densities.

When $|(\T_{T-t}^\alpha)^{-1}(y-R_{T,t}x)|> K$ and 
$|R^2_{\tau,T}y-R^2_{\tau,t}x|\le (T-t)|R^1_{\tau,T}y-R^1_{\tau,t}x|$, we split as in the previous proof the time interval into $I_1\cup I_2:=[t,\frac{T+t}{2}]\cup [\frac{T+t}{2},T] $.
Suppose $\tau \in I_2$. We consider the spatial partition introduced in \eqref{SPATIAL_PART}.

For $z\in D_1$, we have $\bar{p}_{\alpha,\Theta}(t,\tau,x,z) \leq C(T-\tau)^{-\beta(2+\alpha)} \bar{p}_{\alpha,\Theta}(t, T ,x,y)$.
This yields a regularization property from Lemma \ref{lemme} when $\beta$ satisfies \eqref{def_beta}. 
For $z\in D_2$ and a given $\delta_0>0 $, we use again the partition $(T-\tau)^\beta |(\T_{T-t}^\alpha)^{-1}( y-R_{T,t}x)| $ $\geq$ or $<$ $\delta_0$.
The case $(T-\tau)^\beta |(\T_{T-t}^\alpha)^{-1}( y-R_{T,t}x)| \leq \delta_0$ yields a regularization in time similarly to the previous proof.

In order for $(T-\tau)^\beta |(\T_{T-t}^\alpha)^{-1}( y-R_{T,t}x)| $ to exceed $\delta_0$, we see that $\tau$ must be lower than $\tau_0:=T-\left( \frac{\delta_0}{|(\T^\alpha_{T-t})^{-1}(y-R_{T,t}x)|}\right)^{\frac{1}{\beta}}
$. In that case, the time singularity is still logarithmically explosive but integrable. We are led to consider:
\begin{equation}\label{CasCritique}
\begin{split}
\Gamma:=\int_{I_2}d\tau\frac{1}{T-\tau}\ind_{\tau\le \tau_0}\int_{D_2} \delta\wedge\{(\tau-t)|(z-R_{\tau,t}x)^1|+|(z-R_{\tau,t}x)^2|\}^{\eta(\alpha \wedge 1)} \bar{p}_{\alpha,\Theta}(t,\tau,x,z)\\
\times \delta\wedge|(z-R_{\tau,T}y)^2|^{\eta(\alpha \wedge 1)} (\bar{p}_\alpha+\breve p_\alpha)(\tau,T,z,y) dz.
\end{split}
\end{equation}
Using iteratively the scaling Lemma \ref{ScalingLemma} we derive:
\begin{eqnarray*}
\frac{|y^1-R_{T,\tau}^1z|}{(T-t)^{\frac1\alpha}}+ \frac{|y^2-R_{T,\tau}^2z|}{(T-t)^{1+\frac{1}{\alpha}}}\ge c_2|(\T_{T-t}^{\alpha})^{-1}(y-R_{T,\tau}z)|\\
\ge c_2 C^{-1}\left\{ |(\T_{T-t}^{\alpha})^{-1}(R_{\tau,t}x-R_{\tau,T}y)|-|(\T_{T-t}^{\alpha})^{-1}(z-R_{\tau,t}x)|\right\}\\
\ge c_2 \{C^{-1}|(\T_{T-t}^{\alpha})^{-1}(R_{\tau,t}x-R_{\tau,T}y)|-C^{-1}(T-\tau)^\beta|(\T_{T-t}^{\alpha})^{-1}(R_{T,t}x-y)|\}\\
\ge c_2 \{C^{-1} -(T-\tau)^\beta \}|(\T_{T-t}^{\alpha})^{-1}(R_{\tau,t}x-R_{\tau,T}y)|, \ c_2>0,C:=C(T)\ge 1,
\end{eqnarray*}
recalling that $z\in D_2$ for the last but one inequality.
Thus, for $T$ small enough and up to a modification of $C$, we have either $|y^1-R_{T,\tau}^1z| \geq C |R_{\tau,t}^1x-R_{\tau,T}^1y|$, or $|y^2-R_{T,\tau}^2z| \geq C (T-t)|R_{\tau,t}^1x-R_{\tau,T}^1y|$. In both cases,  $\bar{p}_\alpha(\tau,T,z,y) \leq \frac{C}{|R_{\tau,t}^1x-R_{\tau,T}^1y|^{2+\alpha}}\theta(|\M_{T-t}^{-1}(R_{\tau,t}^1x-R_{\tau,T}^1y)|)$.
 This yields from Proposition \ref{EST_DENS_GEL} $\bar p_\alpha(\tau,T,z,y)\le C \bar p_\alpha(t,T,x,y) $.
In our current case, we then derive from \eqref{CTR_DEL_TAULETAU0_FIRSTDOM_PARTIAL_DEP} 
that:
\begin{eqnarray*}
\delta \wedge |(z-R_{\tau,T}y)^2|^{\eta(\alpha \wedge 1)} \bar{p}_\alpha (\tau,T,z,y) \leq \delta \wedge \{(T-t)|(x-R_{t,T}y)^1|+|(x-R_{t,T}y)^2|\}^{\eta(\alpha \wedge 1)} \bar{p}_\alpha (t,T,x,y). 
\end{eqnarray*}
Consequently, we can bound \eqref{CasCritique} by:
\begin{eqnarray*}
\Gamma\le C(T-t)^{\omega}
\bar{p}_\alpha (t,T,x,y)\\
+\int_{I_2}d\tau\frac{1}{T-\tau}\ind_{\tau\le \tau_0}\int_{D_2} \delta\wedge\{(\tau-t)|(z-R_{\tau,t}x)^1|+|(z-R_{\tau,t}x)^2|\}^{\eta(\alpha \wedge 1)} \bar{p}_{\alpha,\Theta}(t,\tau,x,z)\\
\times \delta\wedge|(z-R_{\tau,T}y)^2|^{\eta(\alpha \wedge 1)} \breve p_\alpha(\tau,T,z,y) dz:=\Gamma_1+\Gamma_2.
\end{eqnarray*}
It thus remains to handle $\Gamma_2$ which derives from the rediagonalization. We write:
\begin{eqnarray*}
\Gamma_2\le C\bar p_{\alpha,\Theta}(t,T,x,y)
\int_t^T d\tau \int_{\R^2} \bar p_{\alpha,\Theta}(t,\tau,x,z)\{\delta \wedge \{(\tau-t)|(z-R_{\tau,t}x)^1|\}^{\eta(\alpha \wedge 1)}+\delta \wedge |(z-R_{\tau,t}x)^2|^{\eta(\alpha\wedge 1)} \}\\
\times\frac{1}{(T-\tau)^{1+\frac 1 \alpha}}\frac{\delta\wedge |(z-R_{\tau,T}y)^2|^{\eta(\alpha \wedge 1)}}{(1+\frac{|(z-R_{\tau,T}y)^2|}{(T-\tau)^{1+\frac 1\alpha}})^{1+\frac 1\alpha}}dz\\
\le C\bar p_{\alpha,\Theta}(t,T,x,y)\int_t^T d\tau  \int dz_2 \frac{1}{(T-\tau)^{1+\frac 1\alpha}}\frac{\delta\wedge |(z-R_{\tau,T}y)^2|^{\eta(\alpha \wedge 1)}}{(1+\frac{|(z-R_{\tau,T}y)^2|}{(T-\tau)^{1+\frac 1\alpha}})^{1+\frac 1\alpha}}\{(\tau-t)^{(1+\frac 1\alpha)[\eta(\alpha \wedge 1)-1]}+ \\
\frac{\delta \wedge |(R_{\tau,t} x-z)^2|^{\eta(\alpha \wedge 1)}}{(\tau-t)^{1+\frac1 \alpha}(1+\frac{|(z-R_{\tau,t}x)^2|}{(\tau-t)^{1+\frac1\alpha}})^{1+\alpha}}\}\\
\le C\bar p_{\alpha,\Theta}(t,T,x,y)\int_t^T d\tau \big\{(T-\tau)^{(1+\frac 1\alpha)\eta(\alpha \wedge 1)}(\tau-t)^{(1+\frac 1\alpha)[\eta(\alpha \wedge 1)-1]}\\
+(\tau-t)^{(1+\frac 1 \alpha)\{\eta(\alpha \wedge 1)-1/2\}}(T-\tau)^{(1+\frac 1 \alpha)\{\eta(\alpha \wedge 1)-1/2\}}\big\}\le C\bar p_{\alpha,\Theta}(t,T,x,y)(T-t)^{2(1+\frac 1 \alpha)\eta(\alpha\wedge 1)-\frac 1 \alpha},
\end{eqnarray*}
proceeding as in \eqref{HOLDER} and using Lemma \ref{lemme} for the last but one inequality.
This indeed gives a regularizing effect recalling that we have assumed $1\ge \eta >\frac{1}{(\alpha \wedge 1)(1+\alpha)} $.
Note that the case $\tau\in I_1 $ could be handled similarly, see Lemma \ref{premier coup}.
The controls become:
\begin{equation}\label{qxH}
q_{\alpha,\Theta} \otimes |H|(t,T,x,y) \leq C(T-t)^\omega \left( \bar{p}_{\alpha,\Theta}(t,T,x,y) + 
 \log \Big(K \vee |(\T_{T-t}^\alpha)^{-1}(y-R_{T,t}x)| \Big) q_{\alpha,\Theta}(t,T,x,y) \right).
\end{equation}
We point out that the important contribution in the above equation is the factor $(T-t)^\omega$, whose power will grow at each iteration. This key feature gives the convergence of the series \eqref{parametrix}.

Now, for $\rho_{\alpha,\Theta} \otimes |H|(t,T,x,y)$, we still bound $|H|$ using Lemma \ref{CTR_KER_PTW}:
\begin{equation}
\begin{split}
\rho_{\alpha,\Theta} \otimes |H|(t,T,x,y)\leq C\int_t^T d\tau \int_{\R^{2}} dz \log\Big(K\vee |(\T_{\tau-t}^{\alpha})^{-1}(z-R_{\tau,t}x)|\Big) \\
\times \delta \wedge \{(\tau-t)|(z-R_{\tau,t}x)^1|+|(z-R_{\tau,t}x)^2|\}^{\eta(\alpha \wedge 1)} \bar{p}_{\alpha,\Theta}(t,\tau,x,z)
\times \frac{\delta\wedge|z-R_{\tau,T}y|^{\eta(\alpha \wedge 1)}}{T-\tau} (\bar{p}_{\alpha}+\breve{p}_{\alpha})(\tau,T,z,y).
\end{split}
\end{equation}
W.r.t. the previous contribution, the main difference comes from the logarithm.
However, the lines of the proof remain the same. Suppose first that 
 $|(\T_{T-t}^\alpha)^{-1}(y-R_{T,t}x)|\leq K$.
 Depending on the time parameter $\tau$, we can show that we always have either
$\bar{p}_{\alpha,\Theta}(t,\tau,x,z) \leq C\bar {p}_{\alpha,\Theta}(t,T,x,y)$ 
or $\bar{p}_{\alpha}(\tau,T,z,y)\leq C\bar{p}_{\alpha}(t,T,x,y)\le C\bar{p}_{\alpha,\Theta}(t,T,x,y)$.
The second case occurs when $\tau \in I_1$. Using the notations of the previous proof, this yields:
\begin{eqnarray*}
\rho_{\alpha,\Theta} \otimes_{|I_1} |H|(t,T,x,y) &\leq& C\bar{p}_{\alpha,\Theta}(t,T,x,y)\int_{I_1}d\tau\int_{\R^{2}} \log\Big(K\vee |(\T_{\tau-t}^{\alpha})^{-1}(z-R_{\tau,t}x)|\Big) \\
&&\times \frac{\delta \wedge \{(\tau-t)|(z-R_{\tau,t}x)^1|+|(z-R_{\tau,t}x)^2|\}^{\eta(\alpha\wedge 1) }}{\tau-t}\bar{p}_{\alpha,\Theta}(t,\tau,x,z) dz,
\end{eqnarray*}
and we conclude by Lemma \ref{lemme2}.
In the case when $\bar{p}_{\alpha,\Theta}(t,\tau,x,z) \leq C\bar{p}_{\alpha,\Theta}(t,T,x,y)$, which happens for $\tau \in I_2$,  we have:
$$
|(\T_{\tau-t}^\alpha)^{-1}(z-R_{\tau,t}x)| \leq C(|(\T_{\tau-t}^\alpha)^{-1}(y-R_{T,t}x)|+|(\T_{\tau-t}^\alpha)^{-1}(y-R_{T,\tau}z)|)\leq C(K+|(\T_{T-\tau}^\alpha)^{-1}(y-R_{T,\tau}z)|).
$$
Plugging this inequality into the logarithm and taking out the first density, we can bound:
\begin{eqnarray*}
\rho_{\alpha,\Theta} \otimes_{|I_2} |H|(t,T,x,y) &\leq &C \bar{p}_{\alpha,\Theta}(t,T,x,y)\int_{I_2}  d\tau\int_{\R^{2}} \log\Big(K\vee |(\T_{T-\tau}^\alpha)^{-1}(y-R_{T,\tau}z)|\Big)\\ 
&&\times \frac{\delta \wedge \{ (T-\tau)|(z-R_{\tau,T}y)^1|+|(z-R_{\tau,T}y)^2| \}^{\eta(\alpha \wedge 1)} }{T-\tau}(\bar{p}_{\alpha,\Theta} +\breve{p}_\alpha)(\tau,T,z,y) dz,
\end{eqnarray*}
and once again, we conclude by Lemma \ref{lemme2}. Thus, we have so far managed to show that in the global diagonal regime, $\rho_{\alpha,\Theta} \otimes |H|(t,T,x,y) \le C(T-t)^\omega \bar{p}_{\alpha,\Theta}(t,T,x,y)$.

It remains to deal with the case when $|(\T_{T-t}^\alpha)^{-1}(y-R_{T,t}x)|\geq K$.
Suppose first that $\tau\in I_2$, and that the first component dominates in the global action $|(\T_{T-t}^\alpha)^{-1}(y-R_{T,t}x)|$, i.e. $|(\T_{T-t}^\alpha)^{-1}(y-R_{T,t}x)|\asymp \frac{|y^1-R_{T,t}^1x|}{(T-t)^{1/\alpha}} $. 
We still consider the partition in \eqref{SPATIAL_PART}.

When $z \in D_1$, we can bound
\begin{equation}\label{BORNE_DENS_D1}
\bar{p}_{\alpha,\Theta}(t,\tau,x,z) \le C(T-\tau)^{-\beta(2+\alpha)}\bar{p}_{\alpha,\Theta}(t,T,x,y).
\end{equation}

On the other hand, the triangle inequality and the scaling Lemma \ref{ScalingLemma} yield:
$$
 |(\T_{\tau-t}^\alpha)^{-1}(z-R_{\tau,t}x)| \leq C \Big(
|(\T_{\tau-t}^\alpha)^{-1}(y-R_{T,\tau}z)|+ |(\T_{\tau-t}^\alpha)^{-1}( y-R_{T,t}x)| \Big).
$$
Consequently, up to a modification of $C$, we have either:
$$
 |(\T_{\tau-t}^\alpha)^{-1}(z-R_{\tau,t}x)| \leq C 
 |(\T_{\tau-t}^\alpha)^{-1}( y-R_{T,t}x)|
\mbox{ or }
 |(\T_{\tau-t}^\alpha)^{-1}(z-R_{\tau,t}x)| \leq C 
 |(\T_{\tau-t}^\alpha)^{-1}(y-R_{T,\tau}z)|. 
$$
Define accordingly, 
$$D_{1,1}= \{ z \in D_1;  |(\T_{\tau-t}^\alpha)^{-1}(z-R_{\tau,t}x)| \leq C 
 |(\T_{\tau-t}^\alpha)^{-1}( y-R_{T,t}x)| \},$$
 $$D_{1,2} = \{ z \in D_1;   |(\T_{\tau-t}^\alpha)^{-1}(z-R_{\tau,t}x)| \leq C 
 |(\T_{\tau-t}^\alpha)^{-1}(y-R_{T,\tau}z)| \}.$$
Observe that with this definition, $D_{1,1}$ and $D_{1,2}$ is not a partition of $D_1$. However, $D_1 \subset D_{1,1} \cup D_{1,2} $.

When $z \in D_{1,1}$, we can bound 
$$
\log\Big(K\vee  |(\T_{\tau-t}^\alpha)^{-1}(z-R_{\tau,t}x)| \Big) \leq \log\Big(K\vee  |(\T_{T-t}^\alpha)^{-1}(y-R_{T,t}x)|\Big)+C.
$$
On the other hand, for $\tau\in I_2$, we get from the definition of $D_{1,1} $:
\begin{eqnarray*}
\delta\wedge \{ (\tau-t)|(z-R_{\tau,t}x)^1|+|(z-R_{\tau,t}x)^2|\}^{\eta (\alpha \wedge 1)} \\
\le C (\delta \wedge\{(T-t) |(x-R_{t,T}y)^1|+|(x-R_{t,T}y)^2|\}^{\eta(\alpha \wedge 1)}) .
\end{eqnarray*}

From \eqref{BORNE_DENS_D1}, we thus have:
\begin{eqnarray*}
\rho_{\alpha,\Theta} \otimes_{|I_2,D_{1,1}} |H| (t,T,x,y) 
\le
C \left(\log\Big(K\vee  |(\T_{T-t}^\alpha)^{-1}(y-R_{T,t}x)|\Big) +1\right)\\
\delta \wedge \{(T-t)|(x-R_{t,T}y)^1|+|(x-R_{t,T}y)^2|\}^{\eta(\alpha \wedge 1)} \\
\times \int_{I_2}d\tau \int_{D_{1,1}} \bar{p}_{\alpha,\Theta}(t,\tau,x,z) \frac{\delta\wedge |(z-R_{\tau,T}y)^2|^{\eta(\alpha \wedge 1)}}{T-\tau}(\bar{p}_\alpha+\breve{p}_\alpha)(\tau,T,z,y)dz\\
\le
(T-t)^\omega (\rho_\alpha+q_\alpha)(t,T,x,y),
\end{eqnarray*}
choosing $\beta $ satisfying \eqref{def_beta}.

When $z\in D_{1,2}$, we can bound:
$$
\log\Big(K\vee  |(\T_{\tau-t}^\alpha)^{-1}(z-R_{\tau,t}x)| \Big) \leq \log\Big(K\vee  |(\T_{\tau-t}^\alpha)^{-1}(y-R_{T,\tau}z)|\Big)+C.
$$
Bounding also roughly $\delta \wedge \{(\tau-t)|(z-R_{\tau,t}x)^1|+|(z-R_{\tau,t}x)^2|\}^{\eta(\alpha \wedge 1)}\le \delta$, and using the
bound \eqref{BORNE_DENS_D1}, we can write:
\begin{eqnarray*}
\rho_{\alpha,\Theta} \otimes_{|I_2,D_{1,2}} |H| (t,T,x,y) 
\le
 C\int_{I_2}d\tau\int_{D_{1,2}} \bar{p}_{\alpha,\Theta}(t,\tau,x,z)\left(\log\Big(K\vee  |(\T_{\tau-t}^\alpha)^{-1}(y-R_{T,\tau}z)|\Big) +1\right) \\
\times\frac{\delta\wedge |(z-R_{\tau,T}y)^2|^{\eta(\alpha \wedge 1)}}{T-\tau}(\bar{p}_\alpha+\breve{p}_\alpha)(\tau,T,z,y)dz\\
\le
C \bar{p}_{\alpha,\Theta}(t,T,x,y) \int_{I_2}d\tau  (T-\tau)^{-\beta(2+\alpha)}\\
\times
\hspace*{-.2cm}\int_{D_{1,2}} 
\left(\log\Big(K\vee  |(\T_{\tau-t}^\alpha)^{-1}(y-R_{T,\tau}z)|\Big)+1\right)
\frac{\delta\wedge |z-R_{\tau,T}y|^{\eta(\alpha \wedge 1)}}{T-\tau}(\bar{p}_\alpha+\breve p_\alpha)(\tau,T,z,y) dz.
\end{eqnarray*}
Thus, using Lemma \ref{lemme2}, we have $\rho_{\alpha,\Theta} \otimes_{|I_2,D_{1,2}} |H| (t,T,x,y) \le (T-t)^\omega \bar{p}_\alpha(t,T,x,y)$.

We have to deal with $z \in D_2$.
In this case, \textbf{and because $d=1$}, $\bar{p}_\alpha(\tau,T,z,y) \le C\bar{p}_\alpha(t,T,x,y)$.
As above, we split for a given $\delta_0>0 $, the time interval $I_2$ in 
$(T-\tau)^\beta  |(\T_{T-t}^\alpha)^{-1}(y-R_{T,t}x)| \ge \delta_0$ and $(T-\tau)^\beta  |(\T_{T-t}^\alpha)^{-1}(y-R_{T,t}x)| < \delta_0$.

Assume first that
$(T-\tau)^\beta  |(\T_{T-t}^\alpha)^{-1}(y-R_{T,t}x)| \le \delta_0$.
Then, taking $\delta_0 \le K$ gives that the first density is diagonal.
Hence, the logarithm part disappears, and we have to deal with:

\begin{eqnarray*}
\rho_{\alpha,\Theta} \otimes_{I_2,D_2} |H|(t,T,x,y)  &\le& C\int_{\tau_0}^T d\tau\int_{D_2} \frac{1}{(T-t)^{\frac 2\alpha+1}} \frac{\delta\wedge |(z-R_{\tau,T}y)^2|^{\eta(\alpha \wedge 1)}}{T-\tau} (\bar{p}_\alpha+\breve{p}_\alpha)(\tau,T,z,y)dz\\
&\overset{Lemma \ \ref{lemme}}{\le}& \frac{ \delta_0^{2+\alpha} }{(T-t)^{\frac 2\alpha+1} |(\T_{T-t}^\alpha)^{-1}(y-R_{T,t}x)|^{2+\alpha}}\int_{\tau_0}^T d\tau (T-\tau)^{(1+\frac 1 \alpha)\eta(\alpha \wedge 1) -1 -\beta(2+\alpha)}\\
&\le& (T-t)^\omega \bar{p}_\alpha(t,T,x,y).
\end{eqnarray*}

Finally, we have to deal with the case $(T-\tau)^\beta  |(\T_{T-t}^\alpha)^{-1}(y-R_{T,t}x)| \ge \delta_0$.
Observe that, on $I_2$, this imposes that $\tau \in [\frac{T+t}{2}, \tau_0]$, with $\tau_0$ defined above.
In the considered set, we have from \eqref{CTR_DEL_TAULETAU0_FIRSTDOM}:
\begin{eqnarray*}
|(z-R_{\tau,T}y)^2|&\le& |(z-R_{\tau,t}x)^2| + C\{(T-t)|(x-R_{t,T}y)^1|+|(x-R_{t,T}y)^2|\} \\
&\le& C(1+ (T-\tau)^\beta)\{(T-t)|(x-R_{t,T}y)^1| +|(x-R_{t,T}y)^2|\}.
\end{eqnarray*}
Plugging this estimate into the convolution and recalling for $z\in D_2$, $\bar{p}_\alpha(\tau,T,z,y) \le C\bar{p}_\alpha(t,T,x,y)$, we obtain from Lemma \ref{lemme2} and the previous controls for the contribution in $\breve p_\alpha $:

\begin{eqnarray*}
\rho_{\alpha,\Theta} \otimes_{I_2,D_2} |H|(t,T,x,y)  \le C\bar{p}_{\alpha,\Theta}(t,T,x,y)\\
\left((\delta\wedge \{(T-t)|(x-R_{t,T}y)^1|+|(x-R_{t,T}y)^2|\}^{\eta(\alpha \wedge 1)}) \int_{\frac{T+t}{2}}^{\tau_0} d\tau \frac{1}{T-\tau} (\tau-t)^{\omega}+(T-t)^\omega\right).
\end{eqnarray*}
Hence, integrating over $\tau$ yields the logarithmic contribution:
$$
\rho_{\alpha,\Theta} \otimes_{I_2,D_2} |H|(t,T,x,y) \le C(T-t)^\omega \rho_{\alpha,\Theta}(t,T,x,y).
$$

In order to complete the proof, we have to specify how to proceed in the remaining cases, that is  when $\tau \in I_2$ and the second component dominates or when $\tau\in I_1$.
When a fast component dominates, as we have seen in the previous proof, we can compensate the singularities brought by the kernel $H$, and conclude directly with Lemmas \ref{lemme} and \ref{lemme2}. When $\tau \in I_1$, we can adapt the previous strategy following the procedure described in Lemma \ref{premier coup}.
\end{proof}

Using the previous lemmas, we get the following result.

\begin{corollary}\label{KER_ITER_ET_CONV_SERIE} Under the Assumptions of Lemma \ref{DeuxiemeCoup},
there exists $C_{\ref{KER_ITER_ET_CONV_SERIE}}:=4C_{\ref{LEMME_IT_KER}}>0$, s.t.
for all $T\in (0,T_0],\ T_0=T_0(\H)\le 1,\  (x,y)\in (\R^{2})^2$, $t\in [0,T)$, $\forall k\in \N $:
\begin{eqnarray*}
|\tilde p_\alpha\otimes H^{(2k)}(t,T,x,y)|\le C_{\ref{KER_ITER_ET_CONV_SERIE}}^{2k} (T-t)^{k\omega} \Big( (T-t)^{k\omega} \bar{p}_{\alpha,\Theta}(t,T,x,y) + (\bar{p}_{\alpha,\Theta} + \bar{q}_{\alpha,\Theta})(t,T,x,y) \Big)  \\
|\tilde p_\alpha\otimes H^{(2k+1)}(t,T,x,y)|\le C_{\ref{KER_ITER_ET_CONV_SERIE}}^{2k+1}(T-t)^{k\omega} \Big( (T-t)^{(k+1)\omega} \bar{p}_{\alpha,\Theta}+ (T-t)^\omega(\bar{p}_{\alpha,\Theta} + \bar{q}_{\alpha,\Theta})+ \bar{q}_{\alpha,\Theta} \Big)(t,T,x,y).
\end{eqnarray*}
\end{corollary}

\begin{proof}
We prove the estimate by induction. 
The idea is to use the controls of Lemmas \ref{premier coup} and \ref{DeuxiemeCoup} gathered in Lemma \ref{LEMME_IT_KER} to get from an estimate to the following one.
The bounds may not be very precise, as we will sometimes bound $(T-t)^{k\omega} \le 1$, but they are sufficient to prove the convergence of the Parametrix series \eqref{parametrix}.

\textbf{Initialization:}\\
Since $(T-t)^\omega(\bar p_{\alpha,\Theta} + \bar q_{\alpha,\Theta})\geq 0$, we clearly have:
$$
|\tilde p_\alpha\otimes H(t,T,x,y) | \leq C_{\ref{LEMME_IT_KER}} \Big((T-t)^\omega \bar{p}_{\alpha,\Theta}+ \bar q_{\alpha,\Theta} + (T-t)^\omega(\bar p_{\alpha,\Theta} + \bar q_{\alpha,\Theta}) \Big)(t,T,x,y).
$$
Now, using Lemmas \ref{premier coup} and \ref{DeuxiemeCoup}, we have:
\begin{eqnarray*}
|\tilde{p}_\alpha \otimes H^{(2)}(t,T,x,y)| &\leq&  C_{\ref{LEMME_IT_KER}} \Big((T-t)^\omega| \bar{p}_{\alpha,\Theta}\otimes H| + |\bar q_{\alpha,\Theta}  \otimes H| \Big)(t,T,x,y)\\
&\le& C_{\ref{LEMME_IT_KER}} \Big(C_{\ref{LEMME_IT_KER}}(T-t)^{2\omega}\bar{p}_{\alpha,\Theta} + C_{\ref{LEMME_IT_KER}}(T-t)^\omega \bar{q}_{\alpha,\Theta} + C_{\ref{LEMME_IT_KER}}(T-t)^\omega (\bar{p}_{\alpha,\Theta} + \bar{q}_{\alpha,\Theta}) \Big)(t,T,x,y)\\
&\le& (2C_{\ref{LEMME_IT_KER}})^2(T-t)^\omega  \Big((T-t)^{\omega}\bar{p}_{\alpha,\Theta} + (\bar{p}_{\alpha,\Theta} + \bar{q}_{\alpha,\Theta}) \Big)(t,T,x,y).
\end{eqnarray*}

\textbf{Induction}:\\
Suppose that the estimate for $2k$ holds. Let us prove the estimate for $2k+1$.
\begin{eqnarray*}
|\tilde{p}_\alpha \otimes H^{(2k+1)}|(t,T,x,y) &\leq& (4C_{\ref{LEMME_IT_KER}} )^{2k} (T-t)^{k\omega} \Big( (T-t)^{k\omega} |\bar{p}_{\alpha,\Theta} \otimes H| (t,T,x,y) + |(\bar{p}_{\alpha,\Theta} + \bar{q}_{\alpha,\Theta})\otimes H |(t,T,x,y) \Big)\\
&\le& (4C_{\ref{LEMME_IT_KER}})^{2k} (T-t)^{k\omega} \Big( C_{\ref{LEMME_IT_KER}}(T-t)^{k\omega} ( (T-t)^\omega \bar{p}_{\alpha,\Theta} +\bar{q}_{\alpha,\Theta} )(t,T,x,y)\\
&& + C_{\ref{LEMME_IT_KER}}((T-t)^\omega \bar{p}_{\alpha,\Theta} +\bar{q}_{\alpha,\Theta} )(t,T,x,y)  + C_{\ref{LEMME_IT_KER}}(T-t)^\omega ( \bar{p}_{\alpha,\Theta}+ \bar{q}_{\alpha,\Theta})(t,T,x,y) \Big).
\end{eqnarray*}
Recalling that $T-t\leq 1$, we have $(T-t)^{k \omega}\bar{q}_{\alpha,\Theta} \leq (T-t)^{\omega}\bar{q}_{\alpha,\Theta}$.
Thus:
\begin{eqnarray*}
|\tilde{p}_\alpha \otimes H^{(2k+1)}|(t,T,x,y) \\
\le (4C_{\ref{LEMME_IT_KER}})^{2k} (T-t)^{k\omega} \Big( C_{\ref{LEMME_IT_KER}}
(T-t)^{(k+1)\omega} \bar{p}_{\alpha,\Theta}  +2C_{\ref{LEMME_IT_KER}}(T-t)^\omega ( \bar{p}_{\alpha,\Theta} +\bar{q}_{\alpha,\Theta} ) 
+ C_{\ref{LEMME_IT_KER}}\bar{q}_{\alpha,\Theta})
\Big)(t,T,x,y)\\
\le (4C_{\ref{LEMME_IT_KER}})^{2k}(2C_{\ref{LEMME_IT_KER}}) (T-t)^{k\omega} \Big( 
(T-t)^{(k+1)\omega} \bar{p}_{\alpha,\Theta}  +(T-t)^\omega ( \bar{p}_{\alpha,\Theta} +\bar{q}_{\alpha,\Theta} )+\bar{q}_{\alpha,\Theta})
\Big)(t,T,x,y),
\end{eqnarray*}
which gives the announced estimate.\\
Suppose now that the estimate for $2k+1$ holds. Let us prove the estimate for $2k+2$.
\begin{eqnarray*}
|\tilde p_\alpha\otimes H^{(2k+2)}(t,T,x,y)|\\
\le (4C_{\ref{LEMME_IT_KER}})^{2k+1}(T-t)^{k\omega} \Big( (T-t)^{(k+1)\omega} |\bar{p}_{\alpha,\Theta} \otimes H |+ (T-t)^\omega |(\bar{p}_{\alpha,\Theta} + \bar{q}_{\alpha,\Theta})\otimes H|+ |\bar{q}_{\alpha,\Theta}\otimes H| \Big)(t,T,x,y)\\
\le (4C_{\ref{LEMME_IT_KER}})^{2k+1}(T-t)^{k\omega} \Big(  C_{\ref{LEMME_IT_KER}} (T-t)^{(k+1)\omega} [(T-t)^\omega \bar{p}_{\alpha,\Theta} +\bar{q}_{\alpha,\Theta}]\\
 +C_{\ref{LEMME_IT_KER}}(T-t)^\omega[ \{(T-t)^\omega \bar{p}_{\alpha,\Theta} + \bar{q}_{\alpha,\Theta}\} +  C_{\ref{LEMME_IT_KER}}(T-t)^\omega(\bar{p}_{\alpha,\Theta} + \bar{q}_{\alpha,\Theta}) ] +C_{\ref{LEMME_IT_KER}}(T-t)^\omega(\bar{p}_{\alpha,\Theta} + \bar{q}_{\alpha,\Theta})
 \Big)(t,T,x,y)\\
 \le (4C_{\ref{LEMME_IT_KER}})^{2k+2}(T-t)^{(k+1)\omega} \Big( (T-t)^{(k+1)\omega}\bar{p}_{\alpha,\Theta} + (\bar{p}_{\alpha,\Theta} + \bar{q}_{\alpha,\Theta}) \Big)(t,T,x,y),
\end{eqnarray*}
where to get to the last equation, we used the fact that $(T-t)^\omega \bar{p}_{\alpha,\Theta} \le \bar{p}_{\alpha,\Theta} $, and  $(T-t)^{k \omega}\bar{q}_{\alpha,\Theta} \leq\bar{q}_{\alpha,\Theta}$. 
\end{proof}

\begin{appendix}

\mysection{Proof of the diagonal lower bound for the frozen density.}

\label{DIAG_LB}

In this section we prove the diagonal lower bound for the frozen density.
Recall from Proposition \ref{LAB_EDFR}, that the frozen density $p_{\Lambda_s}$ writes for all $z\in \R^{nd} $ as:
\begin{eqnarray*}
p_{\Lambda_s}(z) &=&  \frac{\det ( \M_{s-t}  )^{-1}  }{(2 \pi) ^{nd}}\int_{\R^{nd}}  e^{-i\langle q, ( \M_{s-t}  )^{-1}  z \rangle}
  \exp \left( -(s-t) \int_{\R^{nd}} \{1-\cos(\langle q, \xi \rangle)\}  \nu_S(d\xi) \right) dq\\
 &=&  \frac{\det ( \T_{s-t}^\alpha  )^{-1}  }{(2 \pi) ^{nd}}\int_{\R^{nd}}  e^{-i\langle q, ( \T_{s-t}^\alpha  )^{-1}  z \rangle}
  \exp \left( -(s-t) \int_{\R^{nd}} \{1-\cos(\langle \frac{q}{(s-t)^{\frac 1\alpha}}, \xi \rangle)\}  \nu_S(d\xi) \right) dq.
\end{eqnarray*}
The complex exponential can be written as a cosine.
Denoting $\bar{x}$ the projection of $x\in \R^{nd}$ on the sphere, we change variable to the polar coordinates by setting $q = |q| \bar q$, where $(|q|,\bar q)\in \R^+ \times S^{nd-1}$.
Also, we take a parametrization of the sphere by setting $\bar q= (\theta, \phi) \in [0,\pi] \times S^{nd-2}$, along the axis defined by $(\T_{s-t}^\alpha)^{-1}  z$.
Set finally $\tau = \cos(\theta)$, the density writes:
\begin{eqnarray}\label{DENS_DIAG_DEV_PREAL}
p_{\Lambda_s}(z) &=&  \frac{\det ( \T_{s-t}^{\alpha}  )^{-1}  }{(2 \pi) ^{nd}}\int_0^{+\infty}d|q| |q|^{nd-1}\int_{-1}^1d\tau (1-\tau^2)^{\frac{nd-3}{2}} \int_{S^{nd-2}}d\phi   \nonumber \\
&& \qquad \cos( |q| | (\T^{\alpha}_{s-t}  )^{-1}  z | \tau)
  \exp \left( -(s-t) \int_{\R^{nd}}\{1-\cos(\langle\frac{ \bar q |q|}{(s-t)^{1/\alpha}}, \xi \rangle)\}  \nu_S(d\xi) \right) .
\end{eqnarray}
The idea is the following: since  $| (\T^{\alpha}_{s-t}  )^{-1}  z |$ is small, we can expand the cosine
and show that the first term is positive, giving the two-sided diagonal estimate. We focus on the diagonal lower bound.

\begin{proposition}[\textbf{Diagonal Lower bound}]\label{diagonal expansion}
For $K $ sufficiently small, there exists $C_K$ s.t. for all $z\in \R^{nd},\ |{ ( \T^{\alpha}_{s-t} ) }^{-1}  z|\le K$:
$$
 p_{\Lambda_s}(z) \ge C_K \det {  (\T^{\alpha}_{s-t})  }^{-1}.
$$
%
\end{proposition}

\begin{proof}
There is no difference with the non degenerate case for the diagonal expansion, see \cite{kolo:00}.
For small $|{ ( \T^{\alpha}_{s-t} ) }^{-1}  z|$, we use Taylor's formula to expand $\cos(|q| |{ ( \T^{\alpha}_{s-t} ) }^{-1}  z| \tau)$ in equation \eqref{DENS_DIAG_DEV_PREAL}:

\begin{eqnarray}
p_{\Lambda_s}(z)&=& \frac{\det {(  \T^{\alpha}_{s-t})  }^{-1} }{(2 \pi) ^{nd}}
 \int_0^{+\infty}  d|q| |q|^{nd-1}
 \int_{-1}^{1} d\tau (1-\tau^2)^{\frac{nd-3}{2}} \left( \sum_{k=0}^{N} \frac{(-1)^k}{(2k)!}|q|^{2k}|{ ( \T^{\alpha}_{s-t} ) }^{-1}  z|^{2k} \tau^{2k} + \tilde{R}_N(|{ ( \T^{\alpha}_{s-t} ) }^{-1}  z|)\right)\nonumber\\
&& \times \int_{S^{nd-2}} d\phi
  \exp \left( -  (s-t)\int_{\R^{nd}}\{ 1-\cos( \langle \frac{\bar{q}|q|}{(s-t)^{1/\alpha}}, \xi \rangle )\} \nu_S(d \xi) \right)\nonumber\\
 &=&  \frac{\det (\T^\alpha_{s-t})^{-1}}{(2 \pi) ^{nd}}
 \sum_{k=0}^{N} \frac{(-1)^k}{(2k)!} |{ ( \T^{\alpha}_{s-t} ) }^{-1}  z|^{2k}  \int_0^{+\infty}  d|q| |q|^{2k+nd-1}
 \int_{-1}^{1} d\tau (1-\tau^2)^{\frac{nd-3}{2}}\tau^{2k}\nonumber\\
&& \times \int_{S^{nd-2}} d\phi
  \exp \left( -  (s-t)\int_{\R^{nd}}\{ 1-\cos( \langle \frac{\bar{q}|q|}{(s-t)^{1/\alpha}}, \xi \rangle )\} \nu_S(d \xi) \right) + R_N(|{ ( \T^{\alpha}_{s-t} ) }^{-1}  z|).  \label{EXPANSION_DENS_DIAG}
\end{eqnarray}

The estimate on the coefficient also serves to estimate the remainder $R_N(|{ ( \T^{\alpha}_{s-t} ) }^{-1}  z|)$.
To bound the coefficient, we use the domination condition in \eqref{DOMI_NU} and the property that $g $ is non-increasing:
\begin{eqnarray*}
\exp \left(-(s-t)\int_{\R^{nd}}\{ 1-\cos( \langle \frac{\bar{q}|q|}{(s-t)^{1/\alpha}}, \xi \rangle )\} \nu_S(d \xi) \right)\\
\ge \exp\left(-(s-t)g(0)\int_{\R^+} \frac{d\rho}{\rho^{1+\alpha}}\int_{S^{nd-1}}\{ 1-\cos( \langle \frac{\bar{q}|q|}{(s-t)^{1/\alpha}}, \rho \eta \rangle )\bar \mu(d\eta)\}\right)\\
=\exp\left(-g(0)\int_{\R^+} \frac{d\rho}{\rho^{1+\alpha}}\int_{S^{nd-1}}  \{1-\cos(\langle\bar{q}|q|, \rho \eta \rangle)\} \bar \mu(d\eta)\}\right)
=\exp\left(-c_\alpha g(0)|q|^\alpha\int_{S^{nd-1}}|\langle \bar q,\eta \rangle |^\alpha\bar \mu(d\eta)\right)\\
\ge \exp\left(-\bar c |q|^\alpha \right), \bar c:=\bar c(\alpha,\H)\ge 1,
\end{eqnarray*}
using that $\bar \mu $ satisfies \textbf{[H-4]} for the last inequality. The above control can be used to give a lower bound 
for the even terms in the previous expansion \eqref{EXPANSION_DENS_DIAG}. On the other hand, similarly to the proof of Proposition \ref{LAB_EDFR}, we get
$$\exp \left(-(s-t)\int_{\R^{nd}}\{ 1-\cos( \langle \frac{\bar{q}|q|}{(s-t)^{1/\alpha}}, \xi \rangle )\} \nu_S(d \xi) \right)\le \exp\left(-\bar c^{-1} |q|^\alpha \right), $$
for $|q|>1 ,$
which can be used to derive lower bound for the odd terms of the expansion \eqref{EXPANSION_DENS_DIAG}.
Note that the coefficient $a_k( \overline{{ ( \T^{\alpha}_{s-t} ) }^{-1}  z} )$ depends on $ \overline{{ ( \T^{\alpha}_{s-t} ) }^{-1}  z}$ because of the choice of the parametrization of the sphere $S^{nd-2}$.

\end{proof}

\mysection{Off-diagonal Estimates on the Kernel $H$.}
\label{controles_phi}
We thoroughly exploit the decomposition of the density used by Watanabe \cite{wata:07} in the stable case followed by Sztonyk \cite{szto:10} in the tempered one.
From the identity \eqref{density_Lambda}, we have:
\begin{equation}
\label{SCALE_STAB}
\forall z\in \R^{nd},\ p_{\Lambda_T}(z) =\det(\M_{T-t}^{\alpha})^{-1} p_{S}(T-t, (\M_{s-t}^{\alpha})^{-1}z),
\end{equation}
where $ (S_u)_{u\ge 0}$ has L\'evy measure $\nu_{S} $.

For a fixed $T-t $ we can write $S_{T-t} = M_{T-t}+N_{T-t}$ where $(M_u)_{u\ge 0} $ and $(N_u)_{u\ge 0} $ are two independent processes with respective generators:
\begin{eqnarray*}
{\mathcal L}^M \varphi(z)&=&\int_{\R^{nd}}(\varphi(z+\xi)-\varphi(z)-\frac{\langle \nabla \varphi(z),\xi\rangle}{1+|\xi|^2} ) \I_{|\xi|\le (T-t)^{1/\alpha}}\nu_{S}(d\xi ),  \\
{\mathcal L}^N\varphi(z)&=&\int_{\R^{nd}}(\varphi(z+\xi)-\varphi(z)-\frac{\langle \nabla \varphi(z),\xi\rangle}{1+|\xi|^2} ) 
\I_{|\xi|> (T-t)^{1/\alpha}}
\nu_{S}(d\xi ),
\end{eqnarray*}
for all $z\in \R^{nd}$ and $\varphi \in C_0^2(\R^{nd},\R) $.
We have separated the jumps that are at the typical scales, i.e. $(T-t)^{1/\alpha} $, from the big ones which induce a compound Poisson process. It can be proved similarly to Proposition \ref{LAB_EDFR} that $M_{T-t} $ has a density, intuitively the \textit{small} jumps generate the density.
We therefore disintegrate $p_{S}(T-t,.)$ in the following way:
\begin{equation}
\label{DESINT}
\forall z\in \R^{nd},\ p_{S}(T-t,z) = \int_{\R^{nd}} p_M(T-t,z-\bar z) P_{N_{T-t}}(d\bar z),
\end{equation}
where $P_{N_{T-t}}$ stands for the law of $N_{T-t} $.
Now, the following properties hold for the L\'evy-It\^o decomposition. 
\begin{lemma}[Density estimate on the Martingale part and associated derivatives.]\label{EST_DENS_MART}
For all $m\ge 1$, there exists $C_m\ge 1$ s.t. for all $T-t> 0, z\in\R^{nd}$, 
$$
p_M(T-t,z) \le C_m (T-t)^{-nd/\alpha} \left( 1+ \frac{|z|}{(T-t)^{1/\alpha}}\right)^{-m}.
$$
Also, for all $m\ge 1 $ and all multi-index $\beta ,\ |\beta|\le 2$,
$$|\partial_z^\beta p_M(T-t,z)|\le C_m (T-t)^{-(nd+|\beta|)/\alpha} \left( 1+ \frac{|z|}{(T-t)^{1/\alpha}}\right)^{-m}.$$
\end{lemma}

\begin{proof}
Similarly to the proof of Proposition \ref{LAB_EDFR} we write:
\begin{eqnarray*}
p_M(T-t,z)= \frac{1}{(2\pi)^{nd}} \int_{\R^{nd}}dpe^{-i\langle p,z \rangle} \exp \left( - (T-t)\int_{\R^{nd}} \{1- \cos( \langle p, \xi \rangle)\}   \ind_{\{|\xi|\le (T-t)^{1/\alpha}\}} \nu_{S}(d\xi)\right).
\end{eqnarray*}
Changing variables in $(T-t)^{1/\alpha}p =q$ yields:
\begin{eqnarray} \label{EXP_DENS_M}
p_M(T-t,z)= \frac{1}{(2\pi)^{nd}} (T-t)^{-nd/\alpha}\int_{\R^{nd}}dqe^{-i\langle q,\frac{z}{(T-t)^{1/\alpha}} \rangle}\\
\times \exp \left(-(T-t)  \int_{\R^{nd}}\{ 1-\cos(\langle q,\frac{\xi}{(T-t)^{1/\alpha}}\rangle) \}  \ind_{\{|\xi|\le (T-t)^{1/\alpha}\}} \nu_{S}(d\xi)\right).
\end{eqnarray}
Let us now denote 
$$\hat{f}_{T-t}(q) :=\exp \left((T-t)  \int_{\R^{nd}} \{\cos( \langle  q , \frac{\xi}{(T-t)^{1/\alpha}} \rangle ) -1\}\I_{|\xi|\le (T-t)^{1/\alpha}}  \nu_{S}(d\xi)\right).$$
Since the L\'evy measure in the above expression has finite support, we get from  Theorem 3.7.13 in Jacob \cite{Jacob1} that $\hat f_{T-t}$ is infinitely differentiable as a function of $q$.
Moreover,  
\begin{eqnarray*}
|\partial_q \hat f_{T-t}(q)|&\le& (T-t)\int_{\R^{nd}} \frac{|\xi|}{(T-t)^{\frac 1 \alpha}}| \sin(\langle   q , \frac{\xi}{(T-t)^{1/\alpha}}\rangle)|\I_{|\xi|\le (T-t)^{\frac1 \alpha}} \nu_S(d\xi) \\
&&\times
\exp \left((T-t)  \int_{\R^{nd}} \{\cos( \langle  q , \frac{\xi}{(T-t)^{1/\alpha}} \rangle ) -1\}\I_{|\xi|\le (T-t)^{1/\alpha}}  \nu_{S}(d\xi)\right).
\end{eqnarray*}
Write now:
\begin{eqnarray*}
&&(T-t)\int_{\R^{nd}} \frac{|\xi|}{(T-t)^{\frac 1 \alpha}}| \sin(\langle   q , \frac{\xi}{(T-t)^{1/\alpha}}\rangle)|\I_{|\xi|\le (T-t)^{\frac1 \alpha}} \nu_S(d\xi)\\
&\le& C(T-t)\int_{r\le (T-t)^{1/\alpha}} dr \frac{r^{nd-1}}{r^{d+1+\alpha}}\frac{r}{(T-t)^{\frac 1 \alpha}}   (\I_{\alpha<1}+ \I_{\alpha\ge 1}|q| \frac{r}{(T-t)^{1/\alpha}})\\
&\le & C(T-t)\int_{r\le (T-t)^{1/\alpha}} dr \frac{r^{-\alpha}}{(T-t)^{\frac 1 \alpha}}   (\I_{\alpha<1}+ \I_{\alpha\ge 1}|q| \frac{r}{(T-t)^{1/\alpha}})\le C(1+|q|).
\end{eqnarray*}
Thus:
\begin{eqnarray*}
|\partial_q \hat f_{T-t}(q)|
&\le & C(1+|q|) \exp\left(  (T-t)\int_{\R^{nd}} \{\cos( \langle  q , \frac{\xi}{(T-t)^{1/\alpha}} \rangle)  -1\}\nu_{S}(d\xi)\right) \\ 
&&\times \exp(2(T-t)\nu_{S}(B(0,(T-t)^{1/\alpha})^c))
\le 
C (1+|q|)\exp(-C^{-1}|q|^\alpha),\ C\ge 1,
\end{eqnarray*}
since from \eqref{DOMI_NU}, $\nu_{S}(B(0,(T-t)^{1/\alpha})^c) \le C/(T-t)$ and that the proof of Proposition \ref{LAB_EDFR} also yields that 
$$\exp\left((T-t)\int_{\R^{nd}} \{\cos( \langle  q , \frac{\xi}{(T-t)^{1/\alpha}} \rangle)  -1\}\nu_{S}(d\xi)\right)\le C\exp(-C^{-1}|q|^\alpha).$$
Similarly, for all $l\in \N $: 
\begin{eqnarray*}
|\partial_q^l \hat f_{T-t}(q)|&\le& 
C_l (1+|q|^l)\exp(-C^{-1}|q|^\alpha),\ C_l\ge 1.
\end{eqnarray*}
Thus, $\hat f_{T-t}$ belongs the Schwartz space. 
Denoting by $f_{T-t}$ its Fourier transform, we have:
$$
\forall m \ge 0, \ \forall z \in \R^{nd}, \exists C_m\ge 1 s.t.: |f_{T-t}(z)| \le C_m (1+|z|)^{-m}.
$$
Now since $p_M(T-t,z) = (T-t)^{-nd/\alpha} f_{T-t}(z/(T-t)^{1/\alpha})$, the announced bound follows. The control concerning the derivatives is derived similarly.

\end{proof}

Besides, the following control holds for the Poisson measure.

\begin{lemma}[Controls for the Poisson measure]\label{EXP_POISSON}
For all $T-t> 0$, $P_{N_{T-t}}$ is a Poisson measure.
Since ${\rm dim}({\rm supp}(\bar \mu))=d $ we have the following estimates. There exists a constant $C>0$ s.t. For all $ z \in \R^{nd}, r>0 $:
\begin{equation}
\label{EST_POISSON}
P_{N_{T-t}}(B(z,r))\le \frac{C}{\theta ((T-t)^{1/\alpha})}(T-t)r^{d+1} (1+\frac{r^\alpha}{T-t}\frac{\theta((T-t)^{1/\alpha})}{ \theta(r)})|z|^{-(d+1+\alpha)}\theta(|z|).
\end{equation}
\end{lemma}

\begin{proof} In the stable case, i.e. $\theta=1 $,
this result is a consequence of Lemma 3.1 in \cite{wata:07} and the intrinsic stable scaling. In the tempered case, it follows from Corollary 6 in Sztonyk \cite{szto:10}.
\end{proof}

Let us observe that the above control also yields the upper-bound estimate for the density in Propositions \ref{EST_DENS_GEL}, \ref{EST_DENS_GEL_T} in the off-diagonal regime. Precisely from \eqref{SCALE_STAB}, \eqref{DESINT}, Lemma \ref{EST_DENS_MART} and \eqref{EST_POISSON} one gets:
\begin{eqnarray*}
\tilde p_\alpha(t,T,x,y)&\le& C_m\det(\M_{T-t})^{-1}(T-t)^{-nd/\alpha}\int_{\R^{nd}} (1+|\M_{T-t}^{-1}(R_{T,t}x-y) -\bar z|/(T-t)^{1/\alpha})^{-m} P_{N_{T-t}}(d\bar z)\\
&\le& C_m \det(\T_{T-t}^\alpha)^{-1}\int_{0}^1P_{N_{T-t}}(\{\bar z\in \R^{nd}:(1+(T-t)^{-1/\alpha}|\M_{T-t}^{-1}(R_{T,t}x-y)-\bar z|)^{-m} |>s  \})ds\\
&\le& C_m \det(\T_{T-t}^\alpha)^{-1}\int_{0}^1P_{N_{T-t}}(B(\M_{T-t}^{-1}(R_{T,t}x-y), s^{-1/m}(T-t)^{1/\alpha}))ds\\
&\le &C_m C \det(\T_{T-t}^\alpha)^{-1}  \frac{(T-t)^{1+(d+1)/\alpha}}{\theta((T-t)^{1/\alpha})}\int_0^1 s^{-(d+1)/m}
 (1+s^{-\alpha/m}\frac{\theta((T-t)^{1/\alpha})}{\theta((T-t)^{1/\alpha}s^{-1/m})} ) ds\\
&&\times|\M_{T-t}^{-1}(R_{T,t}x-y)|^{-(d+1+\alpha)}
\theta\left( |\M_{T-t}^{-1}(R_{T,t}x-y)|\right)\\
&\le & \frac{C_m}{\theta(1)}C \det(\T_{T-t}^\alpha)^{-1}(1+|(\T_{T-t}^{\alpha})^{-1}(R_{T,t}x-y)|)^{-(d+1+\alpha)}\theta\left( |\M_{T-t}^{-1}(R_{T,t}x-y)|\right)\\
&&\times \int_0^1 [s^{-(d+1)/m}+s^{-(d+1+\alpha+\tilde \eta)/m} ds],
\end{eqnarray*}
using for the last inequality that $\theta $ is non-increasing and exploiting  that the doubling condition in \textbf{[T]} is equivalent to the fact that there 
exists $c>0,\ \tilde \eta \ge 0$ s.t. $\frac{\theta(r)}{\theta(R)}\le c\left(\frac rR \right)^{-\tilde \eta},\ 0< r\le R $, see e.g. \cite{bass:95}.
Choosing  $m>d+1+\alpha+\tilde \eta $ then gives the result, i.e. there exists $C\ge 1$ s.t. for all $0\le t<T, \ (x,y)\in (\R^{nd})^2,\ \tilde p_\alpha(t,T,x,y)\le C\bar p_\alpha(t,T,x,y) $.

Moreover, the previous procedure, associated with Lemma
\ref{EST_DENS_MART}, allows to handle the small jumps in the estimation of $(L_t-\tilde L_t^{T,y})\tilde p_\alpha(t,T,x,y)=(L_t^M-\tilde L_t^{T,y,M})\tilde p_\alpha(t,T,x,y)+(L_t^N-\tilde L_t^N)\tilde p_\alpha(t,T,x,y)$.
Introducing $\nu(x,A):=\nu( \{z\in \R^d: \sigma(x)z\in A\}) $, we write for a given parameter $a\in (0,1\wedge K)$ and $x\in \R^{nd}$: 
\begin{eqnarray*}
(L_t^M-\tilde L_t^{T,y,M})\varphi(x)&=&\int_{\R^d} (\varphi(x+Bz)-\varphi(x)-\langle \nabla_{z^1} \varphi(x),z\rangle)\I_{|z|\le  a (T-t)^{1/\alpha}} 
(\nu(x,dz)-\nu(R_{t,T}y,dz)),\\
(L_t^N-\tilde L_t^{T,y,N})\varphi(x)&=&\int_{\R^d} (\varphi(x+Bz)-\varphi(x))\I_{|z|> a (T-t)^{1/\alpha}}(\nu(x,dz)-\nu(R_{t,T}y,dz)).
\end{eqnarray*}
The Lipschitz property of the density of the spectral measure $ \mu$ in \textbf{[H-4]}, the non-degeneracy and H\"older continuity of $\sigma $
and the properties concerning the tempering function in \textbf{[HT]} yield that $\nu(.,dz) $ is $\eta (\alpha \wedge 1) $ H\"older continuous w.r.t. its first parameter and there exists $ C\ge 1$ s.t. uniformly in $z\in \R^d$, $|(\nu(x,dz)-\nu(R_{t,T}y,dz))|\le C(\delta \wedge  |x-R_{t,T}y|^{\eta (\alpha \wedge 1)})\theta(|z|) |z|^{-(d+\alpha)} dz$. The condition that for all $r>0,\  r\sup_{ u \in [\kappa^{-1},\kappa]} g'(u r)\le c \theta(r)$ appearing in \textbf{[HT]} is needed here to control the difference on the tempering functions. 
We now get:
\begin{eqnarray*}
|(L_t^M-\tilde L_t^{T,y,M})\tilde p_\alpha(t,T,x,y)|=\\
\bigg|\int_{\R^d}\big(\tilde p_\alpha(t,T,x+Bz,y)-\tilde p_\alpha(t,T,x,y)-\langle \nabla_{x^1}p_\alpha(t,T,x,y),z\rangle \big) ((\nu(x,dz)-\nu(R_{t,T}y,dz)))\bigg|\\
\le  C \det(\M_{T-t})^{-1}[\delta \wedge |x-R_{t,T}y|^{\eta (\alpha \wedge 1)}] \int_{\R^d} \big |\int_{\R^{nd}} \big\{ p_M(T-t, \M_{T-t}^{-1}(R_{T,t}x+Bz-y) -\bar z)\\
-p_M(T-t, \M_{T-t}^{-1}(R_{T,t}x-y) 
-\bar z)- \langle \nabla_{x^1} p_M(T-t,\M_{T-t}^{-1}(R_{T,t}x-y)-\bar z ), z\rangle\I_{\alpha\ge 1}  \big\} \\
P_{N_{T-t}}( d\bar z)\big|\I_{|z|\le a(T-t)^{1/\alpha}} \theta(|z|)\frac{dz}{|z|^{d+\alpha}}.
\end{eqnarray*}
The idea is now to perform a Taylor expansion on $p_M $ to compensate the singularities in $z$. We assume for simplicity that $\alpha \in (0,1)$ which allows to perform the Taylor expansion at order $1$ only. It suffices to expand at order 2 to handle the case $\alpha  \in [1,2) $. We get from Lemma \ref{EST_DENS_MART}:
\begin{eqnarray}
|(L_t^M-\tilde L_t^{T,y,M})\tilde p_\alpha(t,T,x,y)|\le C\det(\M_{T-t})^{-1}[\delta \wedge |x-R_{t,T}y|^{\eta (\alpha \wedge 1)}] \nonumber \\
\int_{\R^d} \int_{\R^{nd}} \big\{ \sup_{|\tilde z|\in (0,a(T-t)^{1/\alpha}]}| \nabla_{x_1} p_M(T-t, \M_{T-t}^{-1}(R_{T,t}x+B\tilde z-y) -\bar z)|    |z| \big\}
  P_{N_{T-t}}(d\bar z)\I_{|z|\le a(T-t)^{1/\alpha}} \frac{dz}{|z|^{d+\alpha}}\nonumber \\
  \le \frac{CC_m \det(\T_{T-t}^\alpha)^{-1}}{(T-t)^{1/\alpha}}[\delta \wedge |x-R_{t,T}y|^{\eta (\alpha \wedge 1)}] \nonumber \\
  \int_{\R^d} \int_{\R^{nd}}  \sup_{|\tilde z|\in (0,a(T-t)^{1/\alpha}]} \left(1+\frac{|\M_{T-t}^{-1}(R_{T,t}x+B\tilde z-y)-\bar z )|}{(T-t)^{1/\alpha}}\right)^{-m}P_{N_{T-t}}(d \bar z)\I_{|z|\le a(T-t)^{1/\alpha}} |z|\frac{dz}{|z|^{d+\alpha}} \nonumber\\
 \le  \frac{CC_m\det(\T_{T-t}^\alpha)^{-1}}{(T-t)^{1/\alpha}}[\delta \wedge |x-R_{t,T}y|^{\eta (\alpha \wedge 1)}]\int_{0}^{at^{1/\alpha}} dr r^{-\alpha}\nonumber\\
 \times \int_{\R^{nd}}   \left( (1-a)+\frac{|\M_{T-t}^{-1}(R_{T,t}x-y)-\bar z )|}{(T-t)^{1/\alpha}}\right)^{-m}P_{N_{T-t}}(d\bar z)\nonumber \\
 \le \frac{C[\delta \wedge |x-R_{t,T}y|^{\eta (\alpha \wedge 1)}]}{T-t} \bar p_\alpha(t,T,x,y).\label{CTR_H_PETITS_SAUTS}
\end{eqnarray}
This therefore gives the expected control for the small jumps in the kernel, i.e. the operator $L_t^M-\tilde L_t^{T,y,M} $ acting on $\tilde p_\alpha(t,T,x,y) $ yields a bound homogeneous to the upper-bound $\bar p_\alpha(t,T,x,y) $ up to an additional multiplicative singularity of the form $\frac{C[\delta \wedge |x-R_{t,T}y|^{\eta  (\alpha \wedge 1)}]}{T-t}$.\\

The delicate part,  yielding the \textit{rediagonalization} phenomenon which might deteriorate the estimates in the degenerate framework, comes from the large jumps.  
We now specify how in the off-diagonal regime, when
$|(x-R_{t,T}y)^1|/(T-t)^{1/\alpha}\asymp |(\T_{T-t}^{\alpha})^{-1}(x-R_{t,T}y)| \overset{{\rm Lemma} \ \ref{ScalingLemma}}{\asymp} |(\T_{T-t}^{\alpha})^{-1}(R_{T,t}x-y)|$
, that is when the \textit{slow} component dominates, a \textit{bad} rediagonalization phenomenon can occur.
Let us now discuss the various possible cases. Fix $\varepsilon>0 $. 
\begin{trivlist}
\item[-] If $z\not \in B((R_{t,T}y-x)^1,\varepsilon |(x-R_{t,T}y)^1|):=B_{\varepsilon,t,T,x,y}$ then
 $|(\T_{T-t}^\alpha)^{-1} (x+Bz- R_{t,T}y)|\ge (|z-(R_{t,T}y-x)^1|)/(T-t)^{1/\alpha}\ge \varepsilon |(x-R_{t,T}y)^1|/(T-t)^{1/\alpha} $. Hence $\tilde p_\alpha(t,T,x+Bz,y) $ is off-diagonal and $ \tilde p_\alpha(t,T,x+Bz,y)\le C\bar p_\alpha(t,T,x,y)$. Thus:
\begin{eqnarray}
\label{CTR_LOIN_BOULE}
\int_{z\not \in 
B_{\varepsilon,t,T,x,y}} |\{\tilde p_\alpha(t,T,x+Bz,y)-\tilde p_\alpha(t,T,x,y)\}| \I_{|z|> a (T-t)^{1/\alpha}} |\nu(x,dz)-\nu(R_{t,T}y,dz)|\nonumber \\
\le C [\delta\wedge |x-R_{t,T}y|^{\eta (\alpha \wedge 1)}] \bar p_\alpha(t,T,x,y)\int_{\R^d}\I_{|z|>a (T-t)^{1/\alpha}} \frac{dz}{|z|^{d+\alpha}}\nonumber\\
\le \frac{C[\delta\wedge |x-R_{t,T}y|^{\eta (\alpha \wedge 1)}]}{T-t}\bar p_\alpha (t,T,x,y).
\end{eqnarray}
\item[-] If $z\in 
B_{\varepsilon,t,T,x,y}$ we can write:
\begin{eqnarray}
\int_{z\in B_{\varepsilon,t,T,x,y}}|\tilde p_\alpha(t,T,x+Bz,y)-\tilde p_\alpha(t,T,x,y)| \I_{|z|> a (T-t)^{1/\alpha}} |\nu(x,dz)-\nu(R_{t,T}y,dz)| \nonumber\\
\le C[\delta\wedge |x-R_{t,T}y|^{\eta (\alpha \wedge 1)}]  \left\{ \frac{\theta(|(x-R_{t,T}y)^1|)}{|(x-R_{t,T}y)^1|^{d+\alpha}}\int_{z\in B_{\varepsilon,t,T,x,y}} \tilde p_\alpha(t,T,x+Bz,y)dz +\frac{\bar p_\alpha(t,T,x,y)}{T-t}\right\}\nonumber\\
\le C[\delta\wedge |x-R_{t,T}y|^{\eta (\alpha \wedge 1)}]  \bigg\{ \frac{\bar p_\alpha(t,T,x,y)}{T-t} \nonumber\\
+ \frac{ \theta(|(x-R_{t,T}y)^1|)}{|(x-R_{t,T}y)^1|^{d+\alpha}}\int_{z\in B_{\varepsilon,t,T,x,y}} \det(\T_{T-t}^{\alpha})^{-1} (1+|(\T_{T-t}^{\alpha})^{-1}(x+Bz-R_{t,T}y)|)^{-(d+1+\alpha)}dz \bigg\}\label{REDIAG_INT}\\
\le C[\delta\wedge |x-R_{t,T}y|^{\eta (\alpha \wedge 1)}]  \bigg\{ \frac{\bar p_\alpha(t,T,x,y)}{T-t}\nonumber \\
+\frac{\theta(|(x-R_{t,T}y)^1|)}{|(x-R_{t,T}y)^1|^{d+\alpha}}\frac{1}{(T-t)^{\frac{(n-1)d}\alpha+\frac{n(n-1)d}2} (1+|\{(\T_{T-t}^{\alpha})^{-1}(x-R_{t,T}y)\}^{2:n}|)^{1+\alpha}} \bigg\}\nonumber\\
\le \frac{C}{T-t}[\delta\wedge |x-R_{t,T}y|^{\eta \alpha \wedge 1}]   (\bar p_\alpha+\breve{p}_\alpha  )(t,T,x,y),\label{CTR_REDIAG}
\end{eqnarray}
using Propositions \ref{EST_DENS_GEL}, \ref{EST_DENS_GEL_T}  and Lemma \ref{ScalingLemma} for the last but second inequality.

From \eqref{CTR_LOIN_BOULE} and \eqref{CTR_REDIAG} we derive:
$$|(L_t^N-\tilde L_t^{T,y,N})\tilde p_\alpha(t,T,x,y)|\le \frac{C}{T-t}[\delta\wedge |x-R_{t,T}y|^{\eta (\alpha \wedge 1)}]   (\bar p_\alpha+\breve{p}_\alpha  )(t,T,x,y), $$
which together with \eqref{CTR_H_PETITS_SAUTS} gives the statement of Lemma \ref{CTR_KER_PTW} in the off-diagonal regime.

\begin{remark}[About the rediagonalization]
Observe that a similar rediagonalization phenomenon occurs in the non-degenerate case as well. The fact is that, in that case we integrate a density in \eqref{REDIAG_INT} and not a marginal. The 
decay of the jump measure
gives in that case up to a multiplicative singularity in $(T-t)^{-1}$ the asymptotic behavior of the stable density. Namely when $n=1 $ we would have $d+\alpha $ instead of $d+1+\alpha $ in \eqref{REDIAG_INT} and in the off-diagonal regime:
$|x-R_{t,T}y|^{-(d+\alpha)}= \frac{1}{T-t}\times \frac{T-t}{|x-R_{t,T}y|^{d+\alpha}}\le \frac{C}{T-t}\frac{1}{(T-t)^{d/\alpha}\left(1+\frac{|x-R_{t,T}y|}{(T-t)^{1/\alpha}}\right)^{d+\alpha}}:=\frac{C}{T-t}\bar p_\alpha(t,T,x,y)$, where $\bar p_\alpha $ indeed corresponds to the upper bound for the large scale asymptotics of a stable process whose spectral measure is absolutely continuous, see again Proposition \ref{EST_DENS_GEL}. In that framework, our proof provides an alternative to the Fourier arguments employed in \cite{kolo:00}.
\end{remark}

\begin{remark}[Loss of Concentration in the stable case]
From equations \eqref{REDIAG_INT}-\eqref{CTR_REDIAG} we see that when $|\{(\T_{T-t}^{\alpha})^{-1}(x-R_{t,T}y)\}^{2:n}|\le K $, i.e. the fast component in the backward dynamics are diagonal, we have a loss of concentration w.r.t. to the worst asymptotic bounds given in Proposition \ref{EST_DENS_GEL}. Note also that in this case the lower bound in that proposition yields:
\begin{eqnarray*}
\left| \int_{\R^d} \tilde p(t,T,x+Bz,y)-\tilde p(t,T,x,y) \I_{|z|> a (T-t)^{1/\alpha}}\frac{dz}{|z|^{d+\alpha}}\right| \ge  -\frac{C}{(T-t)}\bar p_\alpha(t,T,x,y)\\
+\frac{C^{-1}}{|(x-R_{t,T}y)^1|^{d+\alpha}}\frac{1}{(T-t)^{\frac{(n-1)d}\alpha+\frac{n(n-1)d}2} (1+|\{(\T_{T-t}^{\alpha})^{-1}(x-R_{t,T}y)\}^{2:n}|)^{nd(1+\alpha)-d}}\\
\ge \frac{1}{(T-t)^{\frac{(n-1)d}{\alpha}+\frac{n(n-1)d}{2}  } }\big\{\frac{C^{-1}}{|(x-R_{t,T}y)^1|^{d+\alpha}(1+K)^{nd(1+\alpha)-d}}-\frac{C}{|(x-R_{t,T}y)^1|^{d+\alpha+1}}(T-t)^{1/\alpha}\big\}\\
\ge \frac{1}{2(T-t)^{\frac{(n-1)d}{\alpha}+\frac{n(n-1)d}{2}  } }\frac{C^{-1}}{|(x-R_{t,T}y)^1|^{d+\alpha}(1+K)^{nd(1+\alpha)-d}},
 \end{eqnarray*}
if $|(x-R_{t,T}y)^1|\ge \bar K(T-t)^{1/\alpha} $ for $\bar K$  large enough. Hence, if $d=1$ the previous bound is sharp provided $\sigma(t,x)-\sigma(t,R_{t,T}y)\ge \delta >0$. 
\end{remark}
\end{trivlist}

\section{Steps for the proof of Theorem \ref{MART_PB_THM} in the nonlinear case.}
\label{NON_LIN_MART}
We specify in this section how to modify the previous arguments to prove the well posedness of the martingale problem  for the generator of \eqref{NON_LIN_SDE}. We focus on the case $n=2 $. The constraint $d=1$ will appear clearly during the proof. 
The first step consists in choosing a suitable parametrix. This is done similarly to the Gaussian case in \cite{D&M}. Namely, we introduce for given $(T,y)\in \R^+\times \R^d$ and $(t,x)\in [0,T)\times \R^d $ the \textit{frozen process}:
\begin{eqnarray}
(\tilde X_s^{t,x,T,y})^1&=&x^1+\int_t^s  F_1(u,\phi_{u,T}(y))du +\int_t^s \sigma(u,\phi_{u,T}(y))dZ_u,\nonumber\\
(\tilde X_s^{t,x,T,y})^2&=&x^2+\int_t^s  \left\{F_2(u,\phi_{u,T}(y))+ \nabla_{x^1} F_2(u,\phi_{u,T}(y))\left((\tilde X_u^{t,x,T,y})^1-\phi_{u,T}(y)^1\right)\right\}du,\label{DYN_NON_LIN}
\end{eqnarray}
denoting $\phi_{u,T}(y) $ the solution to $\phi_{T,T}(y)=y,\ \dot \phi_{u,T}(y)=F(u,\phi_{u,T}(y)) $, i.e. backward flow associated with the deterministic differential system.
This is a linear dynamics which once integrated through the resolvent yields:
\begin{eqnarray}
\label{DYN_NL_INT}
\tilde X_s^{t,T,x,y}=\tilde \phi_{s,t}^{T,y}(x)+\int_{t}^s \tilde R_{s,u}^{T,y}B\sigma(u,\phi_{u,T}(y))dZ_u,
\end{eqnarray}
where recalling $F(t,x)=(F_1(t,x),F_2(t,x))^*$:
\begin{eqnarray*}
\tilde \phi_{s,t}^{T,y}(x)&=&x+\int_{t}^s F(u,\phi_{u,T}(y))du+\int_{t}^s \left(\begin{array}{cc}0 & 0\\
\nabla_{x^1} F_2(u,\phi_{u,T}(y)) & 0 \end{array}
\right) \left(\tilde \phi_{u,t}^{T,y}(x)-\phi_{u,T}(y)\right)du\\
&=&\tilde R_{s,t}^{T,y}(x)+\int_t^s \tilde R_{s,u}^{T,y}\left\{F(u,\phi_{u,T}(y)) -\left(\begin{array}{cc}0 & 0\\
\nabla_{x^1} F_2(u,\phi_{u,T}(y)) & 0 \end{array}
\right)\phi_{u,T}(y)\right\}du,\\
\partial_s \tilde R_{s,t}^{T,y}&=&\left(\begin{array}{cc}0 & 0\\
\nabla_{x^1} F_2(s,\phi_{s,T}(y)) & 0 \end{array}
\right)\partial_s \tilde R_{s,t}^{T,y},\ \tilde R_{t,t}^{T,y}=I_{2\times 2}.
\end{eqnarray*}
It can be shown, similarly to Section 2.3 in \cite{D&M} that:
\begin{eqnarray}
\label{equiv_FLOWS_NL}
|(\T_{T-t}^{\alpha})^{-1}(\tilde \phi_{T,t}^{T,y}(x)-y)|\asymp |(\T_{T-t}^{\alpha})^{-1}(x-\phi_{t,T}(y))| \asymp |(\T_{T-t}^{\alpha})^{-1}(\phi_{T,t}(x)-y)|.
\end{eqnarray}
It can now be derived from equations \eqref{DYN_NL_INT}, \eqref{equiv_FLOWS_NL} similarly to the proof of Proposition \ref{LAB_EDFR}
that 
\begin{eqnarray*}
\tilde p_\alpha(t,T,x,y)&\le &\frac{C}{(T-t)^{(1+\frac{2}{\alpha})d}}\frac{1}{(1+|(\T_{T-t}^\alpha)^{-1}(\tilde \phi_{T,t}^{T,y}(x)-y)|)^{d+1+\alpha}}\\
&\le &\frac{C}{(T-t)^{(1+\frac{2}{\alpha})d}}\frac{1}{(1+|(\T_{T-t}^\alpha)^{-1}(x-\phi_{t,T}(y))|)^{d+1+\alpha}}\\
&:=&C\bar p_{\alpha,\phi}(t,T,x,y).
\end{eqnarray*}
From the desintegration of the density  in Appendix \ref{controles_phi} (see equation \eqref{DESINT}, Lemma \ref{EST_DENS_MART} and estimate \eqref{EST_POISSON}) we also derive the \textit{global} gradient bounds:
\begin{eqnarray*}
|\partial_{x^1} \tilde p_\alpha(t,T,x,y)|\le  \frac{C}{(T-t)^{\frac1\alpha} 
}\bar p_{\alpha,\phi}(t,T,x,y),\ 
|\partial_{x^2}\tilde p_\alpha(t,T,x,y)|\le  \frac{C}{(T-t)^{1+\frac 1\alpha} 
}\bar p_{\alpha,\phi}(t,T,x,y).\\
\end{eqnarray*}
On the other hand, the control of Lemma \ref{CTR_KER_PTW} concerning the kernel $H$ now writes:
\begin{eqnarray}
\label{CTR_H_NL}
|H(t,T,x,y)|\le C\frac{\delta \wedge |x-\phi_{t,T}(y)|^{\eta(\alpha\wedge 1)}}{T-t} \left\{\bar p_{\alpha,\phi}(t,T,x,y)+\breve{p}_{\alpha,\phi}(t,T,x,y) \right\}
+\nonumber
\\C\left\{ \frac{
|x-\phi_{t,T}(y)|}{(T-t)^{\frac1\alpha} 
}+
\left[\frac{
|(x-\phi_{t,T}(y))^2|}{(T-t)^{1+\frac1\alpha} 
}+\frac{
|(x-\phi_{t,T}(y))^1|(\delta \wedge |(x-\phi_{t,T}(y))^1|^\eta)}{(T-t)^{1+\frac 1\alpha}}\right]
\right\}\bar p_{\alpha,\phi}(t,T,x,y), 
\end{eqnarray}
where
\begin{eqnarray*}
\breve{p}_{\alpha,\phi}(t,T,x,y)=\frac{\I_{|(x-\phi_{t,T}(y))^1|/(T-t)^{1/\alpha}| \asymp |(\T_{T-t}^{\alpha})^{-1}(x-\phi_{t,T}(y))|\ge K}}{(T-t)^{d/\alpha}(1+\frac{|(x-\phi_{t,T}(y))^1|}{(T-t)^{1/\alpha}})^{d+\alpha}}\\
\times \frac{1}{(T-t)^{\frac{(n-1)d}\alpha+\frac{n(n-1)d}2}(1+|((\T_{T-t}^{\alpha})^{-1}(x-\phi_{t,T}(y))^{2:n}|)^{1+\alpha}} 
).
\end{eqnarray*}
We emphasize that in the above controls on $\tilde p_\alpha,\nabla \tilde p_\alpha, H $, we have bounded the tempering function, appearing in the case {\HT},  by a constant. Indeed, this term is not useful to investigate the martingale problem. Also, 
the additional contribution in $H$ coming from the gradient term, which vanishes in the linear case, is derived writing:
\begin{eqnarray*}
|\langle F(t,x)-(F(t,\phi_{t,T}(y))+\left(\begin{array}{cc} 0& 0\\ \nabla_{x^1}F_2(t, \phi_{t,T}(y))& 0
\end{array}\right)
(x-\phi_{t,T}(y)), \nabla \tilde p_{\alpha}(t,T,x,y)\rangle| 
\\\le C\left\{ \frac{
|x-\phi_{t,T}(y)|}{(T-t)^{\frac1\alpha} 
}+
\left[\frac{
|(x-\phi_{t,T}(y))^2|}{(T-t)^{1+\frac1\alpha} 
}+\frac{
|(x-\phi_{t,T}(y))^1|(\delta \wedge |(x-\phi_{t,T}(y))^1|^\eta)}{(T-t)^{1+\frac 1\alpha}}\right]
\right\}\bar p_{\alpha,\phi}(t,T,x,y). 
\end{eqnarray*}
The contributions in \eqref{CTR_H_NL} coming from the non-local part of $H$ can be analyzed as previously. Let us now focus on the term
$$\frac{
|(x-\phi_{t,T}(y))^1|(\delta \wedge |(x-\phi_{t,T}(y))^1|^\eta)}{(T-t)^{1+\frac 1\alpha}}
\bar p_{\alpha,\phi}(t,T,x,y):=G(t,T,x,y),$$
which is the trickiest among the new contributions. 
Indeed, it involves the first component, which has typical scale in $(T-t)^{1/\alpha} $, renormalized by the singularity deriving from the sensitivity w.r.t. the second one, i.e. $(T-t)^{-(1+1/\alpha)} $. Following the proof of Lemma \ref{lemme} we write:
\begin{eqnarray*}
G(t,T,x,y)&\le & C\int_{\R^{2d}}  \frac { (T-t)^{\frac 1\alpha}|Z^1|(\delta \wedge [(T-t)^{\frac 1\alpha}|Z^1|]^\eta)}{(T-t)^{1+\frac 1 \alpha}} \frac{dZ}{(1+|Z|)^{d+1+\alpha}}\\
&\le & \frac{C}{T-t}\int_{\R^d} |Z^1|(\delta \wedge [(T-t)^{\frac 1\alpha}|Z^1|^\eta]))\frac{dZ^1}{(1+|Z^1|)^{1+\alpha}}\\
&\le & C\{(T-t)^{\frac{\eta}\alpha-1}+\frac{1}{T-t}\int_{|Z^1|>K} (\delta^{1/\eta} \wedge [(T-t)^{\frac 1\alpha}|Z^1|])^\eta)\frac{dZ^1}{|Z^1|^{\alpha}}\}\\
&\le & C\{(T-t)^{\frac{\eta}\alpha-1}+\frac{1}{T-t}\int_{|Z^1|>K} (\delta^{1/\eta} \wedge [(T-t)^{\frac 1\alpha}|Z^1|])^\varepsilon)\frac{dZ^1}{|Z^1|^{\alpha}}\},
\end{eqnarray*}
for any $\varepsilon\in [0,\eta] $.
Now, the above integral only converges if $d=1,\alpha>1$ and $\alpha-\varepsilon>1 $ giving 
$$ G(t,T,x,y)\le  C \{ (T-t)^{\frac{\eta}\alpha-1}+(T-t)^{\frac{\varepsilon}\alpha-1}\},$$
which, once integrated in time yields the needed smoothing effect. 

We conclude saying that it seems anyhow difficult to consider this case, i.e. a fully non linear unbounded drift, for the density estimate, since the additional contribution in \eqref{CTR_H_NL} gives non-integrable singularities which we cannot here compensate as for the diffusion with a dependence on the fast variable.  

However, it can be proved under {\HT} for $d=1,n=2,\alpha>1 $  and $F(t,x)=(F_1(t,x), \alpha_t x_1+\tilde F_2(t,x^2))^* $ where the coefficients are bounded measurable in time and s.t. $\alpha_t\in [c_0,c_0^{-1}],\ c_0\in (0,1] $, and $F_1,\tilde F_2 $ are Lipschitz continuous in space, that the density exists and that the estimates of Theorem \ref{MTHM} hold with the non-linear flow. In that case, the most singular term is linear and vanishes in $H$. Assumption \textbf{[HT]} is here crucial and would give instead of the previous control \eqref{CTR_H_NL} that:
\begin{eqnarray*}
|H(t,T,x,y)|\le C\bigg[\frac{\delta \wedge |x-\phi_{t,T}(y)|^{\eta(\alpha\wedge 1)}}{T-t} \left\{\bar p_{\alpha,\phi}(t,T,x,y)+\breve{p}_{\alpha,\phi}(t,T,x,y) \right\}
+\nonumber
\\\left\{ \frac{
|x-\phi_{t,T}(y)|}{(T-t)^{\frac1\alpha} 
}+
\frac{
|(x-\phi_{t,T}(y))^2|}{(T-t)^{1+\frac1\alpha} 
}\right\}\bar p_{\alpha,\phi}(t,T,x,y)\bigg]\theta({|\M_{T-t}^{-1}(x-\phi_{t,T}(y))|}).
\end{eqnarray*}
The first contribution can be analyzed as previously, see Section \ref{SECTION_TECH}. On the other hand the definition of  $\Theta(r):=(1+r)\theta(r),r >0 $ gives:
$$\left\{\frac{
|x-\phi_{t,T}(y)|}{(T-t)^{\frac1\alpha} 
}+
\frac{
|(x-\phi_{t,T}(y))^2|}{(T-t)^{1+\frac1\alpha} 
}\right\}\bar p_{\alpha,\phi}(t,T,x,y)\theta({|\M_{t,T}^{-1}(x-\phi_{t,T}(y))|})\le \frac{C}{(T-t)^{\frac 1 \alpha}}\bar p_{\alpha,\phi,\Theta}(t,T,x,y), $$
where $\bar p_{\alpha,\phi,\Theta}(t,T,x,y):= \bar p_{\alpha,\phi}(t,T,x,y)
 \Theta({|\M_{t,T}^{-1}(x-\phi_{t,T}(y))|})$.
\end{appendix}

\bibliography{MyLibrary.bib}

\begin{thebibliography}{10}

\bibitem{aron:67}
D.~G. Aronson.
\newblock Bounds for the fundamental solution of a parabolic equation.
\newblock {\em Bull. Amer. Math. Soc.}, 73:890--896, 1967.

\bibitem{baru:poli:vesp:01}
E.~Barucci, S.~Polidoro, and V.~Vespri.
\newblock Some results on partial differential equations and asian options.
\newblock {\em Math. Models Methods Appl. Sci}, 3:475--497, 2001.

\bibitem{bass:95}
R.~F. Bass.
\newblock {\em Probabilistic techniques in {A}nalysis}.
\newblock Springer, 1995.

\bibitem{bass_perkins_martingale}
R.F. Bass and E.A. Perkins.
\newblock A new technique for proving uniqueness for martingale problems.
\newblock {\em From Probability to Geometry (I): Volume in Honor of the 60th
  Birthday of Jean-Michel Bismut}, pages 47--53, 2009.

\bibitem{bena:lean:91}
G.~Ben~Arous and R.~L{\'e}andre.
\newblock D\'ecroissance exponentielle du noyau de la chaleur sur la diagonale.
  {II}.
\newblock {\em Probab. Theory Related Fields}, 90(3):377--402, 1991.

\bibitem{BGJ}
K.~Bichteler, J.~B. Gravereaux, and J.~Jacod.
\newblock {\em Malliavin calculus for processes with jumps}.
\newblock Gordon and Breach Science Publishers, January 1987.

\bibitem{cass:09}
T.~Cass.
\newblock Smooth densities for solutions to stochastic differential equations
  with jumps.
\newblock {\em Stochastic Process. Appl.}, 119(5):1416--1435, 2009.

\bibitem{chen:kim:kuma:08}
Z-Q. Chen, P.~Kim, and T.~Kumagai.
\newblock Weighted poincar{\'e} inequality and heat kernel estimates for finite
  range jump processes.
\newblock {\em Mathematische Annalen}, 342(4):833--883, 2008.

\bibitem{cint:meno:poli:12}
C.~Cinti, S.~Menozzi, and S.~Polidoro.
\newblock Two-sided bounds for degenerate processes with densities supported in
  subsets of ${{\mathbb R}}^n$.
\newblock {\em Potential Analysis}, 42-1:39--98, 2015.

\bibitem{D&M}
F.~Delarue and S.~Menozzi.
\newblock Density estimates for a random noise propagating through a chain of
  differential equations.
\newblock {\em Journal of Functional Analysis}, 259(6):1577--1630, September
  2010.

\bibitem{eckm:99}
J.-P. Eckmann, C.-A. Pillet, and L.~Rey-Bellet.
\newblock Non-equilibrium statistical mechanics of anharmonic chains coupled to
  two heat baths at different temperatures.
\newblock {\em Comm. Math. Phys.}, 201--3:657--697, 1999.

\bibitem{azen:etal:81}
R.~Azencott et~al.
\newblock {\em G\'eod\'esiques et diffusions en temps petit}, volume 1984-1985.
\newblock Ast\'erisque, S.M.F., 1981.

\bibitem{fran:12}
J.~Franchi.
\newblock Small time asymptotics for an example of strictly hypoelliptic heat
  kernel.
\newblock {\em To appear in S\'eminaire de Probabilit\'es}, 2012.

\bibitem{frie:64}
A.~Friedman.
\newblock {\em Partial differential equations of parabolic type}.
\newblock Prentice-Hall, 1964.

\bibitem{ishi:kuni:06}
Y.~Ishikawa and H.~Kunita.
\newblock Malliavin calculus on the {W}iener-{P}oisson space and its
  application to canonical {SDE} with jumps.
\newblock {\em Stochastic Process. Appl.}, 116(12):1743--1769, 2006.

\bibitem{jacob}
N.~Jacob.
\newblock {\em Pseudo Differential Operators and Markov Process}.
\newblock Akademie Verlag, 1996.

\bibitem{Jacob1}
N.~Jacob.
\newblock {\em Pseudo differential operators and Markov processes}, volume~1.
\newblock Imperial College Press, 2005.

\bibitem{jean:yor:ches:09}
M.~Jeanblanc, M.~Yor, and M.~Chesney.
\newblock {\em Mathematical methods for financial markets}.
\newblock Springer Finance. Springer-Verlag London Ltd., London, 2009.

\bibitem{kolo:mono:00}
V.~Kolokoltsov.
\newblock {\em Semiclassical {A}nalysis for {D}iffusions and {S}tochastic
  {P}rocesses}, volume LMN, 1724.
\newblock Springer, 2000.

\bibitem{kolo:00}
V.~Kolokoltsov.
\newblock Symmetric stable laws and stable-like diffusion.
\newblock {\em Proceedings of the London Mathematical Society},
  80(03):725--768, 2000.

\bibitem{K&M}
V.~Konakov and S.~Menozzi.
\newblock Weak error for stable driven sdes: expansion for the densities.
\newblock {\em Journal of Theoretical Probability}, 24(2):454--478, 2010.

\bibitem{kona:meno:molc:10}
V.~Konakov, S.~Menozzi, and S.~Molchanov.
\newblock Explicit parametrix and local limit theorems for some degenerate
  diffusion processes.
\newblock {\em Annales de l'Institut Henri Poincar\'e, S\'erie B},
  46--4:908--923, 2010.

\bibitem{kusu:stro:84}
S.~Kusuoka and D.~Stroock.
\newblock Applications of the {M}alliavin calculus. {I}.
\newblock {\em Stochastic analysis (Katata/Kyoto, 1982)}, North-Holland Math.
  Library, 32:271--306, 1984.

\bibitem{kusu:stro:85}
S.~Kusuoka and D.~Stroock.
\newblock Applications of the {M}alliavin calculus. {II}.
\newblock {\em J. Fac. Sci. Univ. Tokyo Sect. IA Math}, 32:1--76, 1985.

\bibitem{kusu:stro:87}
S.~Kusuoka and D.~Stroock.
\newblock Applications of the {M}alliavin calculus. {III}.
\newblock {\em J. Fac. Sci. Univ. Tokyo Sect. IA Math}, 34:391--442, 1987.

\bibitem{lean:85}
R.~L{\'e}andre.
\newblock R{\'e}gularit{\'e} de processus de sauts d{\'e}g{\'e}n{\'e}r{\'e}s.
\newblock {\em Annales de l'IHP Probabilit{\'e}s et statistiques}, 21:125--146,
  1985.

\bibitem{leandre1988regularite}
R.~L{\'e}andre.
\newblock R{\'e}gularit{\'e} de processus de sauts d{\'e}g{\'e}n{\'e}r{\'e}s
  (ii).
\newblock In {\em Annales de l'IHP Probabilit{\'e}s et statistiques},
  volume~24, pages 209--236. Elsevier, 1988.

\bibitem{matt:stua:high:02}
J.~Mattingly, A.~Stuart, and D.~Higham.
\newblock Ergodicity for {SDE}s and approximations: locally {L}ipschitz vector
  fields and degenerate noise.
\newblock {\em Stoch. Proc. Appl.}, 101--2:185--232, 2002.

\bibitem{McKeanSinger}
H.~Mc~Kean and I.~Singer.
\newblock Curvature and eigen values of the {L}aplacian.
\newblock {\em J. Differential Geometry}, pages 43 -- 69, 1967.

\bibitem{menozzi2011parametrix}
S.~Menozzi.
\newblock Parametrix techniques and martingale problems for some degenerate
  kolmogorov equations.
\newblock {\em Electronic Communications in Probability}, 16:234--250, 2011.

\bibitem{picard1996existence}
J.~Picard.
\newblock On the existence of smooth densities for jump processes.
\newblock {\em Probability Theory and Related Fields}, 105(4):481--511, 1996.

\bibitem{priola2009densities}
E.~Priola and J.~Zabczyk.
\newblock Densities for {O}rnstein--{U}hlenbeck processes with jumps.
\newblock {\em Bulletin of the London Mathematical Society}, 41(1):41--50,
  2009.

\bibitem{sato}
K.~Sato.
\newblock {\em L\'evy processes and Infinitely divisible Distributions}.
\newblock Cambridge University Press, 2005.

\bibitem{sheu:91}
S.~J. Sheu.
\newblock Some estimates of the transition density of a nondegenerate diffusion
  {M}arkov process.
\newblock {\em Ann. Probab.}, 19--2:538--561, 1991.

\bibitem{shir:96}
A.N. Shiryaev.
\newblock {\em Probability, Second Edition.}
\newblock Graduate Texts in Mathematics, 95. Springer-Verlag, New York., 1996.

\bibitem{stroock1975diffusion}
D.~W. Stroock.
\newblock Diffusion processes associated with {L}{\'e}vy generators.
\newblock {\em Probability Theory and Related Fields}, 32(3):209--244, 1975.

\bibitem{stro:vara:79}
D.W. Stroock and S.R.S. Varadhan.
\newblock {\em Multidimensional diffusion processes}.
\newblock Springer-Verlag Berlin Heidelberg New-York, 1979.

\bibitem{szto:10}
Pawe{\l} Sztonyk.
\newblock Estimates of tempered stable densities.
\newblock {\em J. Theoret. Probab.}, 23(1):127--147, 2010.

\bibitem{tala:02}
D.~Talay.
\newblock Stochastic {H}amiltonian dissipative systems: exponential convergence
  to the invariant measure, and discretization by the implicit {E}uler scheme.
\newblock {\em Markov Processes and Related Fields}, 8--2:163--198, 2002.

\bibitem{wata:07}
T.~Watanabe.
\newblock Asymptotic estimates of multi-dimensional stable densities and their
  applications.
\newblock {\em Transactions of the American Mathematical Society},
  359(6):2851--2879, 2007.

\bibitem{zhan:14}
X.~Zhang.
\newblock Densities for {SDE}s driven by degenerate {$\alpha$}-stable
  processes.
\newblock {\em Ann. Probab.}, 42(5):1885--1910, 2014.

\end{thebibliography}
\end{document}